\documentclass[a4paper,11pt,reqno]{amsart}

\usepackage[utf8]{inputenc}
\usepackage[T1]{fontenc}
\usepackage{amsfonts,amssymb}
\usepackage{comment}
\usepackage{graphicx}
\usepackage{array}
\usepackage{orcidlink}
\usepackage{enumerate}
\usepackage{url}
\usepackage[all]{xy}
\usepackage{tikz}
\usepackage{esint}
\usetikzlibrary{arrows,shapes,positioning,calc,3d,decorations.pathreplacing,decorations.markings,cd,patterns}
\tikzset{
    side by side/.style 2 args={
        line width=2pt,
        #1,
        postaction={
            clip,postaction={draw,#2}
        }
    }
} 
\tikzstyle{dot}=[shape=circle,fill=black,inner sep=1.5pt]
\tikzstyle{root}=[ultra thick,->]
\tikzstyle{weight}=[orange,ultra thick,->]
\tikzstyle{weightdot}=[dot,orange]
\tikzstyle{wall}=[blue, thin, dashed]
\tikzstyle{auxwall}=[blue!20, thin, dashed]
\tikzstyle{->-}=[postaction={decorate,decoration={markings,mark=at position .5 with {\arrow{to}}}}]
\tikzstyle{-<-}=[postaction={decorate,decoration={markings,mark=at position .5 with {\arrowreversed{to};}}}]
\tikzstyle{->>-}=[postaction={decorate,decoration={markings,mark=at position .5 with {\arrow{to [sep] to}}}}]
\tikzstyle{-<<-}=[postaction={decorate,decoration={markings,mark=at position .5 with {\arrowreversed{to [sep] to};}}}]
\tikzstyle{->>>-}=[postaction={decorate,decoration={markings,mark=at position .5 with {\arrow{to [sep] to [sep] to}}}}]
\tikzstyle{-<<<-}=[postaction={decorate,decoration={markings,mark=at position .5 with {\arrowreversed{to [sep] to [sep] to};}}}]
\tikzstyle{--}=[double,double distance=1pt]
\usepackage{xargs}
\usepackage{stmaryrd}
\usepackage{hyperref}
\usepackage{multirow}
\usepackage{mathrsfs}
\usepackage{etoolbox}

\usepackage{empheq}
\numberwithin{equation}{section}

\usepackage{amsthm}
\theoremstyle{plain}
\newtheorem{introtheorem}{Theorem}
\newtheorem{introconjecture}[introtheorem]{Conjecture}
\newtheorem{introcorollary}[introtheorem]{Corollary}

\newtheorem{theorem}{Theorem}[section]
\newtheorem{lemma}[theorem]{Lemma}
\newtheorem{corollary}[theorem]{Corollary}

\newtheorem{proposition}[theorem]{Proposition}

\theoremstyle{definition}
\newtheorem{definition}[theorem]{Definition}

\newtheorem{example}[theorem]{Example}
\theoremstyle{remark}
\newtheorem*{remark}{Remark}

\newcommand{\R}{\mathbb{R}}
\newcommand{\C}{\mathbb{C}}

\newcommand{\Z}{\mathbb{Z}}
\newcommand{\N}{\mathbb{N}}

\newcommand{\F}{\mathbb{F}}
\newcommand{\E}{\mathbb{E}}

\newcommand{\calB}{\mathscr{B}}
\newcommand{\calC}{\mathcal{C}}

\newcommand{\calS}{\mathcal{S}}

\newcommand{\calF}{\mathcal{F}}

\newcommand{\calH}{\mathcal{H}}

\newcommand{\calL}{\mathcal{L}}
\newcommand{\calP}{\mathcal{P}}
\newcommand{\calZ}{\mathscr{Z}}
\newcommand{\bfG}{\mathbf{G}}
\newcommand{\bfH}{\mathbf{H}}
\newcommand{\catname}[1]{{\normalfont\textbf{#1}}} 

\newcommand{\ProjGrpd}{\catname{P}}
\newcommand{\ProjGrp}{P}
\newcommand{\EProjGrpd}{\catname{Q}}
\newcommand{\EProjGrp}{Q}
\newcommand{\Flat}{\catname{F}}
\newcommand{\Wall}{\catname{Y}}
\newcommand{\WallHat}{\hat{\catname{Y}}}
\newcommand{\FlatEmb}{\catname{F}^X}
\newcommand{\WallEmb}{\catname{Y}^X}
\newcommand{\WallHatEmb}{\smash{\hat{\catname{Y}}}^X}

\newcommand{\Prob}{\operatorname{Prob}}
\newcommand{\Res}{\operatorname{Res}}
\newcommand{\type}{\operatorname{type}}

\newcommand{\Scoxeter}{F}
\newcommand{\Sregular}{F}
\newcommand{\Ssingular}{Y}
\newcommand{\Sdetecting}{H}
\newcommand{\Fregular}{\mathscr{F}}
\newcommand{\Fsingular}{\mathscr{Y}}
\newcommand{\Fdetecting}{\mathscr{H}}
\newcommand{\Meas}{\mathbf{Meas}}
\newcommand{\Mble}{\mathbf{Mble}}
\newcommand{\Set}{\mathbf{Set}}
\newcommand{\CTO}{\mathbf{Set}_{cto}}
\newcommand{\calD}{\mathcal{D}}
\newcommand{\calA}{\mathcal{A}}
\newcommand{\calT}{\mathcal{T}}
\newcommand{\cato}{\textsc{cat}(\oldstylenums{0})}

\newcommand{\defeq}{\mathrel{\mathop{:}}=}

\newcommand{\abs}[1]{\lvert #1 \rvert}

\newcommand{\norm}[1]{\lVert #1 \rVert}
\newcommand{\gen}[1]{\langle #1 \rangle}

\newcommand{\Isom}{\operatorname{Isom}}
\newcommand{\Aut}{\operatorname{Aut}}
\newcommand{\op}{\text{op}}
\newcommandx{\Deltav}[1][1={*}]{\Delta_{#1}}
\newcommandx{\EDeltav}[1][1={*}]{\overline{\Delta}_{#1}}
\newcommand{\Deltaop}{\Delta^\op}
\newcommandx{\Deltaopv}[1][1={*}]{\Delta^\op_{#1}}

\newcommand{\Opp}{\operatorname{Opp}}
\newcommandx{\Oppv}[1][1={*}]{\operatorname{Opp}_{#1}}
\newcommand{\conv}{\operatorname{conv}}
\newcommand{\Hom}{\operatorname{Hom}}
\newcommand{\Homeo}{\operatorname{Homeo}}
\newcommand{\Stab}{\operatorname{Stab}}
\newcommand{\Sym}{\operatorname{Sym}}
\newcommand{\PGL}{\operatorname{PGL}}
\newcommand{\GL}{\operatorname{GL}}

\newcommand{\lk}{\operatorname{lk}}
\newcommand{\id}{\operatorname{id}}
\renewcommand{\setminus}{\smallsetminus}
\newcommand{\pr}{\proj}
\newcommand{\proj}{\operatorname{proj}}
\newcommand{\colim}{\operatorname{colim}}

\newcommand{\newcomment}[4]{%
\newcounter{#2counter}
\expandafter\newcommand\csname #1\endcsname[1]{%
\refstepcounter{#2counter}%
{\color{#4}(#3\arabic{#2counter})}\marginpar{\scriptsize\raggedright\textbf{\color{#4}(#2 \arabic{#2counter}):} ##1}%
}}
\reversemarginpar

\newcommandx{\weight}[1][1={}]{\ifstrempty{#1}{\check{\varpi}}{\varpi_{#1}}}
\newcommandx{\coweight}[1][1={}]{\ifstrempty{#1}{\check{\varpi}}{\check{\varpi}_{#1}}}
\renewcommandx{\root}[1][1={}]{\ifstrempty{#1}{\alpha}{\alpha_{#1}}}
\newcommandx{\coroot}[1][1={}]{\ifstrempty{#1}{\check{\alpha}}{\check{\alpha}_{#1}}}
\newcommandx{\sector}[4][1=S]{{#1}_{#2,#3,#4}}

\definecolor{darkgreen}{rgb}{0,0.6,0}

\begin{document}
\title[Normal subgroup theorem for two-dimensional buildings] {The normal subgroup theorem for lattices on two-dimensional Euclidean buildings}

\date{May 2026}


\excludecomment{internal}

\author[J.~Lécureux]{Jean Lécureux \orcidlink{0000-0002-3820-5517}}
\address{Univ. Bordeaux, IMB, CNRS, UMR 5251, F-33400 Talence, France}
\thanks{J.L.\ was supported by ANR-24-CE40-3137 project PLAGE}
\email{jean.lecureux@math.u-bordeaux.fr}

\author[S.~Witzel]{Stefan Witzel \orcidlink{0000-0002-3536-9938}}
\address{Mathematisches Institut, JLU Gießen, Arndtstr.\ 2, D-35392 Gießen, Germany}
\thanks{S.W.\ was supported through the DFG Heisenberg project WI 4079/6.}
\email{stefan.witzel@math.uni-giessen.de}

\begin{abstract}
  We prove the normal subgroup property for every group that acts properly and cocompactly on a two-dimensional Euclidean building: every normal subgroup has finite index or is contained in the finite kernel of the action. As a consequence, the non-residually finite lattices constructed by Titz Mite and the second author are virtually simple. They are the first known simple lattices on irreducible Euclidean buildings.
\end{abstract}

\maketitle

Margulis's celebrated Normal Subgroup Theorem \cite[Theorem~4]{MargulisBook} asserts that an irreducible lattice $\Gamma$ in a higher rank, semisimple algebraic group $G$ over a local field has all normal subgroups either finite central or of finite index. Each factor of $G$ is either a Lie group acting on a symmetric space of non-compact type (by Lie theory) or is an algebraic group over a local field acting on a Euclidean building (by Bruhat--Tits theory). Thus $\Gamma$ acts as a lattice on a product of symmetric spaces and Euclidean buildings. Conversely, it follows from the classification of Euclidean buildings by Tits and Weiss \cite{Tits86, Weiss} together with Margulis arithmeticity \cite[Theorem~1]{MargulisBook} that most lattices on such products are arithmetic, see \cite{BCL} for a discussion. As an exception there are infinitely many two-dimensional buildings that are not \emph{Bruhat--Tits}, i.e.\ do not arise from algebraic groups, but are \emph{exotic} and also admit lattices. Such lattices cannot be arithmetic.

%
%
%
%
%
%

The Normal Subgroup Theorem has been extended in various directions (see for example \cite{BaderShalom}, \cite{StuckZimmer}, \cite{BoutonnetHoudayer} among many others). Our contribution is to extended it to all uniform lattices on two-dimensional buildings:

\begin{introtheorem}[Normal Subgroup Theorem]\label{thm:main_theorem}
  Let $X$ be an irreducible two-dimensional Euclidean building and let $\Gamma$ be a group acting properly and cocompactly on $X$. If $N \lhd \Gamma$ is a normal subgroup then either $N$ is finite and lies in the kernel of the action on $X$ or it has finite index in $\Gamma$.
\end{introtheorem}

The case where the building $X$ is of type $\tilde A_2$ was recently proven by Bader, Furman, and the first author \cite{BFL} and we will discuss the relationship below. Extending the Normal Subgroup Theorem to exotic lattices can be seen as a result unifying arithmetic and exotic lattices. Besides that an important motivation is an expected contrast between arithmetic and exotic lattices. While arithmetic lattices are linear and consequently residually finite, this is not expected for non-arithmetic lattices. In fact, the following was conjectured in \cite{BCL} in the $\tilde A_2$ case:

\begin{introconjecture}\label{conj:exotic_simple}
  Let $X$ be an irreducible Euclidean building of dimension $\geq 2$.
  Let $\Gamma < \Aut(X)$ act properly and cocompactly on $X$. If $\Gamma$ is not arithmetic then it is virtually simple.
\end{introconjecture}

Our main theorem reduces the conjecture for lattices on exotic buildings to proving that they are not residually finite, see Section~\ref{sec:residual_finiteness}. Although the conjecture predicts simplicity of infinitely many groups, many of which are understood rather explicitly, not a single one was known to satisfy the conjecture until now. However, recently examples of lattices on buildings of type $\tilde C_2$ were constructed by Titz Mite and the second author that are \emph{not} residually finite \cite{TitzMiteWitzel}. In combination with our main theorem this yields:

\begin{introcorollary}\label{cor:tmw_simple}
    There are simple groups acting properly and cocompactly on buildings of type $\tilde C_2$.
\end{introcorollary}

Specifically, the groups $\Gamma^2_1 = \check{\Gamma}^2_1$, $\check{\Gamma}^3_1$, $\check{\Gamma}^3_2$, $\check{\Gamma}^3_3$, $\check{\Gamma}^3_4$ of \cite{TitzMiteWitzel} are simple. An important feature of these simple groups is that they are $\cato$-groups since the building that they act on properly and cocompactly is a $\cato$-space. In particular, they are very explicitly understood geometrically. This is in contrast to other examples of finitely presented simple groups, such as Thompson groups and Kac--Moody groups: they also act properly on $\cato$-spaces but the action is not cocompact and therefore provides much less geometric information on the group. The simple lattices on products of trees found by Burger an Mozes \cite{BurgerMozes00} are $\cato$ but they form a single quasi-isometry class and are less rigid. Indeed, an important feature of a simple lattice $\Gamma$ on a two-dimensional irreducible building $X$ is that it has strong rigidity properties. For instance $\Gamma$ inherits quasi-isometric rigidity from $X$ in a very strong sense, see \cite[Proposition~4.15]{TitzMiteWitzel}. The lattice $\Gamma$ is also known or expected (depending on the type of $X$) to have the same rigidity properties as arithmetic lattices of higher rank with respect to representations and actions on other spaces \cite{StrongT, Antoine}.

\medskip

The general strategy to prove Theorem~\ref{thm:main_theorem} is still the one devised by Margulis in the case of arithmetic groups. Taking the action of $\Gamma$ to be faithful, the plan is to prove that any of its proper quotients has property (T) and is amenable. Since $\Gamma$ has property (T) by \cite{Oppenheim} which passes to quotients we are concerned exclusively with the amenability part. It is a consequence of the following Factor Theorem. Its statement involves the building at infinity of $X$ which is a bipartite graph with sets of vertices $\Delta_1$ and $\Delta_2$ and set of edges $\Delta$. The asymptotic structure of $X$ induces a topology and, in fact, a measure class on all three sets. The obvious maps $\pi_i \colon \Delta \to \Delta_i$ are $\Aut(X)$-equivariant.

\begin{introtheorem}[Factor Theorem]\label{thm:factorintro}
  Let $X$ be an irreducible two-dimensional building with spaces of chambers $\Delta$ and vertices at infinity $\Delta_1$ and $\Delta_2$, and let $\Gamma$ act on $X$ properly and cocompactly. The only $\Gamma$-invariant, weak-* closed subgalgebras of $L^\infty(\Delta)$ are $L^\infty(\Delta_1)$, $L^\infty(\Delta_2)$ and $\C$.
\end{introtheorem}

The Normal Subgroup Theorem \ref{thm:main_theorem} follows from the Factor Theorem \ref{thm:factorintro} as explained in Section \ref{sec:outline}. The article is therefore devoted to proving Theorem~\ref{thm:factorintro}.

For buildings of type $\tilde{A}_2$ both theorems were proven in \cite{BFL} making use of results from \cite{BCL}. We will now discuss ingredients to the proof of Theorem~\ref{thm:factorintro} with an eye on how they relate to results in \cite{BFL} and \cite{BCL}. Roughly speaking, given a subalgebra $A < L^\infty(\Delta)$ the plan is to show that since it is $\Gamma$-invariant and closed, it is invariant under (almost all) \emph{projectivities}. To do so, we first produce a sequence of elements of $\Gamma$ that converges to the needed projectivity in the compact-open topology (Theorem~\ref{thm:contracting}). This crucially uses ergodicity of the \emph{singular Cartan flow} (Theorem~\ref{thm:SingularCartanErgodic}). Combining convergence in the compact-open topology with a martingale convergence result resembling Lebesgue differentiation (Theorem~\ref{thm:martingale_convergence_fV}) we show that the elements of $\Gamma$ converge to the projectivity in the weak-*-topology (Theorem~\ref{thm:convergence}). Since projectivities act highly transitively, this is enough to deduce Theorem~\ref{thm:factorintro}.

The proof involves various measure spaces and showing that maps between them are measure preserving or measure-class preserving. The best-known among them are the \emph{harmonic} measures on $\Delta$ which are indexed by vertices of $X$ and lie in a common measure class; they appear in \cite{Parkinson2} or more recently \cite{RemyTrojan}. Another important example is the measure on the regular and the singular \emph{Cartan flow}, that was the main object of study in type $\tilde{A}_2$ in \cite{BCL} (the regular Cartan flow is sometimes called \emph{Weyl chamber flow} in more algebraic settings). Let us also mention the Haar measure on the group of projectivities and the measure on the space of pairs of opposite chambers. In order to compare measures and to show disintegration results \emph{prouniform measures} were introduced in \cite{BFL} as limit measures of finite spaces.

One of our main contributions is to develop the theory of prouniform measures further and to exploit it systematically. Specifically, in Appendix~\ref{sec:limit_measures} we study how an inverse system of constant-to-one maps and a choice of a counting measure on one object leads to a measure on the limit, and how these measures can be compared and disintegrated. This method is used in Section~\ref{sec:pro-uniform_measures} to construct measures on spaces of embeddings of complexes. The generality of the framework allows us to put constraints on the embeddings we consider (technically implemented by appropriate categories). This allows to subsume Haar measures on tdlc groups in the construction and make the comparison of the singular Cartan flow with the Haar measure of the projectivity group (Proposition~\ref{prop:Y/M}) seamless, which it was not in \cite{BCL}. Another illustration of how useful the approach is is when constructing the measure on the space of pairs of opposite chambers, whose existence is well-known (\cite{Parkinson2,RemyTrojan}): we can use that the effects of a choice of basepoint cancel out without having to explicitly compute these effects (Lemma~\ref{lem:muop_basepoint_change}).

Another main difference with the proof of \cite{BFL} lies in the proof of \mbox{weak-*}-convergence mentioned above. In \cite{BFL} it used an intricate argument applying the Martingale Convergence Theorem to what was called the detecting flow, constructed for that purpose in an \textit{ad hoc} fashion. We give a more streamlined argument, not using the detecting flow at all, only relying on some version of the Martingale Convergence Theorem which plays a role very analogous to the Lebesgue Differentiation Theorem.

All differences so far are improvements rendering arguments clearer and more natural, but are unrelated to the generalization from type $\tilde{A}_2$ to arbitrary two-dimensional buildings. However, there is a new profound technicality as well. In order to apply the prouniform measure machinery in Section~\ref{sec:pro-uniform_measures}, the building needs to satisfy a certain condition, called \emph{symmetry} with respect to a family of complexes in \cite{BFL}. The family is subcomplexes of an apartment for many measures, such as the regular Cartan flow, but consists of subcomplexes of \emph{wall-spaces}, spaces that decompose as a tree times a line, for measures related to the singular Cartan flow. The corresponding analysis in type $\tilde{A}_2$ in \cite{BFL} could be carried out in a rather elementary combinatorial fashion, thanks to the fact that $\tilde{A}_2$ has only three parallel classes of walls. The analogous analysis in our case in Section~\ref{sec:convex_geometry_on_trees} is more involved.

Extending one of the main results of \cite{BCL}, we are able to prove what would be a special case of a Theorem of Howe--Moore in the algebraic setting:

\begin{introtheorem}\label{thm:introergodic}
  Let $X$ be a two-dimensional Euclidean building and let $\Gamma$ be a group acting properly and cocompactly on $X$. The singular Cartan flows $\Fsingular_1, \Fsingular_2$ and the regular Cartan flow $\Fregular$ are ergodic.

  More precisely, every non-trivial translation of $\Sregular$ acts on $\Fregular/\Gamma$ ergodically.
\end{introtheorem}

While we do reprove the main technical core of \cite{BCL}, we cannot generalize its main result. This is because a result due to Schleiermacher \cite{Schleiermacher} about projectivity groups of projective planes is not known to hold for generalized quadrangles or hexagons. Instead, we get the following weaker conclusion:

\begin{introtheorem}\label{thm:intrononlinearity}
  Let $X$ be a two-dimensional Euclidean building and let $\Gamma$ be a group acting properly and cocompactly on $X$. Assume that for some $i \in \{1,2\}$ no cocompact subgroup of the closure of the projectivity group $\overline{P}_i$ admits a faithful representation in $\GL_n(K)$ with $K$ a local field. Then any linear representation of $\Gamma$ has finite image.
\end{introtheorem}

\setcounter{tocdepth}{1}
\tableofcontents

\section{Just infiniteness and residual finiteness}\label{sec:residual_finiteness}

In this brief section we recall the relationship between the normal subgroup property and residual finiteness. In particular, we show how Theorem~\ref{thm:main_theorem} reduces Conjecture~\ref{conj:exotic_simple} to showing that the lattices are not residually finite, and how Corollary~\ref{cor:tmw_simple} follows.

We first recall the relevant definitions. The \emph{finite residual} $\Gamma^{(\infty)}$ of a group $\Gamma$ is the intersection of all its finite index subgroups (and one may equivalently restrict to normal ones). It is the subgroup of all elements of $\Gamma$ that map trivially under any homomorphism to a finite group. The group $\Gamma$ is \emph{residually finite} if $\Gamma^{(\infty)} = \{1\}$. A group is \emph{just-infinite} if every proper quotient is finite or, equivalently, every non-trivial normal subgroup is of finite index. It is \emph{hereditarily just-infinite} if every finite index subgroup is just-infinite.

Note that another way to phrase the normal subgroup property in Theorem~\ref{thm:main_theorem} is to say that the image $\bar{\Gamma}$ of $\Gamma \to \Aut(X)$ (which has finite kernel) is just-infinite. Note also that a finite index subgroup of a lattice on $X$ is itself a lattice on $X$ hence Theorem~\ref{thm:main_theorem} asserts that in fact $\bar{\Gamma}$ is hereditarily just-infinite.

There is the following dichotomy.

\begin{proposition}
  If $\Gamma$ is hereditarily just-infinite then $\Gamma^{(\infty)}$ is either trivial or simple of finite index.

  In particular, $\Gamma$ is either residually finite or virtually simple.
\end{proposition}

In the theory of just-infinite groups \cite{Wilson71}, this is part of a trichotomy, namely a just-infinite group that is not hereditarily just-infinite is a \emph{branch group}, see \cite[Theorem~3(a)]{Grigorchuk00}.

\begin{proof}
  If $\Gamma^{(\infty)} \ne \{1\}$ then it has finite index since $\Gamma$ is just-infinite. Then if $\{1\} \ne N \lhd \Gamma^{(\infty)}$ is non-trivial normal, then $N$ has to have finite index in $\Gamma^{(\infty)}$ because $\Gamma^{(\infty)}$ is just-infinite being finite index in the hereditarily just-infinite group $\Gamma$. Then $N$ has finite index in $\Gamma$ and it follows from the definition of $\Gamma^{(\infty)}$ that $N = \Gamma^{(\infty)}$.
\end{proof}

Combining with Theorem~\ref{thm:main_theorem} we see:

\begin{corollary}\label{cor:virtualsimple}
  If $\Gamma < \Aut(X)$ is discrete and acts cocompactly on $X$ then it is either residually finite or virtually simple. More precisely if $\Gamma = \Gamma^{(\infty)}$ then it is simple.
\end{corollary}

In particular, Corollary~\ref{cor:tmw_simple} follows together with the Main Theorem of \cite{TitzMiteWitzel}.

\section{Euclidean buildings}\label{sec:buildings}

\subsection{Euclidean Coxeter complexes}\label{sec:coxeter_complexes}

Let $\Scoxeter$ be a Euclidean vector space and let $\Scoxeter^*$ be its dual space. Let $\Phi \subseteq \Scoxeter^*$ be a root system
with simple roots $\root[1],\ldots,\root[k] \in \Scoxeter^*$ and fundamental weights $\weight[1],\ldots,\weight[k] \in \Scoxeter$, i.e.\ $\root[i](\weight[j]) = \delta_{ij}$. The dual root system has basis $\coroot[1],\ldots,\coroot[k] \in \Scoxeter$ characterized by the condition that the reflection $s_\alpha \colon \Scoxeter \to \Scoxeter$ is given by $s_{\root}(x) = x - \root(x)\coroot$. We also denote the reflection $x \mapsto x - \coroot(x)\root$ on $\Scoxeter^*$ by $s_\alpha$.

Below we illustrate the irreducible cases of dimension $2$ that will concern us in the paper. We draw the extended (Euclidean) Coxeter diagrams. The roots and dual roots are drawn in black and the weights in orange. The dark blue dashed lines are the lines $\root[i](\cdot) = 1$ while the light blue lines are other lines of the form $\root[i](\cdot) \in \Z$. The solid blue lines are the lines $\root[i](\cdot) = 0, i \ge 1$ and $\root[0](\cdot) = 1$, where $\root[0]$ is the longest root. The grey sector is the Weyl chamber given by $\root[i](\cdot) \ge 0$ for all $i \ge 1$.

\begin{center}
  %
  %
  
\begin{tikzpicture}
  \pgfdeclarelayer{labels}
  \pgfdeclarelayer{weights}
  \pgfdeclarelayer{background}
  \pgfsetlayers{background,main,weights,labels}
  \node at (0,-3) {$A_2$};
  \node[dot] (0) at ($(0,-1.5) + (90:{sqrt(1/3)})$) {};
  \node[dot] (1) at ($(0,-1.5) + (-30:{sqrt(1/3)})$) {};
  \node[dot] (2) at ($(0,-1.5) + (210:{sqrt(1/3)})$) {};
  \node[above] at (0) {$0$};
  \node[below] at (1) {$1$};
  \node[below] at (2) {$2$};
  \draw (0) -- (1) (1) -- (2) (2) -- (0);
  \begin{scope}[shift={(-3.5,0)},scale=1.5]
  \node at (0,-2) {$\Scoxeter$};
  \coordinate (alpha1) at (0:{sqrt(2)});
  \coordinate (alpha2) at (120:{sqrt(2)});
  \foreach \ang in {60,120,...,360}
  {
    \draw[root] (0,0) -- (\ang:{sqrt(2)});
  }
  \node at ($1.3*(alpha1)$) {$\coroot[1]$};
  \node at ($1.3*(alpha2)$) {$\coroot[2]$};
  \coordinate (omega2) at (90:{sqrt(2/3)});
  \coordinate (omega1) at (30:{sqrt(2/3)});
  \begin{pgfonlayer}{weights}
  \foreach \ang in {30,90,...,330}
  {
    \node[weightdot] at (\ang:{sqrt(2/3)}) {};
  }
  \end{pgfonlayer}
  \begin{pgfonlayer}{labels}
  \node[orange] at ($1.3*(omega1)$) {$\weight[1]$};
  \node[orange] at ($1.3*(omega2)$) {$\weight[2]$};
  \end{pgfonlayer}
  \begin{pgfonlayer}{background}
  \foreach \ang in {30,90,150}
  {
    \draw[auxwall] (\ang:{sqrt(8/3)}) -- ({\ang+180}:{sqrt(8/3)});
  }
    \path[fill=black!10,draw=blue] ($2*(omega2)$) -- (0,0) -- ($2*(omega1)$);
    \path[draw=blue] (alpha2) -- (alpha1);
  \foreach \ang in {60,120,...,360}
  {
    \draw[wall] (\ang:{sqrt(2)}) -- ({\ang+120}:{sqrt(2)});
    \draw[auxwall] ($(\ang:{sqrt(2)})+(\ang+90:{sqrt(2/3)})$) -- ($({\ang+120}:{sqrt(2)})+(\ang+30:{sqrt(2/3)})$);
  }
  \end{pgfonlayer}
  \end{scope}
  \begin{scope}[shift={(3.5,0)},scale=1.5]
  \node at (0,-2) {$\Scoxeter^*$};
  \coordinate (alpha1) at (0:{sqrt(2)});
  \coordinate (alpha2) at (120:{sqrt(2)});
  \foreach \ang in {60,120,...,360}
  {
    \draw[root] (0,0) -- (\ang:{sqrt(2)});
  }
  \node at ($1.3*(alpha1)+1.3*(alpha2)$) {$\root[0]$};
  \node at ($1.3*(alpha1)$) {$\root[1]$};
  \node at ($1.3*(alpha2)$) {$\root[2]$};
  \coordinate (omega2) at (90:{sqrt(2/3)});
  \coordinate (omega1) at (30:{sqrt(2/3)});
  \begin{pgfonlayer}{weights}
  \foreach \ang in {30,90,...,330}
  {
    \node[weightdot] at (\ang:{sqrt(2/3)}) {};
  }
  \end{pgfonlayer}
  \begin{pgfonlayer}{labels}
  \node[orange] at ($1.3*(omega1)$) {$\coweight[1]$};
  \node[orange] at ($1.3*(omega2)$) {$\coweight[2]$};
  \end{pgfonlayer}
  \begin{pgfonlayer}{background}
  \foreach \ang in {30,90,150}
  {
    \draw[auxwall] (\ang:{sqrt(8/3)}) -- ({\ang+180}:{sqrt(8/3)});
  }
    \path[fill=black!10,draw=blue] ($2*(omega2)$) -- (0,0) -- ($2*(omega1)$);
    \path[draw=blue] (alpha2) -- (alpha1);
  \foreach \ang in {60,120,...,360}
  {
    \draw[wall] (\ang:{sqrt(2)}) -- ({\ang+120}:{sqrt(2)});
    \draw[auxwall] ($(\ang:{sqrt(2)})+(\ang+90:{sqrt(2/3)})$) -- ($({\ang+120}:{sqrt(2)})+(\ang+30:{sqrt(2/3)})$);
  }
  \end{pgfonlayer}
  \end{scope}
\end{tikzpicture}

  %
  %
\begin{tikzpicture}
  \pgfdeclarelayer{labels}
  \pgfdeclarelayer{weights}
  \pgfdeclarelayer{walls}
  \pgfdeclarelayer{background}
  \pgfsetlayers{background,walls,main,weights,labels}
  \node at (0,-3) {$B_2$};
  \node at (-4,-3) {$\Scoxeter^*$};
  \node at (4,-3) {$\Scoxeter$};
  \node[dot] (0) at (-1,-2) {};
  \node[dot] (2) at (0,-2) {};
  \node[dot] (1) at (1,-2) {};
  \node[below] at (0) {$0$};
  \node[below] at (1) {$1$};
  \node[below] at (2) {$2$};
  \draw (0) edge[--] (2) (2) edge[--] (1);
  \path (0) edge[->-,draw=none] (2) (2) edge[-<-,draw=none] (1);
  \begin{scope}[shift={(-4,0)},rotate=45]
  \coordinate (alpha2) at (-45:2);
  \coordinate (alpha1) at (90:{sqrt(2)});
  \foreach \ang in {90,180,...,360}
  {
    \draw[root] (0,0) -- (\ang:{sqrt(2)});
  }
  \foreach \ang in {45,135,...,315}
  {
    \draw[root] (0,0) -- (\ang:2);
  }
  \node at ($1.3*(alpha2)$) {$\coroot[2]$};
  \node at ($1.3*(alpha1)$) {$\coroot[1]$};
  \end{scope}
  \begin{scope}[shift={(-4,0)}]
  \coordinate (omega2) at (45:{sqrt(2)});
  \coordinate (omega1) at (90:1);
  \begin{pgfonlayer}{weights}
  \foreach \ang in {45,135,...,315}
  {
    \node[weightdot] at (\ang:{sqrt(2)}) {};
  }
  \foreach \ang in {90,180,...,360}
  {
    \node[weightdot] at (\ang:1) {};
  }
  \end{pgfonlayer}
  \begin{pgfonlayer}{walls}
  \foreach \ang in {90,180}
  {
    \draw[auxwall] (\ang:2) -- ({\ang+180}:2);
  }
  \foreach \ang in {45,135}
  {
    \draw[auxwall] (\ang:{sqrt(1/2)}) -- ({\ang+180}:{sqrt(1/2)});
  }
  \foreach \ang in {90,180,...,360}
  {
    \draw[wall] ($(\ang:1)+(\ang+90:1)$) -- ($(\ang:1)+(\ang-90:1)$);
  }
  \foreach \ang in {135,225,...,315}
  {
    \draw[wall] ($(\ang:{sqrt(1/2)})+(\ang+90:{sqrt(2)})$) -- ($(\ang:{sqrt(1/2)})-(\ang+90:{sqrt(2)})$);
  }
  \foreach \ang in {45}
  {
    \draw[blue] ($(\ang:{sqrt(1/2)})+(\ang+90:{sqrt(2)})$) -- ($(\ang:{sqrt(1/2)})+(\ang-90:{sqrt(2)})$);
  }
  \end{pgfonlayer}
  \begin{pgfonlayer}{background}
    \path[fill=black!10,draw=blue] ($1*(omega2)$) -- (0,0) -- ($2*(omega1)$);
  \end{pgfonlayer}
  \begin{pgfonlayer}{labels}
  \node[orange] at ($1.3*(omega2)$) {$\weight[2]$};
  \node[orange] at ($1.3*(omega1)$) {$\weight[1]$};
  \end{pgfonlayer}
\end{scope}
  \begin{scope}[shift={(4,0)},rotate=45,scale={sqrt(1/2)}]
  \coordinate (alpha2) at (-45:{sqrt(2)});
  \coordinate (alpha1) at (90:2);
  \foreach \ang in {90,180,...,360}
  {
    \draw[root] (0,0) -- (\ang:2);
  }
  \foreach \ang in {45,135,...,325}
  {
    \draw[root] (0,0) -- (\ang:{sqrt(2)});
  }
  \node at ($1.3*(alpha1)+2*1.3*(alpha2)$) {$\root[0]$};
  \node at ($1.3*(alpha2)$) {$\root[2]$};
  \node at ($1.3*(alpha1)$) {$\root[1]$};
  \coordinate (omega1) at (45:{sqrt(2)});
  \coordinate (omega2) at (0:1);
  \begin{pgfonlayer}{weights}
  \foreach \ang in {45,135,...,315}
  {
    \node[weightdot] at (\ang:{sqrt(2)}) {};
  }
  \foreach \ang in {90,180,...,360}
  {
    \node[weightdot] at (\ang:1) {};
  }
  \end{pgfonlayer}
  \begin{pgfonlayer}{walls}
  \foreach \ang in {90,180}
  {
    \draw[auxwall] (\ang:2) -- ({\ang+180}:2);
  }
  \foreach \ang in {45,135}
  {
    \draw[auxwall] (\ang:{sqrt(1/2)}) -- ({\ang+180}:{sqrt(1/2)});
  }
  \foreach \ang in {90,180,...,360}
  {
    \draw[wall] ($(\ang:1)+(\ang+90:1)$) -- ($(\ang:1)+(\ang-90:1)$);
  }
  \foreach \ang in {135,225,...,315}
  {
    \draw[wall] ($(\ang:{sqrt(1/2)})+(\ang+90:{sqrt(2)})$) -- ($(\ang:{sqrt(1/2)})-(\ang+90:{sqrt(2)})$);
  }
  \foreach \ang in {45}
  {
    \draw[blue] ($(\ang:{sqrt(1/2)})+(\ang+90:{sqrt(2)})$) -- ($(\ang:{sqrt(1/2)})+(\ang-90:{sqrt(2)})$);
  }
  \end{pgfonlayer}
  \begin{pgfonlayer}{background}
    \path[fill=black!10,draw=blue] ($2*(omega2)$) -- (0,0) -- ($1*(omega1)$);
  \end{pgfonlayer}
  \begin{pgfonlayer}{labels}
  \node[orange] at ($1.3*(omega2)$) {$\coweight[2]$};
  \node[orange] at ($1.3*(omega1)$) {$\coweight[1]$};
  \end{pgfonlayer}
  \end{scope}
\end{tikzpicture}

  %
  %
\begin{tikzpicture}
  \pgfdeclarelayer{labels}
  \pgfdeclarelayer{weights}
  \pgfdeclarelayer{walls}
  \pgfdeclarelayer{background}
  \pgfsetlayers{background,walls,main,weights,labels}
  \node at (0,-3) {$C_2$};
  \node at (-4,-3) {$\Scoxeter$};
  \node at (4,-3) {$\Scoxeter^*$};
  \node[dot] (0) at (-1,-2) {};
  \node[dot] (1) at (0,-2) {};
  \node[dot] (2) at (1,-2) {};
  \node[below] at (0) {$0$};
  \node[below] at (1) {$1$};
  \node[below] at (2) {$2$};
  \draw (0) edge[--] (1) (1) edge[--] (2);
  \path (0) edge[->-,draw=none] (1) (1) edge[-<-,draw=none] (2);
  \begin{scope}[shift={(-4,0)},xscale=-1,rotate=90]
  \coordinate (alpha1) at (-45:{sqrt(2)});
  \coordinate (alpha2) at (90:1);
  \foreach \ang in {90,180,...,360}
  {
    \draw[root] (0,0) -- (\ang:1);
  }
  \foreach \ang in {45,135,...,315}
  {
    \draw[root] (0,0) -- (\ang:{sqrt(2)});
  }
  \node at ($1.3*(alpha1)$) {$\coroot[1]$};
  \node at ($1.3*(alpha2)$) {$\coroot[2]$};
  \coordinate (omega2) at (45:{sqrt(1/2)});
  \coordinate (omega1) at (0:1);
  \begin{pgfonlayer}{weights}
  \foreach \ang in {45,135,...,315}
  {
    \node[weightdot] at (\ang:{sqrt(1/2)}) {};
  }
  \foreach \ang in {90,180,...,360}
  {
    \node[weightdot] at (\ang:1) {};
  }
  \end{pgfonlayer}
  \begin{pgfonlayer}{walls}
  \foreach \ang in {90,180}
  {
    \draw[auxwall] (\ang:2) -- ({\ang+180}:2);
  }
  \foreach \ang in {45,135}
  {
    \draw[auxwall] (\ang:{sqrt(1/2)}) -- ({\ang+180}:{sqrt(1/2)});
  }
  \foreach \ang in {90,180,...,360}
  {
    \draw[wall] ($(\ang:1/2)+(\ang+90:1.5)$) -- ($(\ang:1/2)+(\ang-90:1.5)$);
    \draw[auxwall] ($(\ang:1)+(\ang+90:1)$) -- ($(\ang:1)+(\ang-90:1)$);
  }
  \foreach \ang in {135,225,...,315}
  {
    \draw[wall] ($(\ang:{sqrt(1/2)})+(\ang+90:{sqrt(2)})$) -- ($(\ang:{sqrt(1/2)})-(\ang+90:{sqrt(2)})$);
  }
  \foreach \ang in {0}
  {
    \draw[blue] ($(\ang:1/2)+(\ang+90:1.5)$) -- ($(\ang:1/2)+(\ang-90:1.5)$);
  }
  \end{pgfonlayer}
  \begin{pgfonlayer}{background}
    \path[fill=black!10,draw=blue] ($2*(omega2)$) -- (0,0) -- ($2*(omega1)$);
  \end{pgfonlayer}
  \begin{pgfonlayer}{labels}
  \node[orange] at ($1.3*(omega2)$) {$\weight[2]$};
  \node[orange] at ($1.3*(omega1)$) {$\weight[1]$};
  \end{pgfonlayer}
  \end{scope}
  \begin{scope}[shift={(4,0)},xscale=-1,rotate=90]
  \coordinate (alpha1) at (-45:{sqrt(2)});
  \coordinate (alpha2) at (90:2);
  \foreach \ang in {90,180,...,360}
  {
    \draw[root] (0,0) -- (\ang:2);
  }
  \foreach \ang in {45,135,...,325}
  {
    \draw[root] (0,0) -- (\ang:{sqrt(2)});
  }
  \node at ($1.2*(alpha2)$) {$\root[2]$};
  \node at ($1.3*(alpha1)$) {$\root[1]$};
  \node at ($2*1.2*(alpha1)+1.2*(alpha2)$) {$\root[0]$};
  \coordinate (omega2) at (45:{sqrt(2)});
  \coordinate (omega1) at (0:1);
  \begin{pgfonlayer}{weights}
  \foreach \ang in {45,135,...,315}
  {
    \node[weightdot] at (\ang:{sqrt(2)}) {};
  }
  \foreach \ang in {90,180,...,360}
  {
    \node[weightdot] at (\ang:1) {};
  }
  \end{pgfonlayer}
  \begin{pgfonlayer}{walls}
  \foreach \ang in {90,180}
  {
    \draw[auxwall] (\ang:2) -- ({\ang+180}:2);
  }
  \foreach \ang in {45,135}
  {
    \draw[auxwall] (\ang:{sqrt(1/2)}) -- ({\ang+180}:{sqrt(1/2)});
  }
  \foreach \ang in {90,180,...,360}
  {
    \draw[wall] ($(\ang:1)+(\ang+90:1)$) -- ($(\ang:1)+(\ang-90:1)$);
  }
  \foreach \ang in {135,225,...,315}
  {
    \draw[wall] ($(\ang:{sqrt(1/2)})+(\ang+90:{sqrt(2)})$) -- ($(\ang:{sqrt(1/2)})-(\ang+90:{sqrt(2)})$);
  }
  \foreach \ang in {45}
  {
    \draw[blue] ($(\ang:{sqrt(1/2)})+(\ang+90:{sqrt(2)})$) -- ($(\ang:{sqrt(1/2)})+(\ang-90:{sqrt(2)})$);
  }
  \end{pgfonlayer}
  \begin{pgfonlayer}{background}
    \path[fill=black!10,draw=blue] ($1*(omega2)$) -- (0,0) -- ($2*(omega1)$);
  \end{pgfonlayer}
  \begin{pgfonlayer}{labels}
  \node[orange] at ($1.3*(omega2)$) {$\coweight[2]$};
  \node[orange] at ($1.3*(omega1)$) {$\coweight[1]$};
  \end{pgfonlayer}
  \end{scope}
\end{tikzpicture}

  %
  %
\begin{tikzpicture}
  \pgfdeclarelayer{labels}
  \pgfdeclarelayer{weights}
  \pgfdeclarelayer{background}
  \pgfsetlayers{background,main,weights,labels}
  \node at (0,-3) {$G_2$};
  \node[dot] (0) at (-1,-2) {};
  \node[dot] (2) at (0,-2) {};
  \node[dot] (1) at (1,-2) {};
  \node[below] at (0) {$0$};
  \node[below] at (1) {$1$};
  \node[below] at (2) {$2$};
  \draw (0) -- (2) (2) edge[--,double distance = 1.5pt] (1);
  \path (2) edge[->-] (1);
  \begin{scope}[shift={(-3.5,0)},scale=1]
  \node at (0,-3) {$\Scoxeter^*$};
  \coordinate (alpha1) at (0:{sqrt(2)});
  \coordinate (alpha2) at (150:{sqrt(2/3)});
  \foreach \ang in {60,120,...,360}
  {
    \draw[root] (0,0) -- (\ang:{sqrt(2)});
  }
  \foreach \ang in {30,90,...,330}
  {
    \draw[root] (0,0) -- (\ang:{sqrt(2/3)});
  }
  \node at ($1.3*(alpha1)$) {$\coroot[1]$};
  \node at ($1.3*(alpha2)$) {$\coroot[2]$};
  \coordinate (omega2) at (90:{sqrt(2/3)});
  \coordinate (omega1) at (60:{sqrt(2)});
  \begin{pgfonlayer}{weights}
  \foreach \ang in {60,120,...,360}
  {
    \node[weightdot] at (\ang:{sqrt(2)}) {};
    \node[weightdot] at (\ang+30:{sqrt(2/3)}) {};
  }
  \end{pgfonlayer}
  \begin{pgfonlayer}{labels}
  \node[above right,orange] at ($1*(omega1)$) {$\weight[1]$};
  \node[above,orange] at ($1*(omega2)$) {$\weight[2]$};
  \end{pgfonlayer}
  \begin{pgfonlayer}{background}
  \foreach \ang in {30,90,150}
  {
    \draw[auxwall] (\ang:{sqrt(6)}) -- ({\ang+180}:{sqrt(6)});
  }
  \foreach \ang in {60,120,180}
  {
    \draw[auxwall] (\ang:{sqrt(8)}) -- ({\ang+180}:{sqrt(8)});
  }
    \path[fill=black!10,draw=blue] ($3*(omega2)$) -- (0,0) -- ($2*(omega1)$);
  \foreach \ang in {60,120,...,360}
  {
    \draw[wall] ($(\ang:{sqrt(1/2)})+(\ang+90:{sqrt(6)})$) -- ($(\ang:{sqrt(1/2)})+(\ang-90:{sqrt(6)})$);
    \draw[auxwall] ($(\ang:{sqrt(2)})+(\ang+90:{sqrt(6)})$) -- ($(\ang:{sqrt(2)})+(\ang-90:{sqrt(6)})$);
    \draw[auxwall] ($(\ang:{3*sqrt(1/2)})+(\ang+90:{sqrt(3/2)})$) -- ($(\ang:{3*sqrt(1/2)})+(\ang-90:{sqrt(3/2)})$);
  }
  \foreach \ang in {30,90,...,330}
  {
    \draw[wall] ($(\ang:{sqrt(1/6)})+(\ang+90:{sqrt(6)})$) -- ($(\ang:{sqrt(1/6)})+(\ang-90:{sqrt(6)})$);
    \draw[auxwall] ($(\ang:{sqrt(2/3)})+(\ang+90:{sqrt(5)})$) -- ($(\ang:{sqrt(2/3)})+(\ang-90:{sqrt(5)})$);
    \draw[auxwall] ($(\ang:{sqrt(3/2)})+(\ang+90:{sqrt(4)})$) -- ($(\ang:{sqrt(3/2)})+(\ang-90:{sqrt(4)})$);
    \draw[auxwall] ($(\ang:{sqrt(8/3)})+(\ang+90:{sqrt(3)})$) -- ($(\ang:{sqrt(8/3)})+(\ang-90:{sqrt(3)})$);
    \draw[auxwall] ($(\ang:{sqrt(25/6)})+(\ang+90:{sqrt(2)})$) -- ($(\ang:{sqrt(25/6)})+(\ang-90:{sqrt(2)})$);
    \draw[auxwall] ($(\ang:{sqrt(6)})+(\ang+90:{sqrt(1)})$) -- ($(\ang:{sqrt(6)})+(\ang-90:{sqrt(1)})$);
    \draw[auxwall] ($(\ang:{sqrt(6)})+(\ang+90:{sqrt(2)})$) -- ($(\ang:{sqrt(6)})+(\ang-90:{sqrt(2)})$);
  }
  \foreach \ang in {90}
  {
    \draw[blue] ($(\ang:{sqrt(1/6)})+(\ang+90:{sqrt(6)})$) -- ($(\ang:{sqrt(1/6)})+(\ang-90:{sqrt(6)})$);
  }
  \end{pgfonlayer}
  \end{scope}
  \begin{scope}[shift={(3.5,0)},scale=1]
  \node at (0,-3) {$\Scoxeter$};
  \coordinate (alpha1) at (0:{sqrt(2)});
  \coordinate (alpha2) at (150:{sqrt(6)});
  \foreach \ang in {60,120,...,360}
  {
    \draw[root] (0,0) -- (\ang:{sqrt(2)});
  }
  \foreach \ang in {30,90,...,330}
  {
    \draw[root] (0,0) -- (\ang:{sqrt(6)});
  }
  \coordinate (omega2) at (90:{sqrt(6)});
  \coordinate (omega1) at (60:{sqrt(2)});
  \begin{pgfonlayer}{weights}
  \foreach \ang in {60,120,...,360}
  {
    \node[weightdot] at (\ang:{sqrt(2)}) {};
    \node[weightdot] at (\ang+30:{sqrt(6)}) {};
  }
  \end{pgfonlayer}
  \begin{pgfonlayer}{labels}
  \node at ($3*1.2*(alpha1)+2*1.2*(alpha2)$) {$\root[0]$};
  \node at ($1.3*(alpha1)$) {$\root[1]$};
  \node at ($1.3*(alpha2)$) {$\root[2]$};
  \node[above right,orange] at ($1*(omega1)$) {$\coweight[1]$};
  \node[above,orange] at ($1*(omega2)$) {$\coweight[2]$};
  \end{pgfonlayer}
  \begin{pgfonlayer}{background}
  \foreach \ang in {30,90,150}
  {
    \draw[auxwall] (\ang:{sqrt(6)}) -- ({\ang+180}:{sqrt(6)});
  }
  \foreach \ang in {60,120,180}
  {
    \draw[auxwall] (\ang:{sqrt(8)}) -- ({\ang+180}:{sqrt(8)});
  }
    \path[fill=black!10,draw=blue] ($1*(omega2)$) -- (0,0) -- ($2*(omega1)$);
  \foreach \ang in {60,120,...,360}
  {
    \draw[wall] ($(\ang:{sqrt(1/2)})+(\ang+90:{sqrt(6)})$) -- ($(\ang:{sqrt(1/2)})+(\ang-90:{sqrt(6)})$);
    \draw[auxwall] ($(\ang:{sqrt(2)})+(\ang+90:{sqrt(6)})$) -- ($(\ang:{sqrt(2)})+(\ang-90:{sqrt(6)})$);
    \draw[auxwall] ($(\ang:{3*sqrt(1/2)})+(\ang+90:{sqrt(3/2)})$) -- ($(\ang:{3*sqrt(1/2)})+(\ang-90:{sqrt(3/2)})$);
  }
  \foreach \ang in {30,90,...,330}
  {
    \draw[wall] ($(\ang:{sqrt(3/2)})+(\ang+90:{sqrt(4)})$) -- ($(\ang:{sqrt(3/2)})+(\ang-90:{sqrt(4)})$);
    \draw[auxwall] ($(\ang:{sqrt(6)})+(\ang+90:{sqrt(2)})$) -- ($(\ang:{sqrt(6)})+(\ang-90:{sqrt(2)})$);
  }
  \foreach \ang in {60}
  {
    \draw[blue] ($(\ang:{sqrt(1/2)})+(\ang+90:{sqrt(6)})$) -- ($(\ang:{sqrt(1/2)})+(\ang-90:{sqrt(6)})$);
  }
  \end{pgfonlayer}
  \end{scope} 
\end{tikzpicture}
\end{center}

The spherical Coxeter complex associated to $\Phi$ is the tiling of the boundary sphere of $\Sregular$ obtained by cutting along the codimension-$1$ spheres that are boundaries of the hyperplanes $\root(\cdot) = 0$ for $\root \in \Phi$. The Euclidean Coxeter complex is the tiling of $\Sregular$ obtained by cutting along hyperplanes $\root(\cdot) = k$ for $\root \in \Phi$ and $k \in \Z$. The spheres and hyperplanes are called \emph{walls}, the maximal cells are \emph{chambers}. The spherical Coxeter group $W$ and the Euclidean Coxeter group $\tilde{W}$ are generated by the reflections in these spheres respectively hyperplanes. They act regularly on chambers.

The Euclidean Coxeter group $\tilde{W}$ is a split extension
\[
  A \to \tilde{W} \to W
\]
of the spherical Coxeter group $W$ by the subgroup $A$ of translations. Note that $A$ is naturally identified with the weight lattice $\gen{\weight[1],\ldots,\weight[k]} < \Scoxeter$. A splitting $W \to \tilde{W}$ is given by a choice of \emph{special} vertex, namely a vertex that lies on a wall from every parallel class of walls.

The numbering of roots allows us to assign types to vertices of the spherical Coxeter complexes: we declare that $\omega_i$ points in direction of a vertex of type $i$. Note that for buildings of type $C_n$ the type depends on whether we regard them as type $B_n$ or $C_n$.

We restrict to the two-dimensional case from now on. To uniformly phrase statements involving vertex types we will use the following notation. We will write $\{i,j\} = \{1,2\}$ to indicate that $i$ and $j$ should be the two types. The type of a vertex opposite to a vertex of type $i$ will be denoted $\bar{i}$; it is $j$ in type $A_2$ but $i$ in the other types. This is summarized in the following table

\begin{center}
\begin{tabular}{c|cccc}
  Type & $i$ & $\bar{i}$ & $j$ & $\bar{j}$\\
  \hline
  \multirow{2}{*}{$A_2$} & $1$ & $2$ & $2$ & $1$\\
  & $2$ & $1$ & $1$ & $2$\\
  \hline
  \multirow{2}{*}{$B_2$, $C_2$, $G_2$} & $1$ & $1$ & $2$ & $2$\\
  & $2$ & $2$ & $1$ & $1$\\
\end{tabular}
\end{center}

Now consider the Euclidean Coxeter complex structure on $\Scoxeter$: it is a decomposition into triangles (chambers) such that $\tilde{W}$ acts regularly on triangles. Occasionally we will make use of types of simplices in $\Scoxeter$. In that case we define types so that $0 \in \Scoxeter$ is of type $0$, its neighbor in direction $\weight[2]$ is of type $1$ and its neighbor in direction $\weight[1]$ is of type $2$. Using the chamber regular-action of $W$ this extends to a well-defined type of all vertices. The special vertices are those of type in $\{0,1,2\}$ in type $\tilde{A}_2$, of type in $\{0,2\}$ in type $\tilde{B}_2$, of type in $\{0,1\}$ in type $\tilde{C}_2$, and of type $0$ in type $\tilde{G}_2$. The \emph{type}  $\type(\sigma)$ of a simplex $\sigma$ is the set of types of its vertices. Its \emph{cotype} is $\{0,1,2\} \setminus \type (\sigma)$, in particular panels (edges) have singleton cotype. A map $\varphi \colon X \to Y$ of simplicial complexes with type functions $\type \colon X,Y \to \{0,1,2\}$ is \emph{type-preserving} if $\type(\varphi(\sigma)) = \type(\sigma)$ for all simplices $\sigma$ of $X$.

Let $\{i,j\} = \{1,2\}$. We let $\ell_i \colon \R \to \Scoxeter, t \mapsto (t/\norm{\weight[i]}) \weight[i]$ be the geodesic line spanned by $\weight[i]$. We refer to its endpoints in the visual boundary as $\ell_i(\infty)$ and $\ell_i(-\infty)$. We let
\begin{align*}
  S &= \R_{\ge 0} \cdot \weight[1] + \R_{\ge 0} \cdot \weight[2]\\
    &= \{x \in \Scoxeter \mid \root[1](x) \ge 0, \root[2](x) \ge 0\}
\end{align*}
be the Weyl chamber and refer to its visual boundary as $S(\infty)$. We let
\[
  \Sdetecting_i = \{x \in \Scoxeter \mid \alpha_j(x) \ge 0\}
\]
be the closed halfspace of $\Scoxeter$ containing $S$ bounded by $\ell_i$.

%
%

If $M \subseteq \Scoxeter$ is a subset by the \emph{combinatorial convex hull} $\conv M$ we mean the least convex subcomplex of $\Scoxeter$ that contains $M$. It is the smallest complex containing the metric convex hull of $M$ and is also the intersection of half-apartments
\[
  \conv M = \bigcap_{\root,k} \{x \in \Scoxeter \mid \root(x) \le k\}
\]
where the intersection ranges over pairs $\root \in \Phi,k \in \Z$ with $\root(M) \le k$.


\subsection{Euclidean buildings}
Let $\Scoxeter$ be a Coxeter complex of type  $\tilde{T}_2$ ($T \in \{A,B,C,G\}$).
A \emph{building} of type $\tilde{T}_2$ is a triangle complex $X$ with a type map satisfying the following building axioms, where any (type-preserving) isomorphic copy $\Sigma$ of $\Scoxeter$ in $X$ is called an \emph{apartment}
\begin{enumerate}
  \item for any two simplices $e, f \subseteq X$ there is an apartment $\Sigma$ containing $e$ and $f$
  \item for any two apartments $\Sigma, \Sigma'$ there is a (type-preserving) isomorphism $\Sigma \to \Sigma'$ whose restriction to $\Sigma \cap \Sigma'$ is the identity.
\end{enumerate}

Let $X$ be a locally finite building of type $\tilde{T}_2$. We let $X_i$ denote the set of vertices of type $i$. The \emph{thickness} of $X$ in a \emph{panel} (edge) is the number of chambers (triangles) that contain it. We will assume throughout that $X$ is \emph{thick} meaning that the thickness of every panel is at least $3$. It follows from the fact that opposite residues in spherical buildings are isomorphic that panels of same (co)type have same thickness. Thus there are numbers $q_0,q_1,q_2 \ge 1$ such that every panel of cotype $i$ has thickness $q_i + 1$. The fact just mentioned also implies that $q_0 = q_2$ in type $G$ and $q_0 = q_1 = q_2$ in type $A$.

A \emph{gallery} is a sequence of chambers any consecutive two of which share a panel. Length of a minimal gallery defines a \emph{gallery distance} on chambers. For any simplex $e \subseteq X$ and a chamber $c \subseteq X$, among the chambers containing $e$ there is a unique one that is closest in gallery distance to $c$, it is called the \emph{projection} $\pr_e c$ of $c$ to $e$. More generally, if $f \subseteq X$ is a simplex, the chambers $\pr_e d$ for $d > f$ intersect in a simplex $\pr_e f \supseteq e$, the \emph{projection} of $f$ to $e$.

\subsection{Boundaries of Euclidean buildings}\label{sec:boundaries}

Let $X$ be a two-dimensional Euclidean building. The visual boundary $\partial X$ is a spherical building whose apartments are the visual boundaries $\partial \Sigma$ of apartments. In particular, our discussion assigns a type ($1$ or $2$) to every vertex at infinity. If $e \subseteq X$ is a simplex and $\eta \subseteq \partial X$ is a simplex at infinity, in the convex hull of $e$ and $\eta$ (i.e.\ the minimal convex subcomplex that contains $e$ and contains $\eta$ in its boundary) there is a unique maximal simplex containing $e$; it is the \emph{projection} $\pr_e \eta$ of $\eta$ to $e$.

We denote the space of vertices at ininity of type $i$ by $\Delta_i$. The topology is the topology induced from $\partial X$, that is, if $\rho \colon [0,\infty) \to X$ is a ray with $\xi \defeq \rho(\infty) \in \Delta_i$ then a neighborhood basis of $\xi$ is given by $\{\rho'(\infty) \mid \rho'|_{[0,r]} = \rho|_{[0,r]}\}, r \ge 0$.

We denote by $\Delta$ the space of chambers at infinity. It can be topologized via the embedding $\Delta \to \partial X$ that takes each edge to its midpoint. A more workable description of the topology is as follows. If $S$ is a model Weyl chamber, we also write $S(\infty) \defeq \partial S$ for the corresponding edge at infinity. Then if $\iota \colon S \to X$ is an embedding, a neighborhood basis of $\iota(S(\infty))$ is $\{\iota'(S(\infty)) \mid \iota'|_K = \iota|_K\}$ for $K \subseteq S$ compact.

Pairs of opposite chambers in $\Delta$ form a space
\[
  \Deltaop = \{ (c,d) \in \Delta \times \Delta \mid c \text{ opposite }d\}.
\]
Similarly there are two spaces of pairs of opposite vertices
\[
  \Deltaopv[i] = \{ (\xi,\zeta) \in \Deltav[i] \times \Deltav[\bar{i}] \mid \xi \text{ opposite } \zeta\}
\]
one for each type $i \in \{1,2\}$.

If $\iota \colon \Scoxeter \to X$ is an embedding of a model apartment then the two chambers $\iota(S(\infty))$ and $\iota(-S(\infty))$ are opposite and conversely for any two opposite chambers lie in the boundary of an apartment (unique as a subspace). Thus the natural topology on $\Deltaop$ can be described by saying that a neighborhood of $(\iota(S(\infty)), \iota(-S(\infty)))$ is $\{(\iota'(S(\infty)), \iota'(-S(\infty))) \mid \iota'|_K = \iota|_K\}$ for $K \subseteq \Scoxeter$ compact.

Let $\{i,j\} = \{1,2\}$ and let $\xi$ be a vertex of type $i$. Its \emph{residue} $\Res \xi$ is the space $\{c \in \Delta \mid \xi \in c\} \subseteq \Delta$. It can be (homeomorphically) identified with $\{\zeta \in \Delta_j \mid [i,j] \in \Delta\}$ and we may do so implicitly. If $\zeta \in \Delta_{\bar{i}}$ is a vertex opposite $\xi$ then the building theoretic projection induces a homeomorphism $\pr_\xi|_{\Res \zeta} \colon \Res \zeta \to \Res \xi$, which is called a \emph{perspectivity}.

Before moving on we make a general remark about structures at infinity. Let $\rho,\rho' \colon \R \to X$ be two geodesic rays. There are two extremal asymptotic equivalence relations on rays. The most restrictive one is that of eventually coinciding (or having a common parametrized subray): $\rho|_{[R,\infty)} = \rho'|_{[R,\infty)}$ for some $R$. The most inclusive on is that of having bounded distance $\sup_t d(\rho(t),\rho(t')) < \infty$, which means that they define the same point at infinity. Between these two there are two intermediate relations only quotienting in the parallel respectively transverse direction: the rays can eventually coincide up to translation (have a common setwise subray): $\rho(t) = \rho'(T+t)$ for some $R,T$ and all $t \ge R$; or they can eventually be parallel in the sense that $\rho_{[R,\infty)}$ and $\rho'_{[R,\infty)}$ bound a flat strip so that $\rho(t)$ and $\rho'(t)$ are each others mutual projection for $t \ge R$, this is equivalent to $\rho$ and $\rho'$ defining the same Busemann function. Considering rays that define the same point at infinity up to the relation of eventually coinciding up to translation leads to a refined structure that turns out to be a tree and will interest us next. See \cite{Caprace09} for the general $\cato$ case. We note that for embeddings of sectors $S \to X$ we could do a similar construction, but there is no transverse direction so there are only two main asymptotic equivalence relations: that of having a common parametrized subsector and that of having a common setwise subsector.

The previous discussion leads to a more refined structure than $\Res \xi$ associated to $\xi$. Namely, consider the set $T_\xi$ of rays $\rho \colon [0,\infty) \to X$ with $\rho(\infty) = \xi$ up to the equivalence relation of eventually coinciding up to translation. If $[\rho],[\gamma] \in T_\xi$ are two points then $d(\rho(r),\gamma([0,\infty))$ and $d(\rho([0,\infty),\gamma(s))$ are eventually constant and equal and define a metric $d([\rho],[\gamma])$ on $T_\xi$. With respect to this metric $T_\xi$ is a metric tree whose vertices are $[\rho]$ for $\rho$ a ray contained in a wall. It is $(q_i+1)$-regular unless $(T,i) \in \{(B,2),(C,1)\}$ in which case it is $(q_0+1,q_i+1)$-biregular.

Let $S \subseteq X$ be a Weyl chamber with $S(\infty) \ni \xi$. A \emph{regular} ray is a ray in an apartment which is not parallel to any wall of this apartment. If $\gamma$ is a regular ray in $S$, then the rays $\rho_t$ starting in $\gamma(t)$ and tending toward $\xi$ have the property that $t \mapsto [\rho_t]$ is a ray in $T_\xi$. In this way $\partial T_\xi$ is naturally (and homeomorphically) identified with $\Res \xi$, see for example \cite[Lemma 2.2]{RemyTrojan}.

If $\zeta \in \Deltav[\bar{i}]$ is a vertex of type $\bar i$ we denote by $\Oppv[i](\zeta) \subseteq \Delta_i$ the space of vertices opposite it. If $\rho \colon (-\infty,0] \to X$ is a geodesic ray with $\rho(-\infty) = \zeta$ then $\Oppv[i](\zeta)$ is the set of $\ell(\infty)$ where $\ell \colon \R \to X$ is a geodesic line that coincides with $\rho$ on an infinite interval. We denote by $\Opp(\zeta) \subseteq \Delta$ the space of chambers opposite $\zeta$. If $\rho$ is as above, this is the set of $\iota(S(\infty))$ where $\iota \colon \Sregular \to X$ is an embedding such that $\iota \circ \ell_i$ coincides with $\rho$ on an infinite interval.

\begin{lemma}\label{lem:opp_decomposition_topological}
  Let $\zeta \in \Deltav[\bar{i}]$. The maps
  \begin{align*}
    \Opp(\zeta) &\leftrightarrow \Oppv[i](\zeta) \times \Res(\zeta)\\
    c & \mapsto (c_i,\pr_\zeta c)\\
    \pr_\xi(\tau) & \mapsfrom (\xi,\tau)
  \end{align*}
  where $c_i$ denotes the vertex of type $i$ of $c$, are mutually inverse homeomorphisms.
\end{lemma}

\begin{proof}
  That the maps are mutually inverse is clear from the bijectivity of perspectivities. Let $\iota \colon \Sregular \to X$ be an embedding with $\iota(\ell_i(-\infty)) = \zeta$. Let $K \subseteq \Sregular$ be the convex hull of a compact subcomplex and $\ell_i(-\infty)$ and consider $U_K(\iota) = \{\iota' \colon \Sregular \to X \mid \iota'|_K = \iota|_K\}$. Then the set of $\iota'(S(\infty))$ where $\iota'$ ranges over $U_K(\iota)$ is open in $\Opp(\zeta)$, the set of $(\iota'(\ell_i(\infty)), \iota'(T_i(\infty)))$ where $\iota'$ ranges over $U_K(\iota)$ is open in $\Oppv[i](\zeta) \times \Res(\zeta)$ and these sets form a basis for the topology.
\end{proof}

\subsection{Wall-trees and projectivities}\label{sec:projectivies}

Let $\xi \in \Delta_i$ be a vertex at infinity of $X$ and let $\zeta \in \Delta_{\bar{i}}$ be an opposite vertex. The union $Y_{\xi,\zeta}$ of all lines connecting $\zeta$ and $\xi$ is a convex subcomplex of $X$. Metrically it is isomorphic to a product \[\Ssingular_{\xi,\zeta} \cong \R \times T_{\xi,\zeta}\] of $\R$ and a tree $T_{\xi,\zeta}$ though the cell structure does not decompose as a direct product. We call $Y_{\xi,\zeta}$ the \emph{wall space} associated to $\xi$ and $\zeta$ and $T_{\xi,\zeta}$ the \emph{wall tree}. These spaces can also be called \emph{façade} and \emph{inner façade} (see \cite{RousseauBook}). We denote the projections to factors by
\begin{align*}
  p_T \colon \Ssingular_{\xi,\zeta} &\to T_{\xi,\zeta} && \text{and} & p_\R \colon \Ssingular_{\xi,\zeta} &\to \R.
\end{align*}
We denote by $\Aut(\Ssingular_{\xi,\zeta})$ the group of type-preserving automorphisms. We take $\Aut^0(\Ssingular_{\xi,\zeta})$ to be the subgroup that fixes $\xi$ and $\zeta$ and preserves the bipartition of $T_{\xi,\zeta}$; it has index at most $4$ in $\Aut(\Ssingular_{\xi,\zeta})$. The group $\Aut^0(\Ssingular_{\xi,\zeta})$ decomposes as a direct product of $\Z$ and $\Aut^0(T_{\xi,\zeta})$, the group of automorphisms of $T_{\xi,\zeta}$ that preserve the bipartition.

Restricting a line to a ray towards $\xi$ respectively $\zeta$ induces maps
\[
  T_\zeta \leftarrow \Ssingular_{\xi,\zeta} \to T_\xi
\]
both of which are $p_T$ up to an isomorphism. In particular there are isometries $T_\zeta \leftarrow T_{\xi,\zeta} \to T_\xi$ that compose to an isometry $\pr_\xi \colon T_\zeta \to T_\xi$. It is compatible with perspectivities and refines them in the sense that the following diagram commutes:
\begin{center}
\begin{tikzcd}
\partial T_\zeta \arrow[d, "\cong"] \arrow[r, "\pr_\xi"] & \partial T_\xi \arrow[d, "\cong"] \\
\Res \zeta \arrow[r, "\pr_\xi"]                          & \Res \xi                         
\end{tikzcd}
\end{center}
We therefore call them \emph{perspectivities} as well.

As a result we obtain the \emph{groupoid of projectivities} $\ProjGrpd$ whose objects are the vertices at infinity $\Deltav[1] \cup \Deltav[2]$ and whose morphisms $\zeta \to \xi$ are the homeoomorphisms $\Res(\zeta) \to \Res(\xi)$ or, equivalently, the isomorphisms $T_\zeta \to T_\xi$ that arise as compositions of perspectivities.

Since ($X$ and thus) $\partial X$ is a thick building, any two vertices of the same type can be connected by a sequence of pairwise opposite vertices. As a consequence in type $A$ the groupoid of projectivities is connected in the sense that there is a projectivity from any object $\zeta$ to any object $\xi$, while it has two components in types $C$ and $G$: $\Deltav[1]$ and $\Deltav[2]$. The isotropy group of $\ProjGrpd$ in $\xi$ will be denoted $\ProjGrp_\xi$ and called the \emph{projectivity group}. Note that projectivity groups in the same component are conjugate and therefore isomorphic, hence we have one abstract projectivity group $P = P_1 = P_2$ in type $A$ and two projectivity groups $P_1$ and $P_2$, one for each vertex type, in types $C$ and $G$.

\section{Amenable groups, amenable actions, and factors}\label{sec:outline}

\subsection{Amenable actions}

The main part of the proof of Theorem \ref{thm:main_theorem} is to prove that a proper quotient of $\Gamma$ is amenable. 

There are many equivalent definitions of amenable groups. The one we will use here is the following: a locally compact group $G$ is amenable if and only if each time it acts by affine maps on a weak-* compact subspace of a dual Banach space, it has a fixed point.



Amenable actions are supposed to generalize the notion of amenable groups, in a way that a group is amenable if and only if its action on a point is. However, it turns out that the various generalizations we could define from these various conditions are not equivalent. 

One of these definitions is the following. An action of a locally compact group $G$ on a locally compact space $Z$ is \emph{topologically amenable} if there exists a sequence of continuous maps $\mu_n \colon Z \to \Prob(G)$ such that $\norm{\mu_n(\gamma x) - \gamma \mu_n(x)} \to 0$ uniformly for $(\gamma,x)$ in every compact subspace of $G\times Z$. Here $\Prob(G)$ can be regarded as lying in the dual of the space of compactly supported continuous functions on $G$ and the norm is the dual norm: $\norm{\mu} = \sup_{\norm{f}_\infty \le 1} \int f d\mu$, and continuity of $\mu_n$ is understood with respect to the weak-* topology.


The following follows from the main theorem of \cite{Lecureux10} (see also \cite{KaimanovichAmenable} or \cite{RSAmenable} for previous results).

\begin{theorem}\label{thm:action_amenable}
  The action of $\Gamma$ on $\Delta$ is topologically amenable.\qed
\end{theorem}


The most commonly used notion of an amenable action is amenability in the sense of Zimmer \cite{ZimmerBook}. The relationship between the various notions of amenable actions is clarified in the book \cite{AnantharamanRenault}. One instance of such a relation is the following implication. A self-contained proof can be found in \cite[Proposition~2.5]{GuirardelHorbezLecureux22}.

\begin{proposition}\label{prop:amenable_action_map}
  Let $G$ be a group with a topologically amenable action on compact topological space $Z$. Let $G$ act affinely on a convex, weak-$*$-compact subset $K$ of a dual Banach space. Then there exists a $G$-equivariant measurable map $Z \to K$. \qed
\end{proposition}

 In particular, if $Z$ is equipped with a $G$-quasi-invariant probability measure, then pushing the measure by the map $Z\to K$ gives rise to a measurable quotient. In that case this is a slightly weaker version of Zimmer's amenability. 

\subsection{From the Factor Theorem to amenability}

In this paragraph we work in the category of measure spaces , which is the category whose objects are measure spaces, and morphisms are measurable, measure-class preserving maps, which are identified when the agree almost everywhere.

By a \emph{factor} of a measure space $Z$ we mean a measure class-preserving map $Z \to Q$. If $Z$ is a standard measure space, by \cite[Theorem~3.3.4]{AnantharamanPopa}, a factor is equivalent to a subalgebra of $L^\infty(X)$.  

There is a natural, $\Aut(X)$-invariant measure class on the space $\Delta$, classically called \emph{harmonic measure}, defined for example in \cite{Parkinson2} or \cite{RemyTrojan}. We will come back to this definition in Section \ref{sec:measures_on_delta}, and give a precise definition in our setting of prouniform measures. This measure will be denoted $\mu_\Delta^o$ (the measure itself depends on an origin $o$, although the measure class does not)

There are 4 obvious $\Gamma$-equivariant factors of $\Delta$, which are $\id:\Delta\to \Delta$, $\pi_1:\Delta\to \Delta_1$, $\pi_2:\Delta\to \Delta_2$, and $\pi_{\bullet}: \Delta\to \{\bullet\}$. 

Thus Theorem~\ref{thm:factorintro} can equivalently be rephrased as follows:

\begin{theorem}\label{thm:factor}
  Let $X$ be an irreducible two-dimensional building with spaces of chambers and vertices at infinity $\Delta$, $\Delta_1$, $\Delta_2$, and let $\Gamma < \Aut(X)$ be a uniform lattice. 

  Let $\pi:(\Delta,\mu_\Delta^o)\to (Z,\nu)$ be a $\Gamma$-equivariant factor. Then there is a $\Gamma$-equivariant  isomorphism of measure spaces $Z \to Y$, where $Y$ is either $\Delta,\Delta_1,\Delta_2$ or $\{\bullet\}$, which conjugates $\pi$ to $\id,\pi_1,\pi_2$ or $\pi_{\bullet}$.
\end{theorem}

\begin{proof}[Proof of Theorem~\ref{thm:main_theorem} using Theorem~\ref{thm:factor}]
  By passing to a finite index subgroup we may assume that the action of $\Gamma$ preserves types. We also replace $\Gamma$ by its image in $\Aut(X)$. Let $N \lhd \Gamma$ be a normal subgroup. We want to show that $N = \{1\}$ or that $\Gamma/N$ is amenable. Since $\Gamma$ has property (T) by \cite{Oppenheim} so does $\Gamma/N$. It then follows that $\Gamma/N$ is finite in the latter case.

  Let $\Gamma/N$ act on a convex, weak-$*$-compact subset $K$ of a dual Banach space. We regard this as an action of $\Gamma$ that restricts trivially to $N$. 
  The action of $\Gamma$ on $\Delta$ is amenable by Theorem~\ref{thm:action_amenable}. By Proposition~\ref{prop:amenable_action_map} there is thus a $\Gamma$-equivariant measurable map $\Delta \to K$. Pushing a measure in the canonical measure class on $\Delta$ to $K$ we regard this map as a factor. By Theorem~\ref{thm:factor} $K$ is one of $\Delta$, $\Delta_1$, $\Delta_2$ and a point, up to measure-class preserving isomorphism. Since the action of $\Gamma$ in $L^\infty(K)$ is faithful on the first three spaces (Lemma~\ref{lem:aefaithful}) it follows that $N = \{1\}$. In the last case the measure on $K$ is concentrated on a single point which is then a fixed point, showing that $\Gamma/N$ is amenable.
\end{proof}

Our goal for the rest of the paper is therefore to prove Theorem~\ref{thm:factorintro}.

\section{Flows}\label{sec:flows}

\subsection{Marked wall spaces}

We describe the spaces of marked flats and of restricted marked wall trees following \cite[Section~5.3]{BCL}. They are the spaces underlying the regular and singular Cartan flow respectively.

We start with a general basic remark that will be relevant throughout: we will often consider spaces $\Hom(A,B)$ where $A$ and $B$ are locally finite affine cell complexes (for instance $A = S$, $B = X$) and morphisms are cellular embeddings, possibly with additional constraints. The remark is that these spaces carry a unique reasonable topology which is at the same time (induced from) the topology of pointwise convergence and the compact open topology. The reason is that the image of a compact subspace is determined by the images of a finite subcomplex supporting it. For this reason we will not mention the topology explicitly. We will use the following notation to describe open neighborhoods. If $\iota \colon A \to B$ is a map and $C \subseteq A$ is a compact subset then
\[
  U_C(\iota) = \{\iota' \in \Hom(A,B) \mid \iota|_C = \iota'|_C\}
\]
is an open neighborhood of $\iota$. Since $B$ will always be locally finite, $U_C(\iota)$ is also compact when $C \ne \emptyset$.

The space of \emph{marked flats} is the set
\[
  \Fregular = \{\iota \colon \Sregular \to X \mid \iota \text{ injective and type-preserving}\}
\]
of type preserving embeddings of $\Sregular$ into $X$.

Coming to the space of (restricted) marked wall trees we fix a few choices. First let $\iota_0 \colon \Sregular \to X$ be an embedding of the model Coxeter complex $\Sregular$ into the building. Now let $\{i,j\} = \{1,2\}$. The embedding $\iota_0$ determines vertices at infinity $\xi_i = \iota_0(\ell_i(\infty))$ of type $i$ and $\zeta_i = \iota_0(\ell_i(-\infty))$ of type $\bar{i}$. We denote the associated wall space by $\Ssingular_i \defeq \Ssingular_{\xi_i,\zeta_i}$, the associated wall tree by $T_i$, and the associated projectivity group by $P_i$. We regard $\Ssingular_i$ as an abstract wall space with a chosen prefered embedding $\iota_0 \colon \Ssingular_i \to X$. In this distinction the vertices $\zeta_i$ and $\xi_i$ will be regarded as lying in the boundary of $\Ssingular_i$ rather than that of $X$. 

Let $A_i = \gen{\weight[i]} < A$ be the subgroup that leaves $\ell_i$ setwise invariant. Note that the action of $A_i$ on $\Sregular \subseteq \Ssingular_i$ canonically extends to a an action of $A_i$ on $\Ssingular_i$ by automorphisms that leave the projection $p_T \colon \Ssingular_i \to T_i$ invariant. Regarding $A_i$ as automorphisms of the abstract wall space $\Ssingular_i$ this gives an extension
\begin{equation}\label{eq:wall_space_tree_group_extension}
1 \to A_i \to \Aut^0(\Ssingular_i) \to  \Aut^0(T_i) \to 1
\end{equation}
where we denote the projection on the right by $p_T$ as well. We generally write the actions of $A$ on $\Fregular$ and of $A_i$ on $\Fsingular_i$ multiplicatively but sometimes write the actions of $A$ on $\Sregular$ and of $A_i$ on $\Ssingular_i$ additively: $(a.\iota)(x) = \iota(-a + x)$.

Let $\hat{\Fsingular}_i$ be the space of \emph{marked wall spaces}, namely type preserving embeddings $\Ssingular_i \to X$. It admits a (faithful) action by $\Gamma \times \Aut^0(\Ssingular_i)$ given by $((g,\alpha).\iota)(x) = g(\iota(\alpha^{-1}(x)))$. We define the relation $\sim_+$ (resp. $\sim_-$) by saying that $\iota\sim_+\iota'$ (resp. $\iota \sim_- \iota'$) if and only if for every $x\in Y_i$, if $\ell\colon\R_+\to X$ is the ray $[x,\iota(\xi_i))$ (resp. $[x,\iota(\zeta_i))$) then $\iota=\iota'$ on $\ell([t,+\infty))$ for some $t>0$ (which might depend on $x$).

We let $\sim$ be the equivalence relation on $\hat{\Fsingular}$ generated by the relation $\sim_+$, the relation $\sim_-$ and the orbit relation under $\Gamma \times A_i$. We let $M_i$ be the closure in $\Aut^0(\Ssingular_i)$ of the stabilizer of $[\iota_0]_\sim$ and denote the equivalence relation generated by $\sim$ and the orbit relation of $M_i$ by $\simeq$. We equip $M_i$ with the Haar measure that gives measure one to the stabilizer $(M_i)_0$ of the point $0\in \Sregular\subset \Ssingular_i$. Letting $\EProjGrp_i$ denote the image of $M_i$ under $p_T$ the extension \eqref{eq:wall_space_tree_group_extension} restricts to an extension

\begin{equation}\label{eq:extended_projectivity_group_extension}
1 \to A_i \to M_i \to \EProjGrp_i \to 1
\end{equation}
that splits non-canonically as a direct product.

\subsection{Extended projectivity groupoid}\label{sec:extended_projectivity_groupoid}

Here is another point of view on the construction, introducing notation that will be used throughout the article. First, define $\sigma_+: \hat\Fsingular_i\to \Delta_i$ (resp. $\sigma_-: \hat\Fsingular_i\to \Delta_{\overline{i}}$) as $\sigma_+(\iota)=\iota(\xi_i)$ (resp. $\sigma_-(\iota)=\iota(\zeta_i)$).

Consider the space
$\EDeltav[i]=\{(\xi, \iota)\mid \xi\in\Delta_i, \iota \in \Hom(T_i,T_u)\}$ of marked wall-trees and define a map $\overline \sigma_+ \colon \hat\Fsingular_i\to\EDeltav[i]$ (resp. $\overline \sigma_- :\hat\Fsingular_i\to\EDeltav[\bar{i}]$) as $\sigma_\pm(\iota)=(\sigma_\pm(\iota),\iota_{\pm\infty})$ where $\iota_{\pm\infty}(\lim_{r \to \pm\infty}(r,x)) = \lim_{r \to \pm\infty}\iota(r,x)$ when identifying $Y_i$ with $\R \times T_i$. We endow $\EDeltav[i]$ with the quotient topology.

This is well defined as the maps $\overline{\sigma}_+ \colon \hat\Fsingular_i \to \EDeltav[i]$ and $\overline{\sigma}_- \colon \hat\Fsingular_{\bar{i}} \to \EDeltav[i]$ induce the same topology on $\EDeltav[i]$. To see this note first that there is a homeomorphism $s \colon \Ssingular_i \to \Ssingular_{\bar{i}}$ that when restricted to $\Sregular_i$ takes $0$ to $0$, $\weight[i]$ to $-\weight[\bar{i}]$: if $i = \bar{i}$ it can be taken to act on every apartment as reflection along a line through $0$ and otherwise it identifies $\Ssingular_i$ and $\Ssingular_j$. It induces a homeomorphism $\Fsingular_i \to \Fsingular_{\bar{i}} \colon \iota \mapsto \iota \circ s$. Now $\overline\sigma_+(\iota) = \overline\sigma_-(\iota \circ s)$ and the claim follows.

\begin{lemma}
  The first projection $\overline \Delta_i\to \Delta_i$ is continuous.
\end{lemma}

\begin{proof}
  This follows from the fact that the projection $\hat\Fsingular_i\to \Delta_i$ is continuous.
\end{proof}

Note that $\iota \in \hat\Fsingular_i$ is uniquely determined up to $A_i$ by its values under $\overline \sigma_+$ and $\overline \sigma_-$, in other words:

\begin{lemma}
  The map $\overline \sigma_+ \times \overline \sigma_- \colon \hat\Fsingular_i \to \EDeltav[i] \cup \EDeltav[\bar{i}]$ descends to an injection $A_i \backslash \hat\Fsingular_i \to \EDeltav[i] \cup \EDeltav[\bar{i}]$.\qed
\end{lemma}

By construction, we have $\iota\sim_+\iota'$ if and only if $\overline \sigma_+(\iota)=\overline\sigma_+(\iota')$, and similarly for $\sim_-$.

 \begin{remark}
 There is an action of $\ProjGrpd$  on $\EDeltav[1] \cup \EDeltav[2]$  where $p.(\zeta',\iota)$ for $p \colon T_\zeta \to T_\xi$ and $(\zeta',\iota) \in \EDeltav[1] \cup \EDeltav[2]$ is defined if $\zeta = \zeta'$ and is then given by $(\xi, p \circ \iota)$. 
   The space $\EDeltav[1] \cup \EDeltav[2]$ is a marked version of the set $\bigsqcup_{\xi \in \Deltav[1] \cup \Deltav[2]} T_\xi$ underlying our groupoid $\ProjGrpd$. Unlike $\bigsqcup_{\xi \in \Deltav[1] \cup \Deltav[2]} T_\xi$ the space $\EDeltav[1] \cup \EDeltav[2]$ induces a non-discrete topology on $\Delta_1 \cup \Delta_2$. We expect that $\ProjGrpd$ should naturally carry the structure of a topological groupoid with respect to this coarser topology: in it, a projectivity $T_\zeta \to T_\xi$ can be close to a projectivity $T_{\zeta'} \to T_{\xi'}$ when $(\xi',\zeta')$ is distinct from but close to $(\xi,\zeta)$. However, we will not need this here and consequently we will work with the set $\bigsqcup_{\xi \in \Deltav[1] \cup \Deltav[2]} T_\xi$ of unmarked trees in what follows.
 \end{remark}

 The group $\Gamma$ acts on $\bigsqcup_{\xi \in \Deltav[1] \cup \Deltav[2]} T_\xi$ giving rise to a morphism $\zeta \to \gamma.\zeta$ represented by a homeomorphism $T_\zeta \to T_{\gamma.\zeta}$ for every $\zeta \in \Deltav[1] \cup \Deltav[2]$ and every $\gamma \in \Gamma$. We equip $\Homeo(T_\zeta,T_\xi)$ with the natural topology and define the \emph{extended projectivity groupoid} $\EProjGrpd$ to be the groupoid obtained by taking the groupoid generated by $\ProjGrpd$ and $\Gamma$ and taking the closure of morphisms $\zeta \to \xi$ as a subset of $\Homeo(T_\zeta,T_\xi)$. As for $\ProjGrpd$ we denote the isotropy group in $\xi$ by $\EProjGrp_\xi$ and up to isomorphism we have a single isotropy group for each type $i \in \{1,2\}$ (if $X$ is of type $\tilde{A}_2$ both of these are isomorphic as well). The following proposition implies that this isotropy group is in fact isomorphic to the group $Q_i$ in that appears in the extension \eqref{eq:extended_projectivity_group_extension} and we call it the \emph{extended projectivity group}.

 \begin{proposition}\label{prop:groupoid}
   For $\iota,\iota' \in \hat \Fsingular_i$ we have $\iota \simeq \iota'$ if and only if $\overline \sigma_+(\iota)$ and $\overline \sigma_+(\iota')$ lie in the same $\EProjGrpd$-orbit if and only if $\overline \sigma_-(\iota)$ and $\overline \sigma_-(\iota')$ lie in the same $\EProjGrpd$-orbit.
 \end{proposition}

 \begin{proof}
   First note that $\iota$ represents a perspectivity $\sigma_-(\iota) \to \sigma_+(\iota)$ thus $\overline \sigma_+(\iota)$ and $\overline \sigma_-(\iota)$ lie in the same $\ProjGrpd$-orbit and the last two conditions are equivalent. By definition $\iota \sim_+ a.\iota'$ for some $a \in A_i$ if and only if $\overline \sigma_+ = \overline \sigma_-$ and analogously for $\sim_-$. Since the orbit relation under $\EProjGrpd$ also contains the orbit relation under $\Gamma$ it follows that if $\iota \sim \iota'$ then $\overline \sigma_+(\iota)$ and $\overline \sigma_+(\iota')$ lie in the same $\EProjGrpd$-orbit.

   Conversely if $\overline \sigma_+(\iota)$ and $\overline \sigma_+(\iota')$ lie in the same $\EProjGrpd$-orbit, upon applying an element of $\EProjGrpd$ we may assume that $\sigma_+(\iota) = \sigma_+(\iota') = \xi_i$, i.e.\ that $\overline \sigma_+(\iota) = (\xi_i,\iota_\infty)$ and $\overline \sigma_+(\iota') = (\xi_i,\iota_\infty')$ with $\iota_\infty,\iota_\infty' \colon T_i \to T_{\xi_i} = T_i$ elements of $\Aut^0(\Ssingular_i)$. Multiplying by an apropriate element of $\Aut^0(\Ssingular_i)$, we may assume $\iota_\infty \in M_i$. By the definition of $\EProjGrpd$ we know that $\iota_\infty'$ is an accumulation point of a sequence of translates of $\iota_\infty$ under projectivities and elements of $\Gamma$. Since all of these are in relation $\sim$ with $\iota_\infty$ their limit is in relation $\simeq$, i.e.\ $\iota_\infty' \in M_i$.
 \end{proof}

 \subsection{Restricted marked wall spaces}

With these technical preparations in place we can define the \emph{space of restricted marked wall spaces}
\[
  \Fsingular_i = [\iota_0]_{\simeq}
\]
to be the $\simeq$-class of $\iota_0$. By construction it is $\Gamma\times M_i$ invariant.

 \begin{lemma}\label{lem:2-transitive}
  The space $\Fsingular_i$ is closed in $\hat{\Fsingular_i}$. The projection $\EProjGrp_i = p_T(M_i)$ contains the projectivity group $\ProjGrp_i$ and is 2-transitive on $\partial T_i$.

\end{lemma}

\begin{proof}
  That $\EProjGrp_i$ contains $\ProjGrp_i$ is clear by construction and that $\ProjGrp_i$ is $2$-transitive on $\partial T_i \cong \Res(\xi_i)$ was shown by Knarr \cite[Lemma~1.2]{Knarr88}. 

    The group $M_i$ is closed in $\Aut(\Ssingular_i)^0$ by definition. To prove that $\Fsingular_i$ is closed, let us first define $M'_i$ as the stabilizer in $\Aut(\Ssingular_i)^0$ of the class of $\iota_0$, so that $M_i$ is the closure of $M'_i$.

    We first claim that $\Aut(\Ssingular_i)^0$ acts transitively on $\hat\Fsingular_i/\sim$. Indeed, as the groupoid of perspectivities $\ProjGrpd$ acts transitively on $\Delta_i$, we see that each $\sim$-equivalence class intersects each fiber of the natural map $\hat\Fsingular_i \to \Deltaopv[i]$. Since $\Aut(\Ssingular_i)^0$ acts transitively on these fibers, the claim follows.
    
    Hence, the orbit map induces an equivariant continuous bijection 
    $$\hat\Fsingular_i/\sim\to \Aut(\Ssingular_i)^0/M'_i$$
    We therefore have canonical maps
    $$\hat\Fsingular_i\to \hat\Fsingular_i/\sim\to \Aut^0(\Ssingular_i)^0/M'_i\to \Aut(\Ssingular_i)^0/M_i$$
    The space $\Fsingular_i$ is the preimage of a point by the composition of all these maps. Since $M_i$ is closed in $\Aut(\Ssingular_i)^0$, the quotient is Hausdorff, so by continuity $\Fsingular_i$ is indeed closed in $\hat\Fsingular_i$.
\end{proof}


\begin{proposition}\label{prop:restricted_wall_trees}
%

  The action of $M_i$ on  each fiber $(\sigma_+\times\sigma_-)(u,v)$ (for $(u,v)\in\Deltaopv[i]$) is transitive and the map \[\Fsingular_i \to \Deltaopv[i], \iota \mapsto (\iota(\ell_i(\infty),\ell_i(-\infty))\] is surjective.


The action of $M_i$ on $\Fsingular_i$ is proper and free.
\end{proposition}

\begin{proof}
  By transitivity of the action of the groupoid $\EProjGrpd$ on $\Delta_i$, the first projection $\Fsingular_i\to \Delta_i$ is surjective. Applying once again a projectivity it follows that the map  $\Fsingular_i\to \Deltaopv[i]$ is surjective.

  Since $\Aut^0(Y_i)$ acts transitively on each fibers of the map $\sigma_+\times\sigma_-:\hat\Fsingular_i\to \Deltaopv[i]$, by definition of $M_i$ the action of $M_i$ on the fibers of $\sigma_+\times\sigma_-$ is also transitive.

It is clear that the action of $M_i$ on $\Fsingular_i$ is free.  The properness of the action of $M_i$ follows from the fact that it is a closed subgroup of $\Aut(\Ssingular_i)^0$, and the properness of the action of this group on $\hat\Fsingular_i$.
\end{proof}

\begin{lemma}
  The action of $\Gamma$ on $\Fregular$ and on $\Fsingular_i$ is proper and cocompact.
\end{lemma}

\begin{proof}
  Let $X_0$ be the set of vertices of type $0$ of $X$ and consider the map $\operatorname{ev}_0 \colon \Fsingular_i \to X_0, \iota \mapsto \iota(0)$ which is $\Gamma$-equivariant. The map is continuous and proper since $U_{\{0\}}(\iota)$ is compact open for every $\iota \in \Fsingular_i$. Continuity together with properness of $\Gamma \curvearrowright X_0$ shows that the action of $\Gamma$ on $\Fsingular_i$ is proper. That $\Gamma \curvearrowright \Fsingular_i$ is cocompact follows from surjectivity and properness of $\operatorname{ev}_0$ and cocompactness of $\Gamma \curvearrowright \Fsingular_i$. The proof for $\Fregular$ is the same.
\end{proof}

\begin{remark}
    If $\Gamma$ is an arithmetic lattice in a Bruhat-Tits building, then the group of extended projectivities is already the group of projectivities, and it is closed in $\Aut(T_i)$. We do not know in general if there exists examples in which the group of extended projectivities is not closed, or even in which the group of projectivities is of infinite index in the group of extended projectivities.
\end{remark}

While we do not equip $\ProjGrpd$ with the structure of a topological groupoid we need the following fact (which would be part of the verification that it is). 
For $n\geq 2$ let
\[
  \Delta_{i}^{n,\op}=\{(\xi_1,\xi_2,\dots,\xi_n)\in\Deltav[i]\times\Deltav[\bar{i}]\times\ldots\mid\xi_j\text{ opposite }\xi_{j+1},\ 1\le j<n\}
\]
(so $\Deltaopv[i]=\Delta_i^{2,\op}$). If $(\xi,v)\in\Deltaopv[i]$, we write $[\xi,\zeta]$ for the perspectivity from $\zeta$ to $\xi$ (either as a map $\Res(\zeta)\to\Res(\xi)$ or $T_\zeta\to T_\xi$). Similarly, if $(\xi_1,\xi_2,\dots,\xi_n)\in\Delta_i^{n,\op}$ we denote by $[\xi_1,\xi_2,\dots,\xi_n]$ the combinatorial projectivity obtained as the composition
\[
  [\xi_1,\xi_2,\dots,\xi_n]=[\xi_{1},\xi_2]\circ[\xi_{2},\xi_{3}]\circ\dots\circ[\xi_{n-1},\xi_n].
\]
Finally put $\Delta_i^{n,\op}(\xi,\zeta) = \{(\xi_1,\ldots,\xi_n) \in \Delta_i^{n,\op} \mid \xi_1 = \xi, \xi_n = \zeta\}$.

\begin{lemma}\label{lem:projectivity_multiplication_continuous}
  The map
  \begin{align*}
    \Delta_i^{n,\op}(\xi,\zeta) &\to \Homeo(T_\zeta,T_\xi)\\ (\xi_1,\ldots,\xi_n) &\mapsto [\xi_1,\ldots,\xi_n]
  \end{align*}
  is continuous.
\end{lemma}

\begin{proof}
  Let $(\xi_1,\ldots,\xi_n) \in \Delta_i^{n,\op}$ and put $\zeta = \xi_1$ and $\xi = \xi_n$. Let $p \in T_\zeta$ and $r > 0$ be arbitrary defining a compact neighborhood $B_1 \defeq \bar{B}_r(p) \subseteq T_\zeta$. We need to show that there are neighborhoods $U_j$ of $\xi_j$, $1 < j < n$, such that $[\zeta,\xi_2',\ldots,\xi_{n-1}',\xi]$ coincides with $[\zeta,\xi_2,\ldots,\xi_{n-1},\xi]$ when $\xi_j' \in U_j$.

  Let $L_1$ be the union of all lines in $X$ with endpoints $\zeta$ and $\xi_2$ whose projection to $T_\zeta$ lies in $\bar{B}_r(p)$: it is a subspace homeomorphic to $\bar{B}_r(p) \times \R$ and defines a subset $B_2 \subseteq T_{\xi_2}$. Inductively define $L_j$ to be the union of all lines with endpoints $\xi_j$ and $\xi_{j+1}$ whose projection to $T_{\xi_j}$ is in $B_{j}$ and defines a suset $B_{j+1} \subseteq T_{\xi_{j+1}}$.

  There is a point $c_2 \in X$ such that the ray from $c_2$ to $\xi_2$ is contained in both $L_1$ and $L_2$. By local finiteness of the tree we can moreover choose $c_2$ such that it has a relatively compact open neighborhood $C_2$ whose projection to $T_{\xi_2}$ contains $B_2$ and such that $L_1 \cap C_2 = L_2 \cap C_2$. Thus the restriction of the projectivity $[\zeta,\xi_2,\xi_3]$ to $B_1$ can be computed by tracing a ray from $\zeta$ to $C_2$ and then a ray from $C_2$ to $\xi_3$. Similarly there are set $C_j$ whose intersections with $L_{i-j}$ and $L_j$ coincide and whose projection to $T_{\xi_j}$ contains $B_j$. So the restriction of $[\xi_1,\ldots,\xi_n]$ to $B_1$ can be computed by tracing a ray from $\xi_1$ to $C_2$, then a segment to $C_3$, and so on, and finally a ray from $C_{n-1}$ to $\xi_n$. So it suffices to take $U_j$ to consist of all points $\xi_j'$ such that geodesic segments in $L_j \cap C_j$ admit extensions to rays to $\xi_j'$.
\end{proof}

\section{Convex geometry on trees}\label{sec:convex_geometry_on_trees}

In this section we study convex subcomplexes of complexes isometric to products of a line and a tree. 
Understanding inclusions of such complexes will be crucial in the next section, as it will allow us to finally construct the natural measure on $\Fsingular_i$. The restriction to a product of a line and a tree might seem artificial to the reader, but it turns out that this is in some sense optimal, 
as illustrated in Example~\ref{ex:not_product-symmetric}.

\subsection{Convex functions on trees}

Let $T$ be a (possibly finite) metric tree, more specifically, a simplicial tree equipped with a complete geodesic metric that gives every edge a length that is bounded away from $0$ (we will only need trees with finitely many edge lengths).

A function $f \colon T \to \R$ is \emph{convex} if for every geodesic $\gamma \colon [0,\ell] \to T$ the composition $f \circ \gamma$ is convex. It will be formally convenient to also allow the constant function $f \colon T \to \{\infty\}$ to be convex.

Our interest in convex functions stems from the following elementary observation:

\begin{lemma}\label{lem:convex_sets_and_functions}
  Let $f, g \colon T \to \R \cup\{\infty\}$ be convex functions with $f + g \ge 0$. Then
  \[
    C_{T,f,g} = \{(x,y) \in T \times \R \mid f(x) \le y \le -g(x)\} \subseteq T \times \R
  \]
  is closed convex. Conversely if $C \subseteq T \times \R$ is closed convex then $S = p_T(C)$ is a (possibly not simplicial) subtree and
  \begin{align*}
    g(x) &\defeq -\max\{ y \in p_T^{-1}(x) \} && \text{and} & f(x) & \defeq \min\{y \in p_T^{-1}(x)\}
  \end{align*}
  define convex functions $f,g \colon S \to \R \cup \{\infty\}$ such that $C = C_{S,f,g}$.
\end{lemma}

In particular $C_{T,f,g}$ is convex if and only if $C_{T,f,\infty}$ and $C_{T,\infty,g}$ are convex.

We say that $f \colon T \to \R$ is \emph{affine} if it its restriction to every (closed) edge is affine. Thus an affine function is determined by the values it takes on the set $V$ of vertices.

For an affine function $f$ we let
\[
  f_{vw} = \frac{f(w) - f(v)}{d(v,w)}
\]
denote the slope of $f$ on the edge from $v$ to $w$. Note that $f_{vw} = -f_{wv}$.

\begin{lemma}
  An affine function $f \colon T \to \R$ is convex if and only if for every $v \in V$ and every $w,w' \in \lk(v)$ the inequality
  \begin{equation}\label{eq:locally_convex}
    f_{vw} + f_{vw'} \ge 0
  \end{equation}
  is satisfied.
\end{lemma}

\begin{proof}
  Let $\gamma \colon [0,\ell] \to T$ be a geodesic. The function $f \circ \gamma$ is convex if and only if it is locally convex, so it suffices for $f$ to be convex on sets of the form $\gamma(U)$ where $U$ are arbitrarily small open. Such $\gamma(U)$ are contained in the union of two edges when $U$ is small enough. Since $f$ is affine on edges it is certainly convex in interior points of edges. In the vertex $v$ the given condition is visibly necessary and sufficient.
\end{proof}

Write $\inf(f) \defeq \inf \{f(x) : x \in T\}$ and put $M(f) = \{x \in T \mid f(x) = \inf(f)\}$ if $\inf(f) > - \infty$ and $M(f) = \emptyset$ otherwise.

\begin{lemma}
  Let $f \colon T \to \R$ be affine and convex. At every vertex $v \in V$ there is at most one neighbor $w \in \lk(v)$ with $f(w) < f(v)$ and the set $M(f)$ is either a simplicial subtree of $T$ or is empty.
\end{lemma}

\begin{proof}
  The first claim is immediate from \eqref{eq:locally_convex}. That $M$ is convex is just Lemma~\ref{lem:convex_sets_and_functions} applied with $g$ constant to $\inf(f)$. But since $f$ is affine and convex if it is constant on a neighborhood of an interior point of an edge, it is constant on the whole edge. Thus $M(f)$ is a subcomplex.
\end{proof}

We put $f_v = \min \{f_{vw} : w \in \lk(v)\}$. The previous lemma tells us that if $f_v < 0$ then this minimum is attained in a unique neighbor.

\begin{remark}\label{rem:convex_without_chamber}
Note that there are three (non-disjoint) cases in which $C_{T,f,g}$ has dimension $<2$ (empty interior): if $T$ is a single vertex, if $T$ is a linear graph and $f + g = 0$, or if $f$ and $g$ are constant with $f + g = 0$. Indeed, if $f + g = 0$ and a vertex $v$ of $T$ has degree $\ge 3$ then $f_{vw} + f_{vw'} = 0$ for any two neighbors $w,w' \in \lk v$ and there is at most one neighbor $w$ with $f_{vw} > 0$ and at most one with $f_{vw} < 0$ so $f_{vw} = 0$ for all $w \in \lk(v)$. Of the three cases the third one is the only one that is not contained in a flat.
\end{remark}

\subsection{Elementary extensions in wall spaces}\label{sec:elementary_extensions}

Now let $\Ssingular = \Ssingular_i$ be a wall space equipped with its projections $p_T \colon \Ssingular \to T$ and $p_\R \colon \Ssingular \to \R$. In what follows we will make use of the orientation and origin of $\R$ (rather than just its metric) but it will be clear that nothing depends in a substantial way on the choices. Note that $T$ has two different simplicial structures: the coarser one has as vertices the points where $T$ branches, the finer one has as vertices the projections of vertices of $\Ssingular$. We refer to these as the \emph{coarse} and \emph{fine} structure respectively. The relevance is that if $C \subseteq \Ssingular$ is a convex subcomplex then it can be written as $C_{S,f,g}$ as in Lemma~\ref{lem:convex_sets_and_functions} where $S$ is a subtree and $f$ and $g$ are affine both with respect to the fine structure.

Since we will exclusively be interested in subcomplexes, by the convex hull of $M$ we will again mean the combinatorial convex hull, i.e.\ the least convex subcomplex $\conv M$ that contains $M$.

We may regard $\Ssingular$ as a non-thick building of the same type as $X$ and, in particular, speak of apartments. Note, however, that these have a very simple description: they are the sets $p_T^{-1}(L)$ where $L \subseteq T$ is a biinfinite line (an apartment of $T$).
If $v$ and $w$ are adjacent vertices of $T$ we denote by $T_{vw}$ the closure of the component of $T \setminus \{v\}$ that contains $w$; it is a rooted tree with root $v$.

Let $C \subseteq \Ssingular$ be a convex subcomplex and write it as $C_{S,f,g}$ as in Lemma~\ref{lem:convex_sets_and_functions}. We want to understand the minimal convex complexes $D$ that properly contain $C$. Leaving aside the case that $D$ has dimension $<2$ any such complex is the convex hull of $C$ and a chamber $c$ that shares an edge $e$ with $C$. If this convex hull is minimal among convex complexes properly containing $C$ we call it an \emph{elementary extension} and more specifically the elementary extension of $C$ by $c$ along $e$. Understanding exactly which chambers $c$ actually give rise to an elementary extension is slightly subtle as the following example illustrates.

\begin{figure}[htb]
  \begin{tikzpicture}[scale=.4]
  \pgfdeclarelayer{nodelayer}
  \pgfdeclarelayer{edgelayer}
  \pgfsetlayers{edgelayer,main,nodelayer}
	\begin{pgfonlayer}{nodelayer}
		\node [style=dot] (0) at (-2, 0) {};
		\node [style=dot] (1) at (0, 0) {};
		\node [style=dot] (2) at (2, 1.75) {};
		\node [style=dot] (3) at (2, -1.75) {};
		\node [style=dot] (4) at (-3.75, -2) {};
		\node [style=dot] (5) at (-3.75, 2) {};
		\node [style=dot] (6) at (2.5, 4.5) {};
		\node [style=dot] (7) at (4.75, 2.5) {};
		\node [style=dot] (8) at (5, -1) {};
		\node [style=dot] (9) at (3.25, -4) {};
		\node [style=dot] (10) at (3.75, -6.75) {};
		\node [style=dot] (11) at (6.25, -5) {};
		\node [style=dot] (12) at (7.25, -2.25) {};
		\node [style=dot] (13) at (6.75, 0) {};
		\node [style=dot] (14) at (7.25, 2.75) {};
		\node [style=dot] (15) at (6.25, 4.5) {};
		\node [style=dot] (16) at (4.75, 5.75) {};
		\node [style=dot] (17) at (2.75, 7) {};
		\node [style=dot] (18) at (-3.5, 4) {};
		\node [style=dot] (19) at (-6, 2.25) {};
		\node [style=dot] (20) at (-6.25, -2.5) {};
		\node [style=dot] (21) at (-4, -4) {};
	\end{pgfonlayer}
	\begin{pgfonlayer}{edgelayer}
		\draw [magenta] (0) to (1);
		\draw [->-] (1) to (2);
		\draw [->-] (2) to (6);
		\draw [->-] (2) to (7);
		\draw [->-,magenta] (3) to (9);
		\draw [->-] (3) to (8);
		\draw [] (0) to (4);
		\draw [magenta] (0) to (5);
		\draw [->-] (4) to (21);
		\draw [] (4) to (20);
		\draw [->-,red] (5) to (19);
		\draw [->>-,blue] (5) to (18);
		\draw [->>-] (6) to (17);
		\draw [->>-] (6) to (16);
		\draw [->>-] (7) to (15);
		\draw [->>-] (7) to (14);
		\draw [->>-] (8) to (13);
		\draw [->>-] (8) to (12);
		\draw [->>-,red] (9) to (11);
		\draw [->-,blue] (9) to (10);
		\draw [->-,magenta] (1) to (3);
	\end{pgfonlayer}
\end{tikzpicture}\hfill
\includegraphics[width=.49\textwidth]{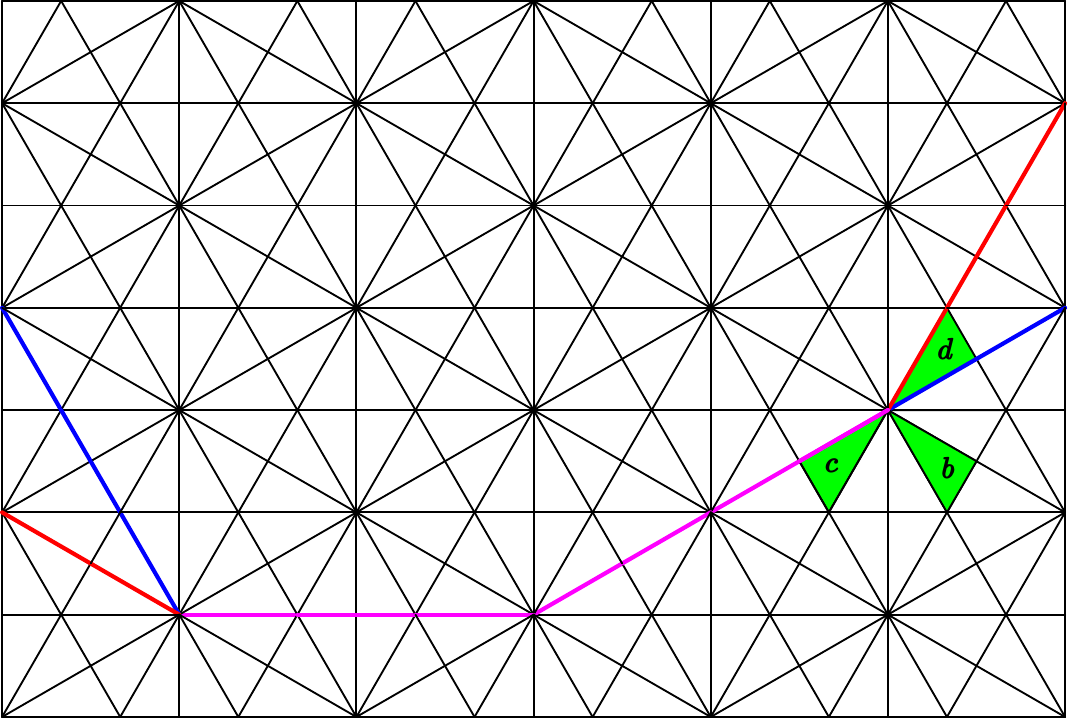}
\caption{A convex set $C = C_{S,f,\infty}$ in $\Ssingular_1$ in type $\tilde{G}_2$. The picture on the left is the tree $S$ together with each edge marked by arrows encoding the slope in the corresponding direction: no arrow means slope $0$ and more arrows mean increasingly steeper slopes among the ones possible in the type. Two maximal lines of $S$ are marked in red and blue respectively. Parts of the preimages of both lines under $p_T$ are (simulteanously) shown on the right. The convex set $C$ is the region above the graph of $f$. The lines show the graph of $f$ in the preimage of the line of the same color. In both pictures the overlap of red and blue is shown in magenta.
The chamber $d$ lies in the read apartment, the chamber $b$ lies in the blue apartment. Looking only at the red apartment it seems that $d$ does not lie in the convex hull of $C$ and $c$ but since $b$ lies in that convex hull, so does $d$.}
\label{fig:example_non-minimal}
\end{figure}
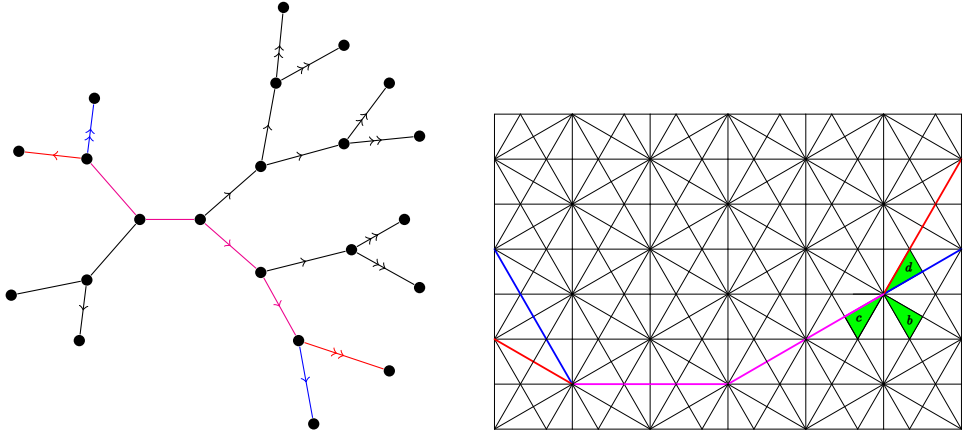
\begin{example}\label{ex:non-admissible_component}
  We consider the convex set $C_{S,f,\infty} \subseteq Y_1$ shown in Figure~\ref{fig:example_non-minimal}. We are interested in the convex hull of $C$ and the chamber $c$. Judging from the apartment indicated in red we would expect $d$ not to lie in the convex hull. However, the apartment indicated in blue shows that $b$ lies in the convex hull. Then looking simultaneously at the red and the blue rightmost strip shows that $d$ lies in the convex hull of $C$ and $b$. On one hand this illustrates that the convex hull of $C$ and $c$ cannot be understood by considering one apartment containing $c$ at a time; on the other hand it shows that the convex hull of $C$ and $c$ is not minimal, namely the convex hull of $C$ and $d$ is properly contained in it.
\end{example}

The rest of the section is technical but is entirely a consequence of understanding the example: If some vertex contains an upward slope that is steeper than the negative of its downward slope then any elementary extension first needs to aim at bringing some such upward slope down before eventually an elementary extension can move properly below the vertex. In order to bring the upward slope down it may be necessary to apply a similar process to a vertex further up. Some subtlety is added by the possibilities that there is no downward slope or that the downward slope is actually flat.

We embark on an investigation of elementary extensions of $C$ along edges of the form $[(v,f(v)), (w,f(w))]$, i.e.\ along the bottom of $C$. Extensions along edges $[(v,g(v)), (w,g(w))]$ are of course completely analogous and so are implicitly covered by the discussion. The remaining case of \emph{vertical} edges $e$, namely edges with $p_T(e)$ a single vertex, is more straightforward and will be treated later.

We continue to consider $C \subseteq \Ssingular$ written as $C_{S,f,g}$. If $S$ is a finite tree, a \emph{segment of constant slope} in $S$ is the image of a geodesic $\gamma \colon [0,\ell] \to S$ such that $f \circ \gamma$ has constant slope. It is \emph{maximal} if it admits no extension of constant slope. Combinatorially a segment of constant slope is a path of edges $[v_0,v_1] \cup \ldots \cup [v_{k-1},v_k]$ such that for $1 \le i < k$ the condition $f_{v_{i-1}v_i} = f_{v_iv_{i+1}}$ is satisfied. It is maximal if $f_{v_0w} + f_{v_0v_1} > 0$ for every $w \in \lk(v_0) \setminus \{v_1\}$ and $f_{v_kw} + f_{v_kv_{k-1}} > 0$ for every $w\in \lk(v_k)\setminus v_{k-1}$. If $S$ is infinite, $\gamma \colon I \to S$ may be defined on $I = \R_{\ge 0}$ or $\R$, and the path of edges $[v_i,v_{i+1}]$ may take index set $\N$ or $\Z$, the condition for maximality only applying to the finite end if there is one.

We say that two segments of constant slope \emph{meet} if they have an edge in common. A \emph{component of constant slope} (component for short) is a (non-empty, non-singleton) subgraph that contains every segment of constant slope that it meets and is minimal with respect to this condition. To get a combinatorial description we may define the relation $\sim$ by saying that $e \sim e'$ for edges $e = [v,w]$, $e'= [v,w']$ that share a vertex if $f_{vw} + f_{vw'} = 0$ and taking the transitive closure. Then components of constant slope are unions over the (closed) edges in an equivalence class. For orientation observe that the edges of minimal non-zero slope pointing rightward in Figure~\ref{fig:example_non-minimal} form two separate components as the two lowest ones cannot be extended to a common segment of constant slope despite having the same slope. Note that every component $R$ is a subtree whose leaves $v$ are endpoints of maximal segments of constant slope and therefore are characterized by the fact that $f_{vw} + f_{vw'} > 0$ for every $w' \in \lk(v) \setminus \{w\}$ where $w$ is the unique neighbor of $v$ in $R$. The relatively interior vertices $v$ have the property that there are two neighbors $w,w' \in R$ with $f_{vw} + f_{vw'} = 0$. They may however have neighbors $x$ with $f_{vw} + f_{vx} > 0$ for every $w \in \lk(v)$. We saw this in Example~\ref{ex:non-admissible_component} and it was the reason why the extension was not minimal.

This motivates calling an component $R$ \emph{admissible} if this does not happen, i.e.\ if for every interior vertex $v$ of $R$ there is a neighbor $w$ such that for every other neighbor $x$ (in $S$!) the equality $f_{vw} + f_{vx} = 0$ holds. This has a simple combinatorial description: every neighbor of $v$ in $S$ is already in $R$.

With these preparations in place, given $C = C_{S,f,g}$ and an admissible component $R$ we define the (downward) extension of $C$ along $R$ to be the complex $D$ characterized by how it intersects certain apartments as follows: if $\Sregular = p_T^{-1}(L)$ is an apartment such that $L$ contains a maximal segment of constant slope of $R$, if $e = [v,w]$  is an edge of $L \cap R$ and if $c$ is the unique chamber of $\Ssingular$ below the edge $[(v,f(v)),(w,f(w))]$ of $F$ then
\begin{equation}
  D \cap F = \conv_F((C \cap F) \cup c).\label{eq:elementary_extension} 
\end{equation}

If $R$ is an admissible component of $g$ the upward extension of $C$ along $R$ is defined analogously.

\begin{proposition}\label{prop:elementary_extension_well-defined}
  Let $R$ be an admissible component. The extension $D$ of $C$ along $R$ is well-defined, i.e.\ completely characterized by \eqref{eq:elementary_extension}, independent of the choice of $e$ and $c$, and is a convex subcomplex of $\Ssingular$.
\end{proposition}

\begin{proof}
  We prove the case where $R$ is a finite tree, the infinite case being similar with fewer constraints.

  Let $M \subseteq R$ be a maximal segment of constant slope, let $L \supseteq M$ be a bi-infinite line so that $E = p_T^{-1}(L)$ is an apartment of $\Ssingular$. Note that $M$ is a segment of $T$ with respect to the fine structure and let $\hat{M}$ be the least subcomplex with respect to the coarse structure that contains $M$. Since $T$ only branches in vertices of the coarse structure, $\hat{M}$ is a segment and $\hat{M} \subseteq L$. The following part of the proof is concerned with analysing the situation inside $E$ and so does not involve branching in the tree. This means for instance that we can identify $L$ with $\R$ and take affine functions to be piecewise affine functions on $\R$.

  Let $C_E = C \cap E$. Let $c$ be a chamber of $E$ as defined before \eqref{eq:elementary_extension}, and let 
  Let $D_E = \conv_E(C_E \cup c)$ be the convex hull of $C_E$ and $C$ inside $E$. According to \eqref{eq:elementary_extension} $D_E$ is supposed to be $D \cap E$. Our plan is to show first that $D_E$ does not depend on the choice of $c$ (or equivalently of the edge $e$) in $E$, and then that $D_E \setminus C_E$ not only is contained in $p_T^{-1}(\hat{M})$ but also is determined by the intersection of $C_E \cup c$ with that strip. Since every other apartment $E'$ that contains $p_T^{-1}(M)$ also contains $p_T^{-1}(\hat{M})$ independence of $E$ will follow.

  Write the convex subset $C_E$ of the Euclidean Coxeter complex $E$ as $C_{N,f,g}$. That is, $N = p_T(C) \cap L$ and we denote the restrictions of $f,g \colon p_T(C) \to \R$ by $f$ and $g$ as well. The function $f$ has constant slope on $M$ and is strictly convex in the endpoints of $M$. From this independence of $e$ (and thus $c$) within $E$ is easily seen: if $e = [(v,f(v)),(w,f(w))]$ and $e' = [(v',f(v')),(w',f(w'))]$ are two edges of $E$ with $v,w,v',w' \in M$ and $c$ and $c'$ are the chambers containing $e$ respectively $e'$ that are not in $C_E$ then they are each others mutual combinatorial projections: $\proj_{e'} c = c'$, $\proj_e c' = c$ since they both lie on the same side of the wall that $e$ and $e'$ lie on. So if $D_E$ contains $c$ then it contains $c'$ and vice versa.

  Now let $v$ be an endpoint of $M$ and let $w \in M$ be such that the segment $[(v,f(v)),(w,f(w))]$ is the edge of the boundary of $C$ that contains $(v,f(v))$ and lies above $M$.
  For a slope $\lambda$ let $x \mapsto (x,h_{v,\lambda}(x))$ with $h_{v,\lambda} \colon M \to \R$ be the ray through $(v,f(v))$ with slope $\lambda$. It is given by
    \[
      h_{v,\lambda} \left(v + t \frac{w-v}{d(v,w)}\right) = f(v) + t \cdot \lambda.
  \]
  The slopes $\lambda$ such that $h_{v,\lambda}$ lies inside a wall depends on the vertex $(v,f(v))$ and if $\lambda$ is one of them then so is $-\lambda$ (because walls extend through the vertex). Let $c$ be the chamber containing $e$ that is not contained in $C$ and let its third vertex be $(x,y)$.

  \begin{figure}
  \begin{tikzpicture}[scale=2]
  \coordinate (WW) at (-.7,-.1) {};
  \coordinate (V) at (0,0) {};
  \coordinate (W) at (.8,1) {};
  \coordinate (X) at (.6,.4) {};
  \node[dot] at (V) {};
  \node[dot] at (W) {};
  \node[dot] at (X) {};
  \node[dot] at (WW) {};
  \node[below] at (V) {$(v,f(v))$};
  \node[above,xshift=-17] at (W) {$(w,f(w))$};
  \node[below,xshift=5] at (X) {$(x,y)$};
  \node[above] at (WW) {$(w',f(w'))$};
  \node[right] at ($-1.7*(WW)$) {$\{(t,h_{v,\lambda}(t))\}$};
  \node at ($.33*(V)+.33*(X)+.33*(W)$) {$c$};
  \node at (-2,.5) {$E$};
  
  \draw ($1.7*(WW)$) -- ($-1.7*(WW)$);
  \draw (V) -- ($1.3*(W)$);
  \draw (V) -- (X) -- (W);
  
  \begin{scope}[shift={(0,-.7)}]
  \node[dot] at (-.7,0) {};
  \node[dot] at (0,0) {};
  \node[dot] at (.8,0) {};
  \node[dot] at (.6,0) {};
  \draw (-1.5,0) -- (1.5,0);
  \node[below] at (-.7,0) {$w'$};
  \node[below] at (0,0) {$v$};
  \node[below] at (.6,0) {$x$};
  \node[below] at (.8,0) {$w$};
  \node at (-2,0) {$L$};
\end{scope}
  \end{tikzpicture}
  \caption{A configuration in the proof of Proposition~\ref{prop:elementary_extension_well-defined}.}
  \label{fig:configuration}
\end{figure}
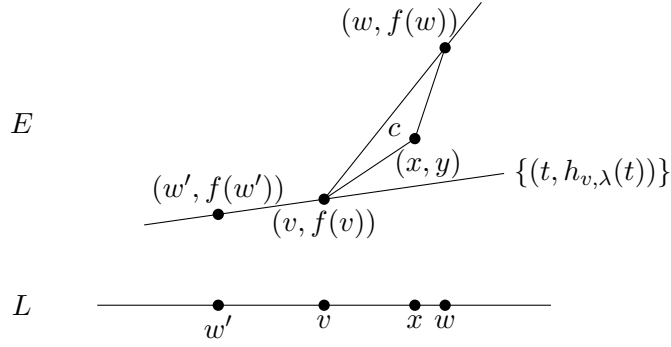

Assume first that $v$ is not in the boundary of $p_T(C)$ so there is a neighbor $w' \in p_T(C_E)$ of $v$ such that $v$ lies between $w$ and $w'$ (see Figure~\ref{fig:configuration}). By assumption $f_{vw'} + f_{vw} > 0$. If $f_{vw'}$ is the slope of a wall through $(v,f(v))$ so is the negative slope $f_{w'v}$. That is, the ray $t \mapsto (t,h_{v,f_{w'v}}(t))$ lies inside a wall. Since $f_{w'v} < f_{vw}$ it follows that $c$ lies above this ray and so $x$ is on the same side of $v$ in $T$ as $w$ and the slope $\lambda_v = (y - f(v))/d(v,x)$ is bounded below by $f_{w'v}$: $\lambda_v$ is the largest slope of a wall trough $(v,f(v))$ with $\lambda_v < f_{vw}$ (which exists by the previous discussion). The upshot is that $D_E$ is contained in the region above the ray $(p,h_{v,\lambda_v}(p))$ since $C$ and $c$ do and the region is convex (that region is $C_{L \cap p_T(C),h_{v,\lambda_v},\infty}$).

  If $u$ is the other endpoint of $M$ and it is also not an endpoint of $p_T(C)$ the same discussion applies leading to a slope $\lambda_u$ and a ray $(p,h_{u,\lambda_u}(p))$ such that $D_E$ lies in the region above this ray as well.
  
  For a last constraint on $D_E$ let $\mu > 0$ be minimal such that $\{(p,f(p)-\mu) : p \in M\}$ lies in a wall (parallel to the wall containing $\{(p,f(p)) : p \in M\}$).

  Now we define a function $h \colon p_T(C_E) \to \R$ by
  \[
    h(p) = \begin{cases}
      f(p) & p \in p_T(C_E) \setminus M\\
      \max\{h_{v,\lambda_v}(p),h_{u,\lambda_u}(p), f(p) - \mu\} & p \in M.
    \end{cases}
  \]
  Note that the definition only depends on $f|_M$. We claim that $h$ is convex. Convexity outside of $M$ follows from convexity of $f$, on the interior of $M$ it is clear by definition. In the endpoint $v$ of $M$ it follows from the fact that $f_{vw'} + f_{vw} > 0$ so that $f_{vw'} + \lambda_v \ge 0$ for every neighbor $w'$ of $v$ other than $w$ (the neighbor in $M$). In the endpoint $u$ an analogous argument applies.

  In summary we have that $C_{M,h,g}$ is a convex set that is described entirely in terms of what happens inside $p_T^{-1}(M)$, it contains $C_E$ and $c$ so it contains $D_E$. Thus $D_E$ consists of $C_E$ and the convex hull of $C_E \cup c$ in $C_{M,h,g}$. Thus $D_E \setminus C_E$ can be constructed in $p^{-1}(M)$ and depends only on $M$ and not on the choice of $E$ containing $M$.

  \medskip

  Now if $v$ is in the boundary of $p_T(C)$ it may be that there is no ray of slope $\lambda < f_{vw}$ contained in a wall, thus possibly $x$ is not in the interior of $M$, in fact, possibly not even in $M$. There are two possibilities how this can happen: $p_T^{-1}(v)$ is a wall, so $v$ is a vertex of the coarse structure and $x = v \in M$; or $p_T^{-1}(v)$ is not a wall, $v$ is a vertex of the fine structure but not of the coarse structure and $x \in \hat{M} \setminus M$. In either case one may define $h$ on $N \cup p_T(c)$ by
  \[
    h(p) = \begin{cases}
      f(p) & p \in N \setminus M\\
      \max\{h_{u,\lambda_u}(p),f(p)-\mu\} & p \in M \cup p_T(c)
    \end{cases}
  \]
  and apply the same reasoning. Note that convexity in $(x,y)$ is not an issue since $x$ lies in the boundary of $N \cup p_T(c)$. A subtlety is that if $x \in \hat{M} \setminus M$ then $g$ needs to be extended to $\hat{M} \setminus M$. This is done such that $g_{vx}$ is the maximal possible slope in $(v,g(x))$ so in particular $g_{vx} \ge g_{wv}$ meaning that $g_{vx} + g_{vw} \ge 0$ (keep in mind that the upper end of $C_E$ is the graph of $-g$).

  If both $u$ and $v$ lie in the boundary of $p_T(C)$ then $M = N$ and independence of the apartment is clear. We still define $h(p) = f(p) - \mu$ and observe that $D_E$ lies in the region above $h$. This concludes the verification that $D$ is well-defined by the condition that $D \cap E = D_E$ for every apartment $E = p_T^{-1}(L)$ with $L$ containing $M$. It is also the end of the part of the proof were we are working in the flat $E$ above $L$.

  Before turning to the proof of convexity, we draw a conclusion from the fact that $D_E$ lies in the region above $\{(p,f(p)-\mu) : p \in \hat{M}\}$. Let $v$ be an interior point of $M$, $w$ a neighbor, and $(v,y)$ and $(w,z)$ the points in the boundary of $D_E$. We assume that $(v,y)$ is a special vertex and that $z > y$. Then $(z-y)/d(v,w) \ge f_{vw}$.

  To see that $D$ is convex we need to show that for any two edges $e = [v,w]$ and $e' = [v,w']$ of $p_T(C)$ sharing a vertex $v$ the set $D \cap p_T^{-1}(e \cup e')$ is locally convex in $v$. Let $R$ be the component containing $e$. If $v$ is a leaf of $R$ or if $e \cup e'$ is constant slope (i.e.\ $f_{vw} + f_{vw'} = 0$) there is a maximal constant slope segment $M$ containing $e$ and a line $L$ containing $M$ and $e'$ so that $E = p_T^{-1}(L)$ is an apartment of the kind just discussed. As we have seen $D_E$ is convex in $E$ so in particular, it is locally convex in $v$. Note also that if $v$ is not a vertex of the coarse structure then any line containing $e$ also contains $e'$ so the same argument applies, so we can assume $v$ to be a vertex of the coarse structure. Since $R$ is admissible, the only remaining case is that there is an edge $e'' = [v,w'']$ so that both $e \cup e''$ and $e' \cup e''$ are constant slope and, in particular, $f_{vw} = f_{vw'} > 0$. In this case we use the observation from the last paragraph: if $(v,y)$, $(w,z)$ and $(w',z')$ are the points in the boundary of $D$ and if $(v,y)$ is special then $(w-v)/d(v,w)$ and $(w'-v)/d(w',v)$ are at least $f_{vw} > 0$ and thus their sum is positive. If $(v,y)$ is not special but $v$ is a vertex of the coarse structure, we observe that the only possible slope is $0$.
\end{proof}

As promised, the remaining kind of elementary extension is simple: Let $e$ be an edge of $C$ that is vertical, i.e.\ $p_T(e)$ is a single vertex $v$ (necessarily of the coarse structure). Let $c$ be some chamber not in $C$ containing $e$ and write $p_T(c) = [v,w]$. Then we let $S' = S \cup p_T(c)$ and extend $f$ and $g$ to $S'$ by taking $f_{vw}$ and $g_{vw}$ to be maximal among the possible slopes. Then $D = C_{S',f',g'}$ is the elementary extension of $C$ along $e$ by $c$. Since we took $f_{vw}$ maximal among the possible slopes $f_{vw} + f_{vw'} \ge 0$ is automatically satisfied for any neighbor $w$ of $v$ in $S$ which shows that $D$ is convex. We also keep in mind that $f'|_S = f$ and $g'|_S = g$. Note that the extension does not actually depend on $e$ so it is really an extension along $p_T^{-1}(v)$.

We come to our main result of the section.

\begin{proposition}\label{prop:elementary_extensions_exist}
  Let $C \subsetneq D \subseteq \Ssingular$ be convex subcomplexes and assume that $D$ contains a chamber. Then there is an elementary extension of $C$ contained in $D$.
\end{proposition}

In particular, it follows that the elementary extensions of $C$ are minimal among the convex complexes that properly contain $C$ and contain a chamber.

\begin{proof}
  Let $c$ be a chamber of $D$ that is not contained in $C$. It is easy to see that we may take $c$ to have an edge $e$ in common with $C$. If $e$ is vertical then there is an elementary extension by $c$ so we assume $e$ to be non-vertical. Write $C = C_{S,f,g}$. By changing orientation of $\R$ thus exchanging $f$ and $g$ if needed we may assume that $e$ is of the form $[(v,f(v)),(w,f(w))]$. If the component of constant slope $R$ of $e$ is admissible there is an extension $D$ of $C$ by $c$. If $L$ is a line in $T$ which contains a maximal segment of constant slope of $R$ and $E = p_T^{-1}(L)$ is the corresponding apartment of $\Ssingular$ we may take $c$ and $e$ inside $E$ and see that since $D$ contains $C$ and $c$ it contains the intersection of $E$ with the elementary extension. Since the difference of the elementary extension and $C$ is contained in the union of such apartments this shows that the extension is contained in $D$.

  So assume that $R$ is not admissible. Hence there are vertices $v,w,w'$ in $R$ with $w$ and $w'$ neighbors of $v$ as well as a neighbor $x$ of $v$ outside of $R$. Tracing along segments of constant slope it is not hard to see that the convex hull of $C$ and $c$ contains the chambers adjacent to $[(v,f(v)),(w,f(w))]$ and $[(v,f(v)),(w,f(w))]$. From this it follows that it contains the full link of $v$ in $p_T^{-1}([v,w] \cup [v,w'])$. But then it must also contain the chamber adjacent to $e' = [(v,f(v)),(x,f(x))]$.

  We replace $e$ by $e'$ which has bigger slope (in absolute value) and induct.
\end{proof}

We record another fact that we will need later on.

\begin{lemma}\label{lem:convex_extend_to_infinity}
  Let $C = C_{S,f,g} \subseteq \Ssingular$ be a convex subcomplex. There exists a minimal convex subcomplex $C_+$ that contains $C$ and for every $x \in T$ admits an $r_x$ such that $\{x\} \times [r_x,\infty) \subseteq C_+$. Writing $C_+ = C_{T,f',\infty}$ one has $f'|_S = f$. Symmetrically there is a least convex subcomplex $C_-$ that contains $C$ and for every $x \in T$ contains $\{x\} \times (-\infty,r_x]$ for $r_x$ sufficiently small and $C_- = C_{T,-\infty,g'}$ with $g'|_S = g$.
\end{lemma}

\begin{proof}
  If $D \supseteq C$ contains $\{x\} \times [r_x,\infty)$ for every $x$ and sufficiently large $r_x$ then visibly $D$ needs to contain all upward elementary extensions of any convex subcomplex. So let $\calC_1$ be the set of all complexes that arise from $C$ by iterated upward extension. This system is directed and so $C_1 \defeq \bigcup \calC_1$ is convex. Note that $C_1 = C_{\hat{S},f',\infty}$ where $\hat{S}$ is the least subtree of $T$ with respect to the coarse structure containing $S$, $f'|_S = f$ and the slope $f_{vw}$ is maximal for $v$ a leaf of $S$ and $w \in \hat{S} \setminus s$ a neighbor.

  Since the assumption on rays is satisfied for all $x \in T$, we have $p_T(D) = T$. So $D$ contains for every convex subcomplex $C'$ and every leaf $v$ of $p_T(C')$ with $p_T^{-1}(v) \cap C'$ non-singleton the extension along the vertical edges $p_T^{-1}(v)$. Thus let $\calC_2$ be the set of all complexes that arise from $C_1$ by extensions along vertical edges. It is again directed and $C_2 \defeq \bigcup \calC_2$ is convex. Since $C_2 \subseteq D$ and $C_2$ has the required property of containing $\{x\} \times [r_x,\infty)$ for every $x$, it follows that $C_2$ is the desired $C_+$. That $f'|_S = f$ follows from the fact that extensions along vertical edges never affect $f'|_S$.
\end{proof}

\section{Pro-uniform measures}\label{sec:pro-uniform_measures}

Our goal in this section is to introduce pro-uniform measures. The general ideas for these constructions were developped in \cite{BFL}, but we will need them in a much greater generality, which is why we develop a different framework.  The general idea of this construction is to define a notion of a uniform measure on the set of embeddings of a given simplicial complex in another. However, in our more general setup we will need to put additional restrictions on these sets of embeddings.  

To do so, it will be convenient to employ basic categorical language, especially (inverse) limits and (directed) colimits. The general constructions in a categorical framework are developped in Appendix~\ref{sec:limit_measures} below, to which we will refer. Recall that a \emph{directed set} is a partially ordered set $I$ in which any two elements $i,j$ have a common upper bound $k \ge i,j$; it may be thought of as a category with objects $I$ and a single morphism $i \to j$ whenever $i \le j$. A \emph{directed} (respectively \emph{inverse}) system in a category $\catname{C}$ consists of an object $Z_i$ for every $i \in I$ and a morphism $Z_i \to Z_j$ (respectively $Z_i \leftarrow Z_j)$ for every $i \le j$ such that the triangles associated to $i \le j \le k$ commute; thus it is a covariant respectively contravariant functor from $I$ to $C$. A \emph{cocone} of a directed system in $\catname{C}$ consists of an object $Z$ and morphisms $Z_i \to Z$ such that the triangles associated to $i \le j$ commute; similarly a \emph{cone} of an inverse system has morphisms $Z_i \leftarrow Z$. A \emph{colimit} of a directed system $(Z_i)_i$ is a cocone $Z$ that is universal in the sense that if $Z'$ is any other cocone then there is a unique morphism $Z \to Z'$ through which every $Z_i \to Z'$ factors; similarly a \emph{limit} of an inverse system $(Z_i)_i$ is a universal cone. Limits or colimits may or may not exist but if they exist they are unique up to unique isomorphism. This justifies to denote them by $\lim_i Z_i$ respectively $\colim_i Z_i$. A basic fact \cite[Theorem~3.4.11]{Riehl16} is that for an object $X$ of $\catname{C}$ the functor $\Hom_\catname{C}(-,X)$ preserves colimits, that is, if $(Z_i)_i$ is a directed system whose colimit exists then
\[\Hom_{\catname{C}}(\colim_i Z_i,X) = \lim_i \Hom_{\catname{C}}(Z_i,X).\]

As usual all our categories will be locally small, meaning that the collection $\Hom_\catname{C}(Z,Z')$ of morphisms from $Z$ to $Z'$ is a set. We will usually equip $\lim_i \Hom_\catname{C}(Z_i,X)$ with the limit topology obtained by giving each $\Hom_{\catname{C}}(Z_i,X)$ the discrete topology. The point of what follows is to also equip it with a limit measure.

To motivate our degree of generality we give two examples that illustrate some complementary aspects. 

\begin{example}\label{ex:prouni_motivation_tree}
  Let $T$ be a $(q_1+1,q_2+1)$-biregular tree and let $L$ be a biinfinite linear graph whose vertices alternately correspond to the two types $1$ and $2$ of vertices of $T$. We wish to define a measure on the space $\calL = \{\iota \colon L \hookrightarrow T \mid \iota \text{ preserves types}\}$ of embeddings of the line $L$ into $T$. To do so we consider the connected subgraphs of $L$ and type-preserving embeddings among them. These data form a category $\catname{L}$ and $L$ is the directed union of the system $(F)_{F}$ of its finite subgraphs, which is a colimit in $\catname{L}$.

  For each subcomplex $M \subseteq L$ we may consider the embeddings $\{\iota \colon M \to T \mid \iota \text{ preserves types}\}$. We can form a category $\catname{L}^X$ which consists of $\catname{L}$ with $T$ as an additional object and where $\Hom_{\catname{L}^X}(M,T)$ is the set of embeddings just considered. We want to equip the $\Hom_{\catname{L}^X}(M,T)$ with measures and start by declaring that if $M = \{v_1\}$ is a vertex of type $1$ then the measure should be counting measure. The important observation is that if $\alpha \in \Hom_{\catname{L}}(M,M')$ is an embedding of subcomplexes, the corresponding restriction morphism $\alpha^* \colon \Hom_{\catname{L}^X}(M',T) \to \Hom_{\catname{L}^X}(M,T)$ is surjective and has all fibers of the same size $[M',M]$ (a polynomial in $q_1$ and $q_2$ depending only on $M$ and $M'$). Thus in order for the maps to be measure-preserving we are forced to equip $\Hom_{\catname{L}^X}(M,T)$ with $1/[M,\{v_1\}]$ times counting measure. This consistently extends to measures on all $\Hom_{\catname{L}^X}(M,T)$, for instance if $v_2$ is a vertex of type $2$ then $\Hom_{\catname{L}^X}(\{v_2\},T)$ receives $(q_2+1)/(q_1+1)$ times counting measure.

  Finally, the fact that $L = \colim_F F$ is a direct limit implies that $\Hom_{\catname{L}^X}(L,T) = \lim_F \Hom_{\catname{L}^X}(F,T)$ is an inverse limit and can be equipped with the limit measure.
\end{example}

\begin{example}\label{ex:prouni_motivation_haar}
  Let $X$ be a countable set. Let $G < \Sym(X)$ be a permutation group and assume that every orbit of  the stabilizer $G_x$ are finite for every $x\in X$ . We may form a category $\catname{X}$  whose objects are non-empty subsets $M$ of $X$ and where
  \[
    \Hom_{\catname{X}}(M,N) = \left\{\iota \colon M \to N \mid \substack{\text{for }F \subseteq M \text{ finite there exists }g \in G\\\text{ with }\iota(x) = g(x)\text{ for }x \in F}\right\}.
  \]
  Note that all maps $M \to N, m \mapsto g(m)$ for $g(M) \subseteq N$ are morphisms and in particular so are all inclusions by taking $g = 1$.

  Note that $\Hom_{\catname{X}}(X,X)$ is the closure $\overline{G}$ of $G$ in the topology of pointwise of convergence, which is locally compact by the assumption that the orbits of $G_x$ be finite. The construction is so that $X$ is the colimit $\colim_F F$ of its finite subsets with inclusion and so $\overline{G} = \Hom_{\catname{X}}(X,X)$ is the inverse limit of the $\Hom_{\catname{X}}(F,X)$.

  If $E\subset F$ are finite subsets of $X$, then the restriction to $E$ is a map $\Hom_{\catname{X}}(F,X)\to\Hom_{\catname{X}}(E,X) $. Using the definition of $\catname{X}$ it is easy to see that this map is surjective, with preimages of constant size $[F,E]:=[G_F:G_E]$ (where $G_F$ and $G_E$ are the pointwise stabilizers of $F$ and $E$ respectively).

  In order to equip these with a measure we pick a basepoint $o \in X$ and decide that $\Hom_{\catname{X}}(\{o\},X)$ should carry counting measure. Our goal is to turn $(\Hom_{\catname{X}}(F,X))_F$ into an inverse system of measure-preserving maps.  
  If a finite set $F$ contains $o$, and if we want the restriction to be measure-preserving we are forced to equip $\Hom_{\catname{X}}(F,X)$ with $1/[F,\{o\}]$ times the counting measure.
  
   For a general finite set $E$ we may choose $F$ containing $E$ and $o$ (for instance $F = E \cup \{o\})$) and consistently define
  \[
    [E,\{o\}] = \frac{[F,\{o\}]}{[F,E]} = \frac{[G_o : G_F]}{[G_E : G_F]}
  \]
  and equip $\Hom_{\catname{X}}(E,X)$ with $1/[E,\{o\}]$ times counting measure. The limit measure on $\overline{G}$ is the Haar measure that gives measure $1$ to $\overline{G}_o$.

  This construction applies for instance to a residually finite group $G$ with $X = \bigsqcup_{N <_{fi} G} G/N$. In that case $\overline{G}$ is the profinite completion; and to $G$ a tdlc group with $X = G/C$ where $C < G$ is compact open and $\bigcap_{g \in G} C^g = \{1\}$.
\end{example}

\begin{remark}
  If we define the category $\catname{X'}$ to have
  \[
    \Hom_{\catname{X'}}(M,N) = \{g|_M^N \mid g\in G, g(M) \subseteq N\}
  \]
  then $G = \Hom_{\catname{X'}}(X,X)$ but there are some colimits we would like to have which do not exist: if $(F_i)_i$ is an ascending sequence of subcomplexes with $\bigcup_i F_i = X$ and $g \in \overline{G} \setminus G$ then $g|_F \colon F \to X$ is a morphism in $\catname{X'}$ but $g$ is not. Hence $X$ is not the colimit of the $F_i$.
\end{remark}

Example~\ref{ex:prouni_motivation_tree} is mostly covered by the discussion in \cite[Sections~1--3]{BFL} (except for the preservation of types). In contrast, Example~\ref{ex:prouni_motivation_haar} needs a more general setup in order to implement the constraint that only certain maps $F \to X$ are admissible. What we gain from this general setup is the ability to define a measure on $\Fsingular_i$ as a limit that is naturally compatible with the Haar measure on $M_i$.

\subsection{Counting extensions}

The general challenge we are considering in this section is the following. We are given two complexes $Z_1$ and $Z_2$ and an embedding $\alpha \colon Z_1 \to Z_2$ and we want to say that the number of extensions $\kappa$ of $\iota$ along $\alpha$ in the commutative diagram
\begin{equation}\label{eq:extension}
\begin{tikzcd}
Z_2 \arrow[r, "\kappa"]                         & X \\
Z_1 \arrow[ru, "\iota"'] \arrow[u, "\alpha"] &  
\end{tikzcd}
\end{equation}
does not depend on $\iota$, i.e.\ the number of $\kappa$ with $\kappa \circ \alpha = \iota$ is the same for all $\iota \colon Z_1 \to X$. This can be phrased as saying that the restriction map $\alpha^* \colon \kappa \mapsto \kappa \circ \alpha$ has all fibers of the same size: it is \emph{constant-to-one} in the sense that there is an $n \in \N_{>0}$ such that $\abs{(\alpha^*)^{-1}(\iota)} = n$ for every $\iota$. The embeddings $\alpha$, $\iota$, and $\kappa$ are always type preserving embeddings, as in Example~\ref{ex:prouni_motivation_tree}, but sometimes are required to satisfy additional constraints, as in Example~\ref{ex:prouni_motivation_haar}. These constraints are implemented by requiring that $\alpha$ be a morphism in a category $\catname{C}$ and $\iota$ and $\kappa$ are morphisms in larger category $\catname{C}^X$ that contains $\catname{C}$ and the object $X$. Then $\alpha \in \Hom_{\catname{C}}(Z_1,Z_2)$ and $\alpha^* \colon \Hom_{\catname{C}^X}(Z_2,X) \to \Hom_{\catname{C}^X}(Z_1,X)$ and our goal will generally be to show that for certain $\alpha$ the map $\alpha^*$ is constant-to-one. In the language of \cite{BFL} this roughly corresponds to saying that $X$ is $\catname{C}$-symmetric but we have the new aspect that $\Hom_{\catname{C}^X}(Z,X)$ might not be all (type-preserving) embeddings.

 We therefore say that (the object of $\catname{C}^X$) $X$ is \emph{$\catname{C}^X$-symmetric} if for every morphism $\alpha \colon Z_1 \to Z_2$ in $\catname{C}^X$ the map $\alpha^* \colon \Hom_{\catname{C}^X}(Z_2,X) \to \Hom_{\catname{C}^X}(Z_1,X)$ is surjective and if it has a finite fiber then it is constant-to-one.

Letters are generally chosen compatibly so that if $\mathscr{C}$ consists of embeddings $C \to X$ then the corresponding category is denoted $\catname{C}$. In particular $\mathscr{C} = \Hom_{\catname{C}^X}(C,X)$ is satisfied.

Throughout the section fix a type $i \in \{1,2\}$.
Concretely we consider the following. Let $\Flat$ be the category whose objects are non-empty convex subcomplexes of $\Sregular$ and whose morphisms are type-preserving embeddings. Let $\FlatEmb$ be the category whose objects are convex subcomplexes of $\Sregular$ and $X$ and whose morphisms are type-preserving embeddings. Similarly, let $\WallHat_i$ be the category whose objects are non-empty convex subcomplexes of $\Ssingular_i$ and whose morphisms are type-preserving embeddings. Let $\WallHatEmb_i$ be the category that additionally has the object $X$ and all type-preserving embeddings.

As in the construction of the singular flow, we are actually interested in a more restricted category. We define the category $\Wall_i$ as follows. Its objects are, as before, non-empty convex subcomplexes of $\Ssingular_i$. However, its morphisms are given by elements of $M_i$: if $m \in M_i$, $Z_1 \subseteq \Ssingular_i$ is a convex subcomplex and $Z_2 \subseteq \Ssingular_i$ is a subcomplex containing $m(Z_1)$ then $Z_1 \to Z_2, x \mapsto m(x)$ is a morphism of $\Wall_i$. The category $\WallEmb_i$ has the objects of $\Wall_i$ as well as $X$. For $Z\in \Wall_i$ we define $\Hom_{\WallEmb_i}(Z,X)$
as the set of 
restrictions of elements of $\Fsingular_i$ to $Z$. Thus a morphism in $\WallEmb_i$ is either a map $Z_1 \to Z_2$ coming from $m \in M_i$ as above or a map $\iota|_{Z_1} \colon Z_1 \to X$ where $\iota \in \Fsingular_i$. Note that $\WallEmb_i$ is closed under composition as $\Fsingular_i$ is $M_i$-invariant. 

We will not always want to restrict to finite complexes and therefore make the following definition. A map $\alpha \colon Z_1 \to Z_2$ of non-empty connected affine cell complexes is \emph{cofinite} if $Z_2$ is the (combinatorial) convex hull of $\alpha(Z_1)$ and finitely many points. So an extension $\kappa \colon Z_2 \to X$ of $\iota \colon Z_1 \to X$ is uniquely determined by $\iota$ and the value of the finitely many points for which there are finitely many choices by local finiteness of $X$.

With these preparations in place we start with results, first for the easier case of $\Flat$ to see the general proof structure.

\begin{lemma}\label{lem:flat_symmetric}
  If $\alpha \colon Z_1 \to Z_2$ is a morphism in $\Flat$ then the restriction morphism $\alpha^* \colon \Hom_\Flat(Z_1,X) \leftarrow \Hom_\Flat(Z_2,X)$ is surjective. If $\alpha$ is cofinite then $\alpha^*$ is constant-to-one.
\end{lemma}

\begin{proof}
  Suppose first that $\alpha$ is cofinite so that $\alpha^*$ is finite-to-one.

  Upon applying an isomorphism of $\Sregular$ we may assume that $\alpha$ is an inclusion.
  Assume first that ${Z_1}$ contains a chamber. Let $\calH_{Z_1}$ and $\calH_{Z_2}$ respectively be the set of half-appartments of $\Sregular$ that contain ${Z_1}$ and ${Z_2}$. Since the space of half-apartments is discrete and $\alpha$ is cofinite the set $\calH_{Z_2} \setminus \calH_{Z_1}$ is finite. So using induction we may assume that it is singleton, $\calH_{Z_2} \setminus \calH_{Z_1} = \{H\}$. Then there are a chamber $c$ of ${Z_1}$ and a chamber $d$ of ${Z_2}$ not in ${Z_1}$ that share an edge $e$ in $\partial H$. Let $i$ be the cotype of $e$. Now if $\iota \colon {Z_1} \to X$ is any embedding there are visibly $q_i$ ways to extend it to an embedding of $\kappa' \colon {Z_1} \cup c \to X$. Further, $\kappa'$ extends to at most one embedding $\kappa \colon {Z_2} \to X$ since ${Z_2} = \conv({Z_1} \cup c)$. That there is indeed an extension follows from the fact that there is an apartment containing the image of $\kappa'$ by \cite[Theorem~11.53]{AbramenkoBrown08}.

  It remains to treat the lower-dimensional cases. If ${Z_2}$ contains a chamber but ${Z_1}$ does not then there is a chamber $d$ of ${Z_2}$ not in ${Z_1}$ containing an edge $e$ of ${Z_1} \cap \partial H$. If $e$ has cotype $i$ an embedding $\iota \colon {Z_1} \to X$ extends in $q_i + 1$ ways to an embedding $\kappa' \colon {Z_1} \cup c \to X$. The rest of the argument is as before.

  Finally assume that ${Z_1} = [u,v]$ is a segment and ${Z_2} = [u,w]$ is a longer segment. Again using induction we assume that $[v,w]$ is an edge where $v$ is a vertex of type $i$ and $w$ of type $j$. Let $\{i,j,k\} = \{0,1,2\}$. If $\iota \colon {Z_1} \to X$ is an embedding then $\lk \alpha(v)$ is a spherical building whose type only depends on $i$ (and the type of $X$) and that has order $(q_j,q_k)$. If $u=v$ then the number of extensions $\kappa \colon {Z_2} \to X$ is the number of vertices of type $j$ in $\lk \alpha(v)$. If $u \ne v$ then the number of extensions is the number of vertices of type $j$ opposite a fixed vertex of the opposite type. Both of these numbers only depend on the type and order of $\lk \alpha(v)$.

  If $\alpha$ is not necessarily cofinite then $Z_2$ is a colimit of cofinite supercomplexes of $Z_1$. Consequently $\Hom_\Flat(Z_2,X)$ is a limit of sets that surject onto $\Hom_\Flat(Z_1,X)$ so $\alpha^*$ is surjective.
\end{proof}

Our main interest is in $\Wall_i$ of course, but before turning to it we treat $\WallHat_i$.

\begin{lemma}\label{lem:wallhat_symmetric}
  For $\alpha \colon Z_1 \to Z_2$ in $\WallHat_i$ the restriction map \[\alpha^* \colon \Hom_{\WallHatEmb_i}(Z_1,X) \leftarrow \Hom_{\WallHatEmb_i}(Z_2,X)\] is surjective. If $\alpha$ is cofinite then $\alpha^*$ is constant-to-one.
\end{lemma}

\begin{proof}
  Let $\alpha \in \Hom_{\WallHat_i}(Z_1,Z_2)$. There is no loss in assuming that $\alpha$ is the inclusion of a subcomplex $Z_1$ into $Z_2$. Assume first that $\alpha$ is cofinite and that $Z_1$ contains a chamber. Then by Proposition~\ref{prop:elementary_extensions_exist} there is an elementary extension of $Z_1$ contained in $Z_2$ and by induction it suffices to prove the case where $Z_2$ is that elementary extension. So we write $Z_1 = C_{S,f,g}$ and assume that $Z_2$ is a downward extension along an admissible component $R$ of $f$ (or symmetrically an upward extension along an admissible component $R$ of $g$) or an extension along a vertical edge.

  We consider the case of a downward extension first. If the slope of $R$ is non-zero let $v \in R$ be the unique leaf of $R$ in which $f|_R$ attains its minimum, if the slope is zero let $v$ be any vertex of $R$. Let $w \in R$ be such that $e = [(v,f(v)),(w,f(w))]$ is an edge that contains $(v,f(v))$ (unique in the case of non-zero slope). Let $d$ be the chamber of $Z_1$ containing $e$ and let $c$ be the other chamber of $Y_i$ adjacent to $e$. Then for every point $x$ of $T$ there is a biinfinite line $L$ that contains $e$, $x$ and a maximal segment of constant slope of $R$ and thus for every point $p = (x,y) \in Z$ there is an apartment $\Sregular = \pi_T^{-1}(L)$ that contains $c$, $d$, $p$ and is such that $Z_2 \cap \Sregular$ is the convex hull of $Z_1$ and $c$ in $\Sregular$, cf.\ Proposition~\ref{prop:elementary_extension_well-defined}. If $p$ is not in $R$ this is clear and if $p$ is in $R$ it follows from the choice of $e$. 

  Let $\iota \colon Z_1 \to X$ be a type-preserving embedding and let $k$ be the cotype of $e$ which is that of $\iota(e)$. Then $\iota(e)$ is contained in $q_k+1$ chambers, one of which is $\iota(d)$. We claim that for each choice $\bar{c}$ of one the other $q_k$ chambers there is a unique extension $\kappa \colon Z_2 \to X$ with $\kappa|_{Z_1} = \iota$ and $\kappa(c) = \bar{c}$. Uniqueness follows from the fact that $Z_2$ is the convex hull of $Z_1$ and $c$ so we only need to prove existence.

  Let $p \in Z_2$ be arbitrary. By the previous discussion there is an apartment $\Sregular$ that contains $c$, $d$ and $p$ and is such that $Z_2$ meets $\Sregular$ in the convex hull of $Z_1$ and $c$ in $\Sregular$. By \cite[Theorem~11.53]{AbramenkoBrown08} there is an apartment $\overline{\Sregular}$ of $X$ that contains $\iota(\Sregular \cap Z_1)$ and $c$. We take $\kappa|_{\Sregular \cap Z_2}$ to be the isomorphism from the convex hull of $Z_1 \cap \Sregular$ and $c$ to the convex hull of $\iota(Z_1) \cap \overline{\Sregular}$ and $\bar{c}$ extending the type-presering isomorphism $c \to \bar{c}$. Since every point $p \in Z_2$ is contained in such an apartment this defines $\kappa$ on all of $Z_2$. It is well-defined since the intersection of two apartments is convex in both of them.
  
  Now consider an extension along a vertical edge $e = [(v,y_1),(v,y_2)]$ where $v \in S$ is a leaf. In fact, let $R = p_T^{-1}(v)$ be the full segment of $Z_1$ that contains $e$. Let $d_1,\ldots,d_\ell$ be the chamber of $Z_1$ that contain $e$ and let $c$ be the chamber by which $Z_1$ is extended. For every $p = (x,y) \in Z_2$ there is an apartment $\Sregular$ of $\Ssingular_i$ that contains $p$, $R$ and $c$ and is such that the intersection of $\Sregular$ and $Z_2$ is the convex hull in $\Sregular$ of $\Sregular \cap Z_1$ and $c$: just take $\Sregular = p_T^{-1}(L)$ where $L$ contains $x$ and $p_T(c)$. If $e$ is of cotype $k$ then $\iota(e)$ is contained in $q_k+1$ chamber, $\ell$ of which are $\iota(d_1),\ldots,\iota(d_\ell)$. We claim that for any choice $\bar{c}$ of one of the remaining $q_k+1 - \ell$ chambers there is a unique extension $\kappa \colon Z_2 \to X$ with $\kappa|_{Z_1} = \iota$ and $\kappa(c) = \bar{c}$. The reasoning is as before: for an apartment $\Sregular$ that contains $R$ and $c$ let $\overline{\Sregular}$ contain $\iota(Z_1 \cap \Sregular)$ and $c$ and define $\kappa$ on $\Sregular \cap Z_2$ by the identification of $\Sregular$ and $\overline{\Sregular}$ that takes $c$ to $\bar{c}$.

  If $Z_1$ does not contain a chamber but $Z_2$ does then no change is necessary to the extension along a vertical edge. In the downward extension case the only change is that there are $q_k+1$ choices for $\bar{c}$ rather than $q_k$.

  If $Z_2$ contains no chamber then by Remark~\ref{rem:convex_without_chamber} there are three possibilities. Two of them are covered by Lemma~\ref{lem:flat_symmetric} and the last one uses the same local argument.

  If $\alpha$ is not cofinite there is an infinite ascending sequence of cofinite extensions of $Z_1$ whose union is $Z_2$. The extensions of $\iota$ to these intermediate complexes are the restrictions of an extension to $Z_2$.
\end{proof}

It is clear from the definition of $\WallHatEmb_i$ that the union of an ascending sequence of complexes is its colimit. In $\WallEmb_i$ this is not immediately clear (cf. the remark following Example~\ref{ex:prouni_motivation_haar}), but the following lemma asserts that it is still true. So in practice colimits in $\WallEmb_i$ can be computed in $\WallHatEmb_i$. The categorically inclined reader may use (the proof of) \cite[Theorem~1.5, Corollary~1.7]{AdamekRosicky94} to promote the statement to say that $\WallEmb_i$ has all $\omega$-small filtered colimits and that the inclusion $\WallEmb_i \to \WallHatEmb_i$ preserves $\omega$-small filtered colimits.

\begin{lemma}\label{lem:colimits_in_wallemb}
  If $(Z_k)_{k \in \N}$ is a directed system in $\WallEmb_i$ then its colimit exists in $\WallEmb_i$ and is a colimit in $\WallHatEmb_i$.
\end{lemma}

\begin{proof}
  For each $k > 0$ there is by definition an $g_k \in M_i$ such that $Z_{k-1} \to Z_{k}$ is $(g_k)|_{Z_{k-1}}^{Z_{k}}$, in particular $Z_{k-1} \subseteq g_k^{-1}(Z_{k})$. Putting $m_k = g_1^{-1} \cdots g_k^{-1}$ and $Z_k' = m_k \cdot Z_k$ it follows that $Z_k' \subseteq Z_\ell'$. Let $Z = \bigcup_k Z_k'$. It comes with morphisms $Z_k \to Z$ represented by $m_k$ such that $(m_k \cdot g_k)|_{Z_{k-1}} = m_{k-1}|_{Z_{k-1}}$ by construction. Let $Z'$ is another subcomplex with morphisms $Z_k \to Z'$ represented by $n_k \in \Aut(Y)$ in such a way that $(n_k \cdot g_k)|_{Z_{k-1}} = n_{k-1}|_{Z_{k-1}}$. Then $(n_\ell \cdot m_\ell^{-1})|_{Z_k}$ is the same for all $\ell \ge k$: indeed $(m_\ell^{-1} \cdot m_k)|_{Z_k} = (g_\ell \cdots g_{k+1})|_{Z_k} = (n_\ell^{-1} \cdot n_k)|_{Z_k}$. Thus
  \[
    U_k \defeq U_{Z_k}(n_km_k^{-1}) = \{m \in \Aut(Y) \mid m|_{Z_k} = n_km_k^{-1}\}
  \]
  is a descending sequence of compact (open) sets so there is an $m \in U \defeq \bigcap_k U_k$ such that $m|_{Z_k} = n_k \cdot m_k^{-1}$ thus defining a morphism $Z \to Z'$ that is compatible with the $m_k$ and $n_k$. If $m' \in U$ represents another compatible morphism $Z \to Z'$ then $m' \in U$ and so $m|_Z = m'|_Z$ is the same morphism. This shows that $m|_Z$ is the colimit in $\WallHatEmb_i$.

  Now the point is that if the $n_k$ lie in $M_i$ then since $M_i < \Aut(Y)$ is closed $U_k \cap M_i$ is compact as well and so $U \cap M_i$ is non-empty, so $m$ may be taken in $M_i$.
\end{proof}

\begin{proposition}\label{prop:wall_symmetric}
  For $\alpha \colon Z_1 \to Z_2$ in $\Wall_i$ the restriction map \[\alpha^* \colon \Hom_{\WallEmb_i}(Z_1,X) \leftarrow \Hom_{\WallEmb_i}(Z_2,X)\] is surjective. If $\alpha$ is cofinite then $\alpha^*$ is constant-to-one.
\end{proposition}

\begin{proof}
  First assume that $\alpha$ is cofinite. Then $\alpha_*$ is finite-to one.

  Upon applying an element $m \in M_i$ we may assume that $\alpha$ is an inclusion and we first assume that $Z_2$ is minimal among the proper convex supercomplexes of $Z_1$. Let us write $Z_1 = C_{S,f,g}$, $Z_2 = C_{S',f',g'}$ and let $\overline{S}$ be the least complex with respect to the coarse structure that contains $S$. Then combining Proposition~\ref{prop:elementary_extensions_exist} and Remark~\ref{rem:convex_without_chamber} there are three possibilities: $Z_2$ may be an upward extension in the sense that $f'|_S = f$ but $g'|_S \ne g$, it may be a downward extension in the sense that $g'|_S = g$ but $f'|_S \ne f$ or it may be a sideways extension with $f'|_S = f$, $g'|_S = g$ but $S'$ strictly containing $\overline{S}$ (what would be an extension by a vertical edge if $Z_2$ contains a chamber).

  In each case we want to replace $Z_1$ and $Z_2$ by larger complexes $\hat{Z}_1$ and $\hat{Z}_2$ as follows (see Lemma~\ref{lem:convex_extend_to_infinity}).
  In the first case $\hat{Z}_1$ is the least convex complex $(Z_1)_-$ that contains $Z_1$ and for every $x \in T$ some ray $\{x\} \times (-\infty,r_x]$  and $\hat{Z}_2$ is $(Z_2)_-$. In the second case we analogously take $\hat{Z}_1 = (Z_1)_+$ and $\hat{Z}_2 = (Z_2)_+$. In the third case we take $\hat{Z}_1 = p_T^{-1}(p_T(Z_1))$ and $\hat{Z}_2 = p_T^{-1}(p_T(Z_2))$ (which are $((Z_1)_+)_-$ and $((Z_2)_+)_-$ respectively). In all cases it is clear that an extension of $\iota|_{\hat{Z}_1}$ to $\hat{Z}_2$ gives rise to an extension of $\iota|_{Z_1}$ to $Z_2$ and we want to show that there are as many such extensions as possible.


  For upward and downward extensions the proof of Lemma~\ref{lem:wallhat_symmetric} asserts that there are as many extensions from $\iota|_{Z_1}$ to $Z_2$ in $\WallHatEmb_i$ as there are from $\iota|_{\hat{Z}_1}$ to $\hat{Z}_2$. The lemma further asserts that an extension to $\hat{Z}_2$ can be further extended to an extension $\kappa$ on all of $\Ssingular_i$. The point is now that since $\iota$ and $\kappa$ coincide on $\hat{Z}_1$ they are in relation $\sim_+$ or $\sim_-$ and so if $\iota \in \Fsingular_i$ then $\kappa \in \Fsingular_i$ as well showing that $\kappa|_{Z_2}$ is a morphism in $\WallEmb_i$.

  Sideways extensions are understood in terms of $M_i$: since $\hat{Z}_1$ contains a line parallel to $\ell_i(\R)$, the embedding $\iota|_{\hat{Z}_1}$ determines a pair of endpoints $(\iota(\ell_i(\infty)), \iota(\ell_i(-\infty))) \in \Deltaopv[i]$ and any extension $\kappa$ will have the same endpoints. Thus by Proposition~\ref{prop:restricted_wall_trees} $M_i$ acts transitively on the possible extensions. Identifying 
\begin{align*}
  \iota|_{\hat{Z}_1}&& \text{with} && \Stab_{M_i}(\iota_{\hat{Z}_1}) = U^{M_i}_{\hat{Z}_1}(\iota) = \{m \in M_i \mid (m.\iota)|_{\hat{Z}_1} = \iota|_{\hat{Z}_1}\}
\end{align*}
and any extension
\begin{align*}
  \kappa|_{\hat{Z}_2} &&\text{with} &&\Stab_{M_i}(\kappa|_{\hat{Z}_2}) = U^{M_i}_{\hat{Z}_2}(\kappa) = \{m \in M_i \mid (m.\kappa)|_{\hat{Z}_2} = \kappa|_{\hat{Z}_2}\}
\end{align*}
gives an identification of the possible extensions $\kappa$ of $\iota$ with the coset space $\Stab_{M_i}(\iota|_{\hat{Z}_2})/\Stab_{M_i}(\iota|_{\hat{Z}_1})$. Note that both groups are compact open so the coset space is finite and the number of extensions is the index.

Finally, if $Z_2$ is not minimal and more generally if $\alpha$ is not cofinite there is an infinite ascending sequence $(Z_{2-1/n})_{n \in \N_{>0}}$ of cofinite supercomplexes of $Z_1$ whose union is $Z_2$ (taking $\alpha$ to be an inclusion for simplicity). Lemma~\ref{lem:colimits_in_wallemb} asserts that $Z_2$ is in fact the colimit of this sequence sequence. The proof so far shows that the inverse system $(\Hom_{\WallEmb_i}(Z_{2-1/n},X))_n$ consists of surjections and it follows that $\Hom_{\WallEmb_i}(Z_2,X) = \lim_n \Hom_{\WallEmb_i}(Z_{2-1/n},X)$ surjects onto $\Hom_{\WallEmb_i}(Z_1,X)$.
\end{proof}

\begin{lemma}\label{lem:flat_contained_in_singular}
  The category $\Flat$ is contained in $\Wall_i$, and the category $\FlatEmb$ is contained in $\WallEmb_i$. That is, for every complex $Z \subseteq \Sregular$ and every type-preserving embedding $\iota_0 \colon Z \to X$ there is a $\iota\in\Fsingular_i$ such that $\iota|_Z=\iota_0$.
\end{lemma}

\begin{proof}
  Using \cite[Theorem~11.53]{AbramenkoBrown08} we may assume that $Z = \Sregular$. Let $\xi = \iota_0(\ell_i(\infty))$ and $\xi' = \iota(\ell_i(\infty))$. By Proposition~\ref{prop:restricted_wall_trees} there is $\iota_1\in\Fsingular_i$ with $\sigma_+\times\sigma_-(\iota)=(\xi,\xi')$, and the $M_i$-orbit of $\iota_1$ is in $\Fsingular_i$. 
  Furthermore by Lemma~\ref{lem:2-transitive}, the action of $M_i$ is 2-transitive on $\partial_\infty T_i$, so that there exists $\iota_2$ in the $M_i$-orbit of $\iota_1$ such that $\iota_2(\Sregular)=\iota_0(\Sregular)$, and in fact $\iota_2$ and $\iota_0$ coincide on the boundary $\partial \Sregular$, and therefore differ by an element of $A$.
  Applying this element of $A$ to $\iota_2$,  we find $\iota_3$ such that $\iota_3|_{\Sregular} = \iota$. 
\end{proof}

\begin{lemma}\label{lem:morphism_independence}
  For cofinite $\alpha \colon Z_1 \to Z_2$ in $\Wall_i$ the fiber size of the restriction morphism $\alpha_* \colon \Hom_{\WallEmb_i}(Z_1,X) \leftarrow \Hom_{\WallEmb_i}(Z_2,X)$ only depends on (the isomorphism type in $\Wall_i$ of) $Z_1$ and $Z_2$, not on $\alpha$.
\end{lemma}

\begin{proof}
  Let $\Wall_{i,\text{cofin}}$ be the subcategory of $\Wall_{i,\text{cofin}}$ having all objects but only cofinite morphisms. Let $\WallEmb_{i,\text{cofin}}$ be the subcategory of $\WallEmb_{i,\text{cofin}}$ that has all objects, the morphisms of $\Wall_{i,\text{cofin}}$ and all embeddings to $X$. We apply Lemma~\ref{lem:comparable} to $\Wall_{i,\text{cofin}}$ and the functor $\Hom_{\WallEmb_{i,\text{cofin}}}(-,X)$ and need to show that any two morphisms $Z_1 \to Z_2$ in $\Wall_{i,\text{cofin}}$ are comparable.

  Let $V_i \in \Wall_i$ be a complex consisting of a single vertex of type $i$. Since for every complex $Z \in \Wall_i$ there is a morphism $V_i \to Z$ for some $i$, it suffices to show that any two morphisms $\iota,\kappa \colon V_i \to Z$ are comparable.

  Let $E_i$ be a complex that is the convex hull of two vertices of type $i$ at minimal distance. Note that $E_i$ admits a type-preserving (reflection) symmetry $\kappa_i$ that exchanges the two vertices of type $i$, i.e.\ the two morphisms $V_i \to E_i$ (it exists in $\Wall_i$ by Lemma~\ref{lem:flat_contained_in_singular}). Now the proof follows in spirit Example~\ref{ex:unicomparable}. If $Z \in \Wall_i$ is arbitrary it may not be the case that any two vertices of type $i$ are connected by a sequence of vertices of type $i$ where any two consecutive ones lie in a common copy of $E_i$. However, there is a cofinite supercomplex $Z' \in \Wall_i$ with an embedding $\omega \colon Z \to Z' \in \Wall_i$ such that any two vertices of type $i$ in $Z$ are connected by such a sequence in $Z'$. By \eqref{item:factor} in the definition of comparability is suffices to show that two morphisms $\iota,\kappa \colon V_i \to Z'$ factoring through $Z$ are comparable. If both $\iota$ and $\kappa$ factor through $E_i$ and differ by $\kappa_i$ then they are comparable. The general case follows by transitivity.
\end{proof}

We end the paragraph by showing that restricting the space we embed to $\Ssingular_i$ (hence considering the space 
 $\hat{\Wall}_i$) cannot be avoided. A first naive idea would be to consider the category $\catname{C}$ of all subcomplexes of $X$ and all type-preserving embeddings. But then it
 is easy to see that if $X$ was $\catname{C}^X$-symmetric then the stabilizer of an edge in $\Aut(X)$ would need to act as the symmetric group on the chambers containing the edge (the chosen permutation on the star $S$ of an edge regarded as an embedding $S \to X$ would need to extend to an embedding $X \to X$). Even in Bruhat--Tits buildings this action is $\PGL_2(q)$ rather than $\Sym(q+1)$ and in exotic buildings it will typically be trivial or very small.

 Knowing the situation for $\tilde{A}_2$-buildings, the next best thing one might hope for is to consider embeddings of products (for example products of trees). For example we might ask for $X$ to be $\catname{C}^X$-symmetric where $\catname{C}$ consists of all complexes that metrically decompose as a direct product and type-preserving embeddings among them. 
 
 The following example illustrates that that is not true either by involving a product in which neither factor is a segment. Thus the correct class of complexes to consider consists of products where one of the factors is a segment, meaning that the complexes are subcomplexes of $\Ssingular_1$ or $\Ssingular_2$.

\begin{figure}[htb]
\includegraphics[page=3,width=.3\textwidth]{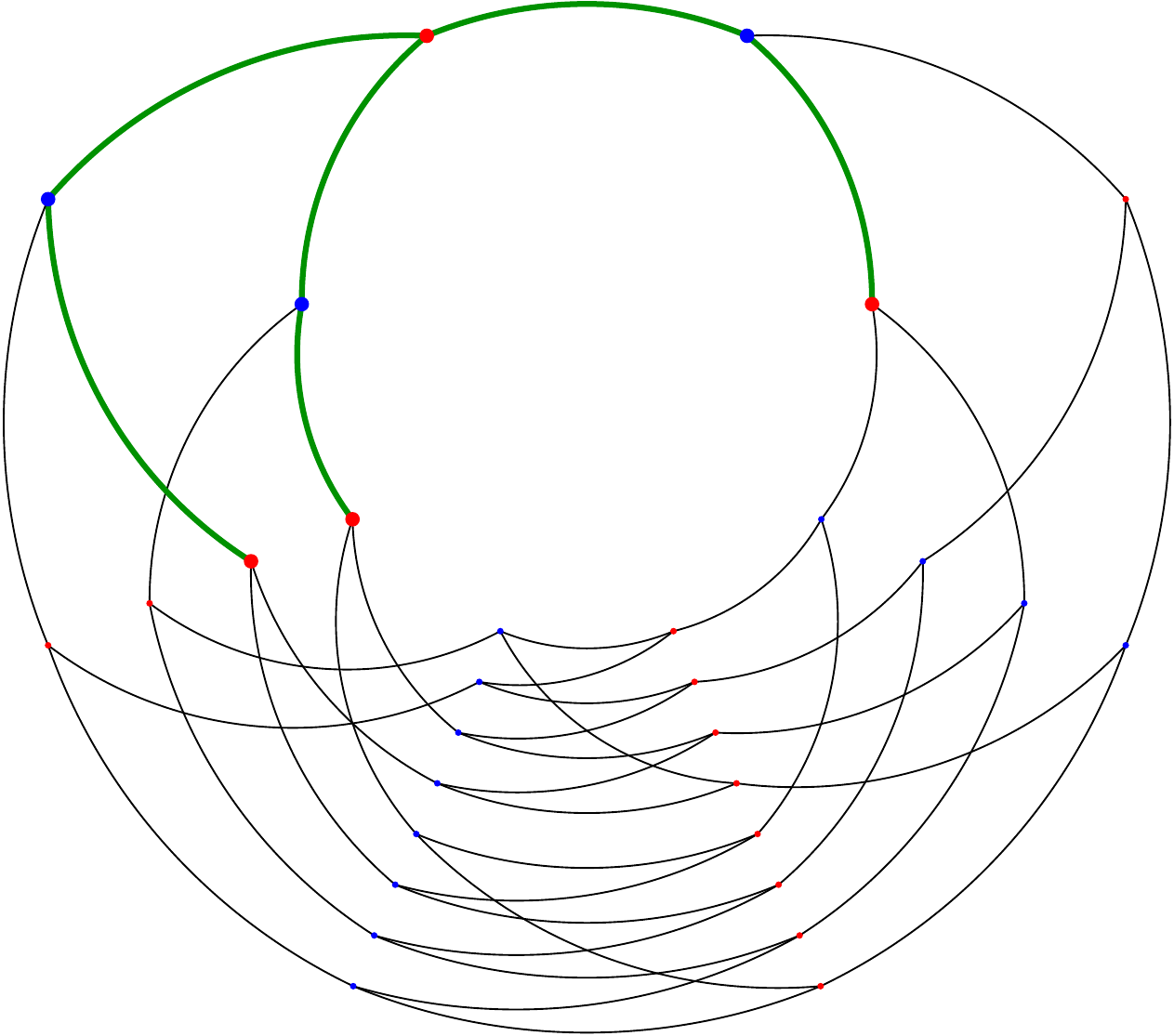}\hfill
\includegraphics[page=2,width=.3\textwidth]{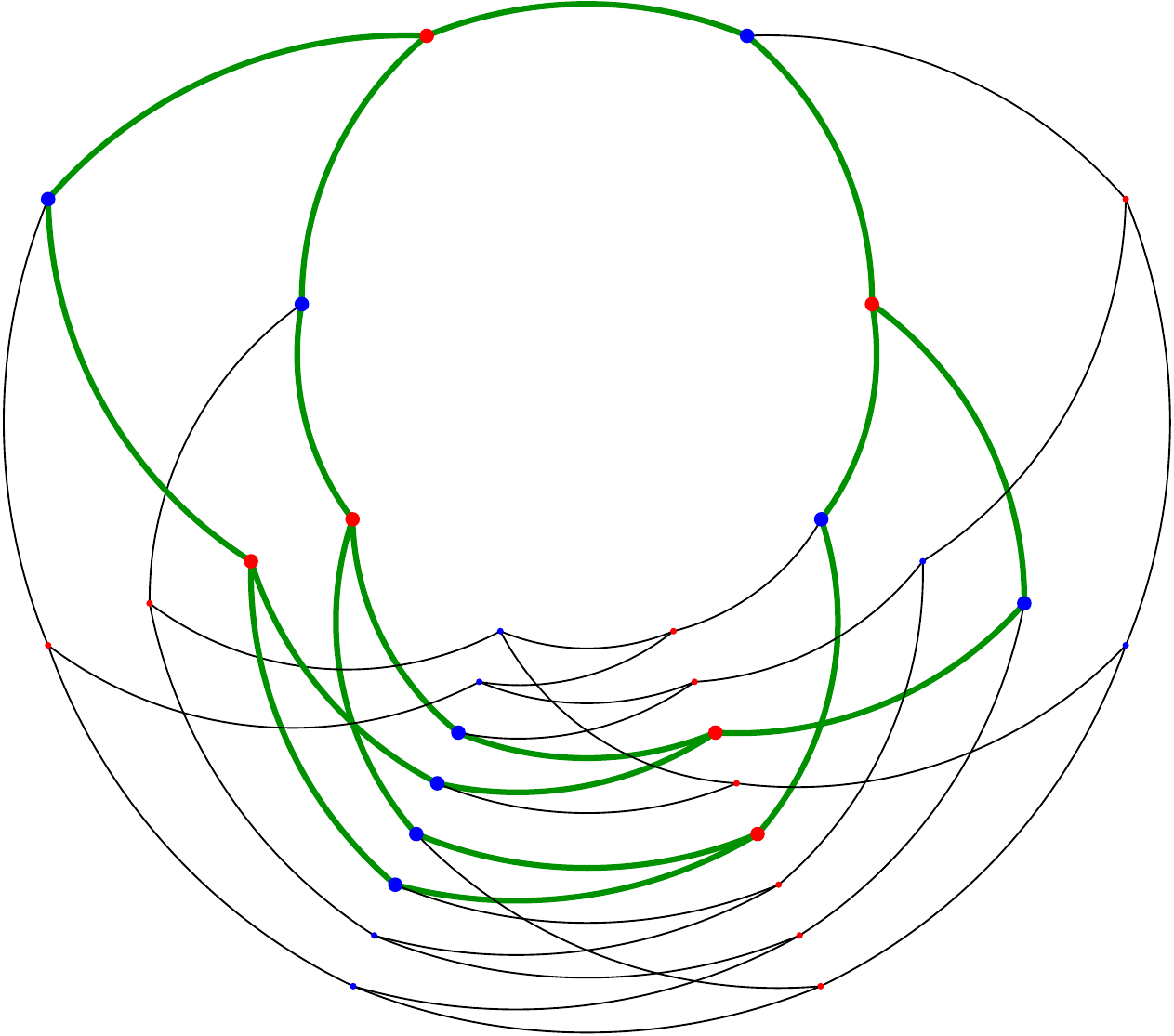}\hfill
\includegraphics[page=4,width=.3\textwidth]{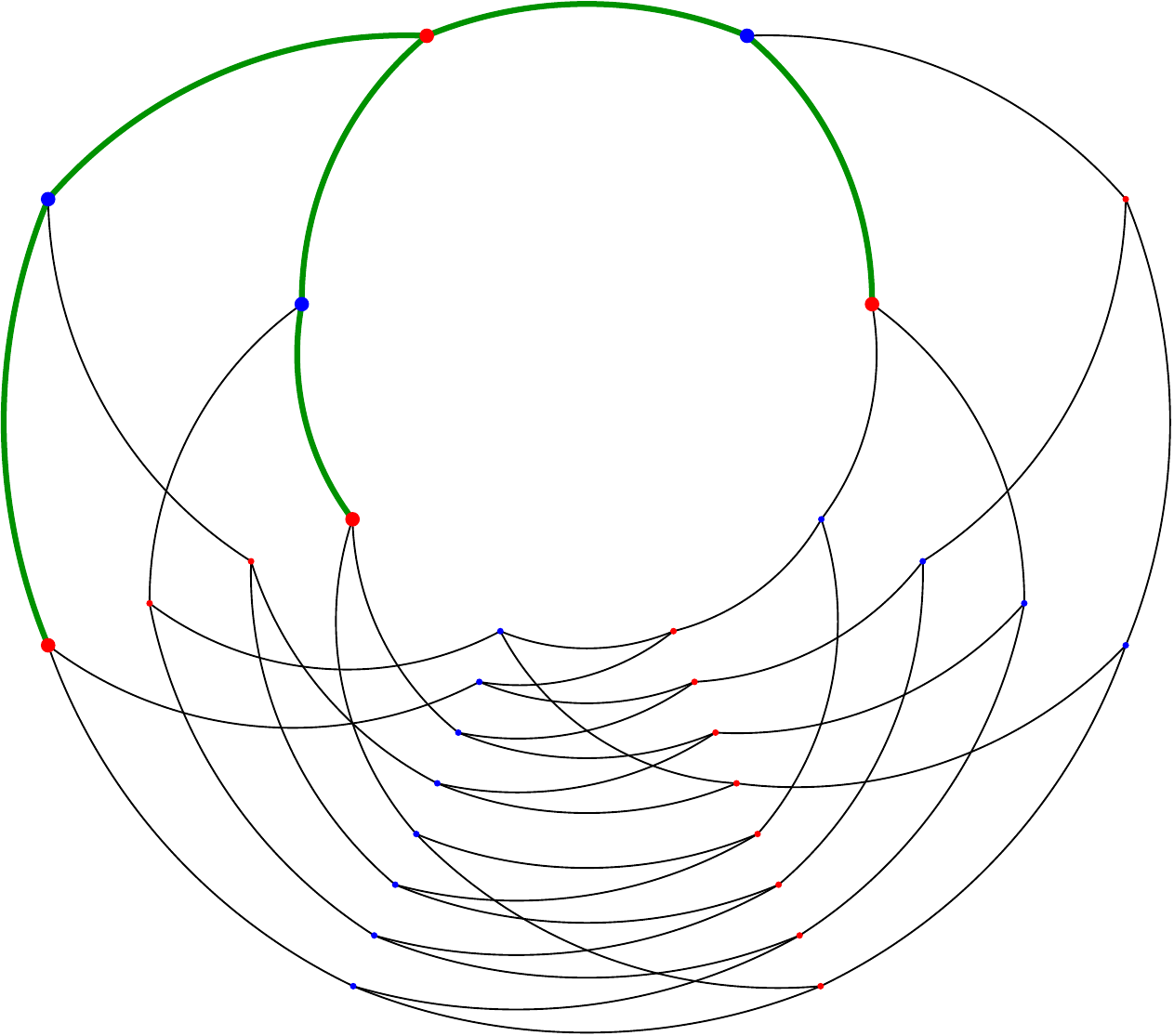}\\
\hfill $Y$\hfill\hfill $Z$ \hfill\hfill $Y'$ \hfill{}
\caption{Subcomplexes of the building of type $C_2$ over $F_2$.}
\label{fig:c2_example}
\end{figure}

\begin{example}\label{ex:not_product-symmetric}
  Figure~\ref{fig:c2_example} shows a building $X$ of type $C_2$ over $\F_2$ together with three convex subcomplexes $Y \subseteq Z$ and $Y'$. Taking the cone, the complexes can be regarded as subcomplexes of the star of a vertex in any building of type $\tilde{C}_2$ of thickness $3$. Note that $Z$ is a non-thick building that is a subdivision of the join of two $3$-element sets (a building of type $A_1 \times A_1$ over $F_2$). For $\iota \colon Y \to Z$ the inclusion $\alpha \colon Y \to X$ extends to the inclusion $Z \to X$. However, the complexes $Y$ and $Y'$ are isomorphic in an obvious way and the composition $\alpha' \colon Y \cong Y' \to X$ does not extend along $\iota$: it is elementary to verify that the convex hull of $Y'$ extended by the two edges out of the rightmost red vertex is all of $X$. Thus the building cannot be symmetric with respect to $Y$, $Y'$, $Z$, $X$ and all type-preserving embeddings among them.
\end{example}

\subsection{The measures on $\Fregular$ and $\Fsingular$}

Using Proposition~\ref{prop:wall_symmetric} the method from Appendix~\ref{sec:limit_measures} to construct measures becomes available.

Let $(Z)_{Z \in \Wall_i \text{ finite}}$ in $\Wall_i$ be the directed system of all non-empty finite convex subcomplexes of $\Ssingular_i$ and inclusions between them. Then $\Ssingular_i = \colim_Z Z$. By Proposition~\ref{prop:wall_symmetric}, $(\Hom_{\WallEmb_i}(Z,X))_Z$ is an inverse system of sets and constant-to-one maps. Recall that $0 \in \Sregular \subseteq \Ssingular_i$ is a special vertex. According to Proposition~\ref{prop:cto_limit_measure} we can equip $\Hom_{\WallEmb_i}(\{0\},X)$ with the counting measure, each $\Hom_{\WallEmb_i}(Z,X)$ with the relatively uniform measure and \[\Fsingular_i = \Hom_{\WallEmb_i}(\Ssingular_i,X) = \Hom_{\WallEmb_i}(\colim_Z Z, X) = \lim_Z \Hom_{\WallEmb_i}(Z,X)\] with the limit measure which we denote by $\mu_\Fsingular$.

Similarly taking the directed system $(Z)_{Z \in \Flat\text{ finite}}$ in $\Flat$ of all non-empty finite convex subcomplexes of $\Sregular$ and inclusions between them, we obtain the limit measure $\mu_\Fregular$ on
\[
  \Fregular = \Hom_{\FlatEmb}(\Sregular,X) = \Hom_{\FlatEmb}(\colim_Z Z, X) = \lim_Z \Hom_{\FlatEmb}(Z,X).
\]

\begin{proposition}\label{prop:measure_singular_to_regular}
  The map $\Fsingular \to \Fregular, \iota \mapsto \iota|_{\Sregular}$ is measure-preserving.
\end{proposition}

\begin{proof}
  This follows from the fact that for $Z \subseteq \Ssingular$ with $Z' = Z \cap \Sregular$ non-empty the restriction map $\Hom_{\WallEmb_i}(Z,X) \to \Hom_{\WallEmb_i}(Z',X)$ is measure-preserving by construction together with Lemma~\ref{lem:flat_contained_in_singular} which asserts that $\Hom_{\WallEmb_i}(Z',X) = \Hom_{\FlatEmb}(Z',X)$.
\end{proof}

Since each of the basic open sets $U_C(\iota) = \{\iota' \mid \iota'|_C = \iota|_C\}$ has positive measure by construction, we observe:

\begin{lemma}\label{lem:fullsupport}
  The measures $\mu_{\Fregular}$ and $\mu_{\Fsingular_i}$ have full support.
\end{lemma}

Note that one might regard $\Fregular$ as $\Fsingular$ equipped with the coarser $\sigma$-algebra consisting of preimages of measurable sets under the map $\Fsingular \to \Fregular$. Similar remarks apply to much of what follows. From that perspective we will be considering $\Fsingular$ equipped with various $\sigma$-algebras. This perspective may be useful when we apply martingale convergence.

From Lemma~\ref{lem:morphism_independence}, Lemma~\ref{lem:flat_contained_in_singular}, and Proposition~\ref{prop:limit_invariant} we get

\begin{lemma}\label{lem:flow_measure_invariance}
  The measure $\mu_\Fregular$ is $\Gamma \times \tilde{W}$-invariant, the measure $\mu_{\Fsingular_i}$ is $\Gamma \times M_i$-invariant.
\end{lemma}

Recall from Section~\ref{sec:coxeter_complexes} that $\tilde{W} = A \rtimes W$, so in particular $\mu_\Fregular$ is $A$-invariant.

\begin{proof}
  Invariance under $\Gamma$ is clear. The only way in which the definition of the measure is not obviously $\tilde{W}$-invariant (respectively $M_i$-invariant) is the normalization to equip $\Hom(\{0\},X)$ with the counting measure. But Lemma~\ref{lem:morphism_independence} guarantees that for every vertex $v$ of type $0$ the set $\Hom(\{v\},X)$ is equipped with the counting measure as well.
\end{proof}

\begin{lemma}\label{lem:probaFregular}
  There exists a finite $A$-invariant measure $\mu_{\Gamma \backslash \Fregular}'$ on $\Gamma\backslash\Fregular$ such that for any measurable section $s \colon \Gamma \backslash \Fregular \to \Fregular$ the (surjective, generally non-injective) map $\Gamma \times \Gamma \backslash \Fregular \to \Fregular, (g,\Gamma \iota) \mapsto g.s(\Gamma \iota)$ is measure preserving.

\end{lemma}

 We define $\mu_{\Gamma \backslash \Fregular}$ to be the normalization of $\mu_{\Gamma \backslash \Fregular}'$ to a probability measure. In particular, Lemma \ref{lem:probaFregular} proves that quotient map $\pi:(\Fregular,\mu_\Fregular)\to (\Gamma\backslash\Fregular,\mu_{\Gamma \backslash \Fregular})$ is measure-class preserving.

\begin{proof}
  Since the action of $\Gamma$ on $X$ and therefore also on $\Fregular$ is proper, the stabilizer $\Gamma_\iota$ is finite for every $\iota \in \Fregular$. Continuity of the action implies that the map $m \colon \Fregular \to \N_{>0}, \iota \mapsto \abs{\Gamma_\iota} = \abs{\pi^{-1}(\pi(\iota))}$ is measurable. If $s \colon \Gamma \backslash \Fregular \to \Fregular$ is a measurable section, we can define a measure $\mu_{\Gamma \backslash \Fregular}'$ by
  \[
    \mu_{\Gamma \backslash \Fregular}'(U) = \int_{s(U)} \frac{1}{m(\iota)} d\mu_{\Fregular}(\iota).
  \]
  Then for any $n \in \N_{>0}$ and $V \subseteq m^{-1}(n)$ we have
  \begin{multline*}
    d\mu_{\Fregular}(V) = \int_{s\pi(V)} \frac{\abs{\{g \in \Gamma \mid g.\iota \in V\}}}{n} d\mu_\Fregular(\iota)\\ = \int_{\pi(V)} \abs{\{g \in \Gamma \mid g.s(\Gamma\iota) \in V\}} d\mu_{\Gamma \backslash \Fregular}'(\Gamma\iota)
  \end{multline*}
  showing that $\Gamma \times \Gamma \backslash m^{-1}(n) \to m^{-1}(n)$ is measure preserving and hence so is $\Gamma \times \Gamma \backslash \Fregular \to \Fregular$.
  Since $\mu_\Fregular$ is $\Gamma$-invariant, the definition of $\mu_{\Gamma \backslash \Fregular}'$ does not depend on the section $s$. Since $\mu_\Fregular'$ is $A$-invariant, so is $\mu_{\Gamma\backslash \Fregular}$.

  It remains to see that $\mu_{\Gamma \backslash \Fregular}'$ is finite. Since $\Gamma$ acts properly discontinuously on $\Fregular$, for every $\iota \in \Fregular$ there exists an $r > 0$ such that the pointwise stabilizer of $U_{B_r(0)}(\iota)$ fixes $\iota$ and by cocompactness $r$ can be chosen uniformly. Clearly the $U_{B_r(0)}(\iota)$ cover $\Fregular$ and since $\Gamma$ acts cocompactly there are $\iota_1,\ldots,\iota_\ell$ such that $\Gamma.\iota$ meets $U_{B_r(0)}(\iota_k)$ for exactly one $k$. Thus $\Omega = \bigcup_k U_{B_r(0)}(\iota_k)$ is a (compact open therefore) measurable fundamental domain. The map $s \colon \Gamma \backslash \Fregular \to \Fregular$ defined by $\Gamma\pi(\iota) \mapsto \iota$ for $\iota \in \Omega$ is measurable as $\Fregular \to \Gamma \backslash \Fregular$ is open. Since $\Omega$ is compact, $\mu_{\Gamma \backslash \Fregular}'$ is finite.
\end{proof}

The same proof gives:

\begin{lemma}\label{lem:probaFsingular}
  There exists a finite $A_i$-invariant measure $\mu_{\Gamma \backslash \Fsingular_i}'$ on $\Gamma\backslash\Fsingular_i$ such that for any measurable section $s \colon \Gamma \backslash \Fsingular_i \to \Fsingular_i$ the (surjective, generally non-injective) map $\Gamma \times \Gamma \backslash \Fsingular_i \to \Fsingular_i, (g,\Gamma \iota) \mapsto g.s(\Gamma \iota)$ is measure preserving.

  In particular, taking $\mu_{\Gamma \backslash \Fsingular}$ to be the normalization of $\mu_{\Gamma \backslash \Fsingular}'$ the quotient map $\Fsingular\to \Gamma\backslash\Fsingular$ is measure-class preserving.
\end{lemma}

 We define $\mu_{\Gamma \backslash {\Fsingular_i}}$ to be the 
 normalization of $\mu_{\Gamma \backslash {\Fsingular_i}}'$ to a probability measure. In particular, Lemma \ref{lem:probaFsingular} proves that quotient map $\pi_i:(\Fsingular_i,\mu_{\Fsingular_i})\to (\Gamma\backslash\Fsingular_i,\mu_{\Gamma \backslash {\Fsingular_i}})$ is measure-class preserving.

\subsection{Measures on $\Delta$}\label{sec:measures_on_delta}

Let $o \in X$ be a vertex of type $0$. We work in the category $\FlatEmb_*$of pointed objects and morphisms in $\FlatEmb$. Let $((Z,0))_Z$ be the directed system of finite convex subcomplexes of the Weyl chamber $S$ that contain $0$ with basepoint $0$, and basepoint preserving inclusions among them. Then $(S,0) = \colim_Z Z$. By Lemma~\ref{lem:flat_symmetric}, $(\Hom_{\FlatEmb_*}((Z,0),(X,o)))_Z$ (where ) is an inverse system of sets and constant-to-one maps. Its limit is \[\calS_o = \{\iota \colon (S,0) \to (X,o) \mid \iota \text{ type-preserving}\}.\] By Proposition~\ref{prop:cto_limit_measure} we can equip  $\Hom_{\FlatEmb_*}((\{0\},0),(X,o))$ (which is a singleton space) with the counting measure, each $\Hom_{\FlatEmb_*}((Z,0),(X,o))$ with the relatively uniform measure and $\calS_o$ with the limit measure. Note that $\calS_o \to \Delta, \iota \mapsto \iota(S(\infty))$ is a homeomorphism (if we equip $\calS_o$ with the limit topology) via which we identify $\calS_o$ and $\Delta$. We denote by $\mu_\Delta^o$ the limit measure on
\[
  \Delta = \calS_o = \Hom_{\FlatEmb_*}(\colim_Z (Z,0), (X,o)) = \lim \Hom_{\FlatEmb_*}((Z,0),(X,o))
\]
and call it the \emph{pro-uniform measure centered at $o$}. Consider the compact open subset
\[\Fregular_o = \{\iota \in \Fregular \mid \iota(0) = o\}\]
of $\Fregular$. Just as in Proposition~\ref{prop:measure_singular_to_regular} the restriction $\Fregular_o \to \calS_o$ is measure-preserving, so we might as well have defined $\mu_\Delta^o$ to be the pushforward measure along the map $\Fregular_o \to \Delta, \iota \mapsto \iota(S(\infty))$ of the restriction of $\mu_\Fregular$ to $\Fregular_o$.

We define $\mu_{\Delta_i}^o$ in this way, namely as the pushforward of $\mu_{\calS_o}$ along $\calS_o \to \Delta_i^o, \iota \mapsto \iota(\ell_i(\infty))$ or equivalently as the pushforward of $\mu_\Fregular$ along $\Fregular_o \to \Delta_i^o, \iota \mapsto \iota(\ell_i(\infty))$. Clearly:

\begin{lemma}
  For every vertex $o \in X$ of type $0$ the map $(\Delta,\mu_\Delta^o) \to (\Delta_i,\mu_{\Delta_i}^o)$ that takes a chamber to its vertex of type $i$ is measure-preserving.
\end{lemma}

We conclude this section by proving that the action of $\Aut^0(X)$ on $L^\infty(\Delta_i)$ is faithful.

\begin{lemma}\label{lem:aefaithful}
    Let $g\in \Aut^0(X)$ be such that $g$ fixes $\mu_{\Delta_i}^o$-almost every point in $\Delta_i$ for $i=1$ or $2$. Then $g$ is the identity.
\end{lemma}

\begin{proof}
  A hyperbolic isometry $h \in \Aut^0(X)$ translates an axis $\ell$, say from $\xi_-$ to $\xi_+$ \cite[Proposition~6.2]{BridsonHaefliger1999}. If $\eta\in \partial_\infty X$ is opposite $\xi_-$, then there is a geodesic $\ell'$ from $\xi_-$ to $\eta$ whose intersection with $\ell$ is a half-line $\rho$. Then $h^n \ell'$ is a geodesic from $\xi_-$ to $h^n \eta$, and it contains $h^n \rho$, which converges (pointwise) to $\ell$. This proves that $h^n \eta$ converges to $\xi_+$, hence is not fixed if $\eta \ne \xi_+$. In particular, if $c\in \Delta$ is a chamber which is opposite to the simplex supporting $\xi_-$ and not containing $\xi_+$, then it is not fixed by $h$, and both of its vertices in the boundary of $c$ are not fixed by $h$ either. Since almost all chambers are opposite any finite set of points, this proves that the set of points fixed by $h$ is of measure zero (in fact, one can show that it is the union of apartments containing $\xi_+$ and $\xi_-$). We conclude that $g$ cannot be hyperbolic.

  
   So $g$ is elliptic, and fixes a point $o$ in $X$, and since $g$ is type-preserving we may assume that $o$ is a vertex. If $\rho$ is a geodesic ray starting from $o$ to a point of $\Delta_i$, then by definition of the prouniform measures, for every $t>0$ the set $(\{\xi\in\Delta_i\mid \rho(t)\in [o\xi)\})$ has positive $\mu_{\Delta_i}^o$ measure, so that this set contains a point fixed by $g$. Hence $\rho(t)$ is fixed for every $t\geq 0$. Since this is valid for every $\rho$, it follows that all geodesic rays from $o$ to a point in $\Delta_i$ are pointwise fixed by $g$. Then $g$ fixes pointwise the convex hull of all these rays, which is all of $X$.
\end{proof}

\subsection{Measures on $\Opp(\zeta)$ and $\Oppv[i](\zeta)$}

By a \emph{fine point at infinity} of a metric space $Z$ we mean a geodesic ray $\rho \colon (-\infty,k] \to Z$ up to the relation of coinciding on an initial interval. It determines a point at infinity $\zeta = \rho(-\infty)$ but in addition is rigid with respect to translations toward $\zeta$ (like a Busemann function) as well as translations in the transverse direction (unlike a Busemann function (cf.\ the remark at the end of Section~\ref{sec:boundaries}); in particular, it determines a point of $T_\zeta$). For simplicity we will abusively refer to $\rho$ rather than its equivalence class as the fine point.
We define the category $\Flat_\partial$ to be $\Flat$ pointed at fine points at infinity: its objects are pairs $(Z,\rho)$ where $Z$ is an object of $\Flat$ and $\rho \colon (-\infty,k] \to Z$ is a fine point at infinity and a morphism $\alpha \colon (Z,\rho) \to (Z',\rho')$ is a map $Z \to Z'$ such that there is a $k \in \R$ with $(\alpha \circ \rho)_{(-\infty,k])} = \rho'|_{(\infty,k]}$, so in particular $\alpha(\rho(-\infty)) = \rho'(-\infty)$. For instance if $Z \subseteq Z'$ are subcomplexes of $\Sregular$ with $\ell_i((-\infty,k]) \subseteq Z$ then $(Z,\ell_i) \to (Z',\ell_i)$ is a morphism in $\Flat_\partial$ (we write $(Z,\ell_i)$ rather than $(Z,\ell_i|_Z)$ for legibility). We let $\FlatEmb_\partial$ be defined analogously. 

We pick a vertex $o \in X$ of type $0$ for normalization and consider a vertex $\zeta \in \Delta_{\bar{i}}$ of type $\bar{i}$.
Let $\rho_{\zeta,o} \colon (-\infty,0] \to X$ be the geodesic ray with $\rho_{\zeta,o}(-\infty) = \zeta$ and $\rho_{\zeta,o}(0) = o$. Then $(X,\rho_{\zeta,o})$ is an object of $\FlatEmb_\partial$.

Consider the directed system of (inclusions of) $(R,\ell_i)$ where $R$ is a convex subcomplex of $H_i$ that is a cofinite extension of $\ell_i((-\infty,t])$ for some $t$ (recall from Section~\ref{sec:coxeter_complexes} that $H_i$ is the half-space bounded by $\ell_i(\R)$ containing the Weyl chamber $S$). Its colimit is $(H_i,\ell_i)$. Consequently the inverse system $(\Hom_{\Flat_\partial}((R,\ell_i),(X,\rho_{\zeta,o})))_{R}$ has limit
\begin{align*}
  \calH_{\zeta,o} &\defeq \Hom_{\FlatEmb_\partial}((H_i,\ell_i),(X,\rho_{\zeta,o}))\\
                  &= \{\iota \colon H_i \to X \mid \iota \circ \ell_i|_{(-\infty,k]} = \rho_{\zeta,o}|_{(-\infty,k]} \text{ for some }k\}.
\end{align*}
The subsystem $(\ell_i((-\infty,t]),\ell_i)_{t}$ has colimit $(\ell_i(\R),\ell_i)$ and thus the inverse system obtained by applying $\Hom_{\FlatEmb_\partial}(-,(X,\rho_{\zeta,o}))$ has limit
\begin{align*}
  \calL_{\zeta,o} &\defeq \Hom_{\FlatEmb_\partial}((\ell_i(\R),\ell_i),(X,\rho_{\zeta,o}))\\
                  &= \{\iota \colon \ell_i(\R) \to X \mid \iota \circ \ell_i|_{(-\infty,k]} = \rho_{\zeta,o}|_{(-\infty,k]} \text{ for some }k\}.
\end{align*}
Note that the maps
\begin{align*}
  \calH_{\zeta,o} &\to \Opp(\zeta) && \text{and} & \calL_{\zeta,o} & \to \Oppv[i](\zeta)\\
  \iota & \mapsto \iota(S(\infty)) &&& \iota & \mapsto \iota(\ell_i(\infty))
\end{align*}
are bijections to the space of chambers respectively vertices opposite $\zeta$ (see Section~\ref{sec:boundaries}) so we may identify the spaces. Both of the systems in $\Hom_{\Flat_\partial}$ we are considering have cofinal connecting morphisms thus the inverse systems in $\Hom_{\FlatEmb_\partial}$ are constant-to-one by Lemma~\ref{lem:flat_symmetric}. Thus we may give $\Hom((\ell_i((-\infty,0]),\ell_i),(X,\rho_{\zeta,o}))$ the counting measure and equip $\Opp(\zeta) = \calH_{\zeta,o}$ with the limit measure which we denote by $\mu^o_{\Opp(\zeta)}$ and $\Oppv[i](\zeta) = \calL_{\zeta,o}$ with the limit measure which we denote by $\mu_{\Oppv[i](\zeta)}^o$. For orientation, note that the set of all vertices $\rho(\infty)$ where $\rho$ is a geodesic line such that $\rho|_{(\infty,0]}$ is the ray $(\zeta,o]$ has measure one. Similarly, the set of chambers $\iota(S(\infty))$ where $\iota$ ranges over embeddings $H_i \to X$ such that $\iota_{\ell_i((-\infty,0])}$ is that ray has measure one.

Let $\xi \in \Oppv[i](\zeta)$ corresponding to the geodesic line $\rho_{\zeta,\xi,o}$ in $\calL_{\zeta,o}$ with $\rho_{\zeta,\xi,o}(\infty) = \xi$ and $\rho_{\zeta,\xi,o}|_{(-\infty,k]} = \rho_{\zeta,o}{(-\infty,k]}$ for some $k$. Then the fiber of $\xi$ of the map $\Opp(\zeta) \to \Oppv[i](\zeta)$ corresponding to the restriction map $\calH_{\zeta,o} = \calL_{\zeta,o}$ is
\[
  \calL_{\zeta,\xi,o} = \{\iota \colon H_i \to X \mid \iota \circ \ell_i = \rho_{\zeta,\xi,o}\}.
\]
It is the limit of the inverse system of embeddings $(R,\ell_i)_R \to (X,\rho_{\zeta,\xi,o})$ where $R$ ranges over complex subcomplexes of $H_i$ that contain $\ell_i(\R)$. This system is constant-to-one and we give the singleton for $R = \ell_i(\R)$ counting measure and denote the limit measure by $\mu_{\Res(\xi)}^o$. Let $S' \subseteq H_i$ be the sector obtained by reflecting $-S$ in $\ell_i$. The maps
\begin{align*}
  \calL_{\zeta,\xi,o} &\to \Res(\xi) && \text{and} & \calL_{\zeta,\xi,o} & \to \Res(\zeta)\\
  \iota & \mapsto \iota(S(\infty)) &&& \iota & \mapsto \iota(S'(\infty))
\end{align*}
are bijective and fit into a commutative diagram
\begin{equation}\label{eq:zeta_xi}
\begin{tikzcd}
                                            & \Res(\xi) \arrow[dd, "\proj_\zeta"', bend right] \\
{\calL_{\zeta,\xi,o}} \arrow[rd] \arrow[ru] &                                                  \\
                                            & \Res(\zeta) \arrow[uu, "\proj_\xi"', bend right]
\end{tikzcd}
\end{equation}
We regard $\mu_{\Res(\xi)}^o$ as a measure on $\Res(\xi)$. Since each $R$ cofinitely containing $\ell_i(\R)$ is a colimit of $R'$ cofinitely containing $\ell_i((-\infty,t])$ such that restriction $\iota \colon \iota|_{R'}$ is bijective for $\iota \in \calL_{\zeta,\xi,o}$ the measure on $\Res(\zeta)$ does not depend on $\xi$ and vice versa.
By construction it is clear that precomposing with the isomorphism $H_i \to H_{\bar{i}}$ that takes $\ell_i(t)$ to $\ell_{\bar{i}}(-t)$ also preserves the measure. It follows that the measures for $\xi \in \Delta_i$ regarding it as $\iota(\ell_i(\infty))$ and regarding it as $\iota(\ell_{\bar{i}}(-\infty))$ coincide.

We obtain a measure-theoretic analogue of Lemma~\ref{lem:opp_decomposition_topological}.

\begin{proposition}\label{prop:disintegration_opp}
  The map
  \begin{align*}
    (\Res(\zeta) \times \Oppv[i](\zeta), \mu_{\Res(\zeta)}^o \times \mu_{\Oppv[i](\zeta)}^o) &\to (\Opp(\zeta), \mu_{\Opp(\zeta)}^o)\\
    (d,\xi) &\mapsto \proj_\xi (d)
  \end{align*}
  is a measure-preserving bijection. In other words, it induces a isometry \[L^1(\Opp(\zeta),\mu_{\Opp(\zeta)}^o) \to L^1(\Res(\zeta) \times \Oppv[i](\zeta), \mu_{\Res(\zeta)} \times \mu_{\Oppv[i](\zeta)})\] and for an element $f$ we have 
\begin{align}
  &\phantom{=}\int_{\Opp(\zeta)} f(\sigma) d\mu_{\Opp(\zeta)}^o(\sigma) \nonumber\\
  &= \int_{\xi \in \Oppv[i](\zeta)} \int_{\sigma \in \Res(\xi)} f(\sigma) d\mu_{\Res(\xi)}^o(\sigma) d\mu_{\Oppv[i](\zeta)}^o(\xi)\nonumber\\
  &= \int_{\xi \in \Oppv[i](\zeta)} \int_{\eta \in \Res(\zeta)} f(\pr_\xi \eta) d\mu_{\Res(\zeta)}^o(\eta) d\mu_{\Oppv[i](\zeta)}^o(\xi). \label{eq:disintegration_opp}
\end{align}
\end{proposition}

\begin{proof}
  For $a \in \Z$ and $b \in \N$ let $R_{a,b}$ be the subset of $\Sregular$ defined by the equations $\alpha_j \ge 0$ (so it is a subset of $H_i$) $\alpha_j \le b$, and $\alpha_i \le a$ (the reader may refer to Section~\ref{sec:coxeter_complexes} to recall the notations). Then all $R_{a,b}$ contain $\zeta_i = \ell_i(-\infty)$ in their boundary and they form a directed system. The colimit is $\colim_{a,b} R_{a,b} = H_i$ while the colimit with $b = 0$ fixed is $\colim_a R_{a,0} = \ell_i(\R)$. If we allow $a = \infty$ respectively $b = \infty$ to mean that the corresponding constraint is lifted this reads $\colim_{a,b} R_{a,b}= R_{\infty,\infty}=H_i$ and $\colim_a R_{a,0} = R_{\infty,0}$.

  We apply Theorem~\ref{thm:disintegration} to the system $(\Hom_{\FlatEmb_\partial}(R_{a,b})_{a \in \Z,b \in \N},X)$ with base index $(0,0)$. Thus we put counting measure $\nu_0^0$ on the space of embeddings $\iota \colon R_{0,0} = \ell_i((-\infty,0]) \to X)$ with $\ell_i(-\infty) = \zeta$ and $\ell_i(0) = o$, equip every space of embeddings $R_{a,b} \to X$ with the relatively uniform measure. We put the limit measure $\mu$ on the space $\calH_{\zeta,o}$ of embeddings $H_i \to X$ which correspond to elements of $\Opp(\zeta)$; and the limit measure $\nu_0$ on the space $\calL_{\zeta,o}$ of embeddings $\ell_i(\R) \to X$ which correspond to elements of $\Oppv[i](\zeta)$. Here the category constrains all embeddings to coincide on a subray $\ell_i((-\infty,t])$ with the ray $\ell_i((-\infty,0]) \to (\zeta,o]$. In Theorem~\ref{thm:disintegration} we define a measure $\lambda^y$ associated to an embedding $y \colon \ell_i(\R) \to X$. Such an embedding is associated to a unique $\xi \in \Oppv[i](\zeta)$, and to define the measure $\lambda^y=\lambda^\xi$ we put the ``counting measure'' on the singleton $\{y\}$, equip the embeddings $R_{\infty,b} \to X$ with the relatively uniform measure, and put the limit measure $\lambda^\xi$ on the space $\calL_{\zeta,\xi,o}$ of embeddings $H_i \to X$ that restrict to $y$, which correspond to elements of $\Res(\xi)$.

  Then by definition $\mu = \mu^o_{\Opp(\zeta)}$, $\nu_0 = \mu^o_{\Oppv[i](\zeta)}$, and $\lambda^\xi = \mu^o_{\Res(\xi)}$. Theorem~\ref{thm:disintegration} asserts that
  \[
    \int_{\calH_{\zeta,o}} f(\iota) d\mu_{\Opp(\zeta)}^o(\iota) = \int_{\calL_{\zeta,o}} \int_{\calL_{\zeta,\xi,o}} f(\iota) d\mu_{\Res(\xi)}^o(\iota) d\mu_{\Oppv[i](\zeta)}^o(\upsilon).
  \]
  Up to identification of spaces this is the first equation in \eqref{eq:disintegration_opp}. The second equation follows from the fact that the maps in \eqref{eq:zeta_xi} are measure-preserving.
\end{proof}

The measure $\mu_{\Delta}^o$ disintegrates in a similar way.

\begin{proposition}\label{prop:disintegration_delta}
  Let $o \in X$ be a vertex of type $0$ and let $\{i,j\} = \{1,2\}$. Then for $f \in L^1(\Delta)$ we have
  \[
    \int_{\sigma \in \Delta} f(\sigma) d\mu_{\Delta}^o = \int_{\xi \in \Deltav[i]} \int_{\sigma \in \Res(\xi)_o} f(\sigma) d\mu_{\Res(\xi)}^o(\sigma) d\mu_{\Deltav[i]}^o(\xi)
  \]
  where $\Res(\xi)_o = \{\iota(S(\infty)) \mid \iota \circ \ell_i|_{[0,\infty)} = \rho\}$ with $\rho \colon [0,\infty) = [o,\xi)$ the ray from $o$ to $\xi$. In other words, $\sigma \in \Res(\xi)_o$ if and only if the projection of $\sigma$ to $o$ contains the projection of $\xi$ to $o$.
\end{proposition}

\begin{proof}
  It suffices to prove the statement when restricted to an open set $U_o(\sigma)$. For $a,b \in \N$ let $S_{a,b}$ be the subset of $S$ defined by the equations $\alpha_i \le a$, $\alpha_j \le b$. Then $S$ is the colimit of the $S_{a,b}$. We consider type-preserving embeddings $S_{a,b} \to X$ that take $0$ to $o$. We put the counting measure on the singleton $S_{0,0} \to \{o\}$ and relatively prouniform measures on the remaining spaces. Applying Theorem~\ref{thm:disintegration} describes the measure on embeddings of $S_{\infty,\infty}$ as an integral over embeddings $S_{\infty,0}$ with respect to extensions from $S_{\infty,0}$ to $S_{\infty,\infty}$. These measures are $\mu_\Delta^o$, $\mu_{\Deltav[i]^o}$ and (in restriction to $\Res(\xi)_o$) $\mu_{\Res(\xi)}^o$.
\end{proof}

\subsection{Locally proportional measures}\label{sec:basic_open_sets}

We will want to specify basic (compact) open subsets of the various spaces just described. The general pattern is as follows: we have a space $\calZ$ (which can be $\calS_o$, $\calH_{\zeta,o}, \calL_{\zeta,o}$) of embeddings of $Z \subseteq \Sregular$ (respectively $S, H_i, \ell_i(\R)$) to $X$ that take $0$ to a base vertex $o \in X$. The space $\calZ$ can be identified with a space $\Omega$ (resp. $\Delta, \Opp(\zeta), \Oppv[i](\zeta)$) of boundary objects of $X$. A basic open neighborhood of $\iota \in \calZ$ is the set of embeddings that coincide with $\iota$ on a compact subset $C \subseteq Z$ which then corresponds to an open neighborhood in $\Omega$. Thus for $C \subseteq Z$ we define
\[
  U_C(\iota) = \{\iota' \in \calZ \mid \iota|_C = \iota'|_C\} \subseteq \calZ
\]
which corresponds to the neighborhood
\[
  U_C^o(\omega) = \{\iota'(Z(\infty)) \mid \iota|_C = \iota'|_C\} \subseteq \Omega
\]
where $\iota \in \calZ$ is uniquely determined by $\iota(0) = o$ and $\iota(Z(\infty)) = \omega$. If $C = \{p\}$ is just a single point we just write $U_p^o$ for $U_C^o$. In the case $\Omega = \Delta$ by far the most common set $C$ to consider will be the chamber $c_0$ at the tip of $S$ (the one containing $0$); for that reason we introduce the additional notation
\begin{equation}
  U^o(\sigma) = U^o_{c_0}(\sigma) = \{\sigma' \in \Delta \mid \pr_o \sigma' = \pr_o \sigma\}\label{eq:Uo}
\end{equation}
in that case.

A useful observation using this notation is that while the inclusion \[(\Opp(\zeta),\mu_{\Opp(\zeta)}^o) \to (\Delta,\mu_\Delta^o)\] is not measure-preserving, its restriction to $U_0^o$ is measure-preserving up to a constant factor. We want to generalize this observation by showing that various measures that we define on $\Delta$ not only define the same measure class but are in fact locally proportional.

One way to think about this is in terms of expectations: If $(Z,\mu)$ is a measure space, $f\in L^1(Z)$ and $U\subset Z$ is of finite, positive measure, we denote by $\fint_U f(x) d\mu(x)$ the average (or expectation) of $f$ over $U$, that is, 
\[\fint_U f(x) d\mu(x)=\frac{1}{\mu(U)}\int_U f(x)d\mu(x).\]
We want to say that for sufficiently small $U$ the expectation does not depend on $U$. This corresponds to the fact that the Radon-Nikodym derivative in \cite[Proposition 6.1]{RemyTrojan} is locally constant.

\begin{proposition}\label{prop:locally_pro_uniform}
  There is a basis $\calB$ for the topology of $\Delta$ consisting of compact open sets and for each $U \in \calB$ a probability measure $\mu_U$ such that the following hold:
  \begin{enumerate}
    \item For every finite set $F$ of vertices of type $0$ and every $\sigma \in \Delta$ there is a neighborhood $U \in \calB$ of $\sigma$ such that for all $o \in F$
      \begin{equation}\label{eq:locally_pro_uniform}
        \mu_U = \frac{1}{\mu^o_\Delta(U)} \mu^o_{\Delta}.
      \end{equation}\label{item:locally_pro_uniform_sufficiently_small}
    \item For every $U \in \calB$ the set $A$ of vertices $o$ of type $0$ such that \eqref{eq:locally_pro_uniform} holds is non-empty and its set of accumulation points in $\Delta$ has non-empty interior.\label{item:locally_pro_uniform_sufficiently_far}
  \end{enumerate}
\end{proposition}

Note that \eqref{eq:locally_pro_uniform} passes to subsets $V \subseteq U$ of positive measure.

\begin{proof}
  We let $\calB$ consist of the sets $U^p(\sigma)$ with $\sigma \in \Delta$ and $p \in X$ a special vertex of type $0$. We claim that if $p,o \in X$ are special vertices, the restrictions of $\mu^o_\Delta$ and $\mu^p_\Delta$ are proportional on $U^p(\sigma)$ if $p$ lies in the sector $\conv(o,\sigma)$.

  To see this let $\iota_0 \colon S \to X$ be the type-preserving embedding with $\iota_0(0) = o$ and $\iota_0(\infty) = \sigma$. Then by the assumption on $p$ there is an $a \in A$ be such that $p = \iota(-a + 0)$. In what follows we use additive notation for the action of $A$ on $\Sregular$ but multiplicative notation for the action on $\Fregular$: $(a.\iota)(x) = \iota(-a + x)$.

  Then $U = U_{-a + c_0}^o(\sigma)$ consists of $\iota(\infty)$ where $\iota(0) = o$ and $\iota(-a + c_0) = \iota_0(-a + c_0)$ but also of $\iota(\infty)$ where $\iota(0) = p$ and $\iota(c_0) = \iota_0(-a + c_0)$:
  \[
    U = U_{-a + c_0}^o(\iota_0) = U_{c_0}^p(a.\iota_0|_S).
  \]
  The map $U_{-a + c_0}^o(\iota_0) \to U_{c_0}^p(a.\iota_0|_S), \iota \mapsto a\iota|_S$ is a bijection because $S$ is the convex hull of $0$ and $aS$ so there is a unique extension of $\iota|_{-a + S}$ to $\iota$ with $\iota(0) = o$.

  It remains to recall that that the limit measures on the space of embeddings $S \to X$ along relatively uniform measures based on counting measures on the spaces of embeddings of $0$, of $-a + 0$ or of the convex hull $\conv(0,-a + 0)$ are all proportional to each other by Lemma~\ref{lem:flat_symmetric}.
    
  Now \eqref{item:locally_pro_uniform_sufficiently_small} follows from the fact that the intersection $\bigcap_{o \in F} \conv(o,\sigma)$ is non-empty and thus contains suitable $p$s.

  For \eqref{item:locally_pro_uniform_sufficiently_far} one observes that for $p$ and $\sigma$ fixed any vertex $o \in \conv(p,\tau)$ of type $0$ is suitable as long as $\pr_p \sigma$ and $\pr_p \tau$ are opposite in the link of $p$.
\end{proof}

We conclude:

\begin{corollary}
  Let $o,o' \in X$ be special vertices. The measures $\mu^o_\Delta$ and $\mu^{o'}_\Delta$ are mutually absolutely continuous and the Radon--Nikodym derivative $d\mu^o_\Delta/d\mu^{o'}_\Delta$ is locally constant.
\end{corollary}

Thus the measures $\mu^o_\Delta$ are locally proportional. They are also locally proportional to the measures on $\Opp(\zeta)$. Note the measures defined on $\Opp(\zeta)$ are globally proportional with varying $o$, so the choice of base vertex does not matter.

\begin{proposition}\label{prop:opp_locally_proportional}
  Let $\zeta \in \Deltav[\bar{i}]$ and let $o \in X$ be special of type $0$. There is a basis $\calB_\zeta$ for the topology on $\Opp(\zeta)$ refining the basis from Proposition~\ref{prop:locally_pro_uniform} such that for every $U \in \calB$ we have
  \[
    \mu_U = \frac{1}{\mu^o_{\Opp(\zeta)}(U)} \mu_{\Opp(\zeta)}^o
  \]
  where $\mu_U$ is the measure from Proposition~\ref{prop:locally_pro_uniform}.

  In particular, the inclusion $\Opp(\zeta) \to \Delta$ is measure-class preserving.
\end{proposition}

\begin{proof}
  Let $\sigma \in \Opp(\zeta)$ let $\Sigma$ be the apartment of $X$ containing $\sigma$ and $\zeta$ in its boundary and let $o' \in \Sigma$ be the vertex such that the rays $(-\infty,0] \to (\zeta,o]$ and $(-\infty,0] \to (\zeta,o']$ coincide on an initial interval. We claim that the restrictions of $\mu_{\Delta}^{o'}$ and $\mu_{\Opp(\zeta)}^{o'}$ coincide on $U \defeq U^{o'}_0(\sigma)$. It then follows that the restrictions to $U^p_0(\sigma)$ (which is open in $U$) are proportional.

  Since the $\mu^o_\Delta$ and $\mu^{o'}_\Delta$ are proportional on a basis of $\Delta$ by Proposition~\ref{prop:locally_pro_uniform} and $o$ and $o'$ define the same measure on $\Opp(\zeta)$ there is no loss in considering $o$ instead of $o'$.

  The map $\calH_{\zeta,o} \to \calS_o, \iota \mapsto \iota|_S$ is bijective when restricted and corestricted to $U^o_0(\sigma)$. Since both measures are limits of relatively uniform measures putting the counting measure on embeddings $\{0\} \to X$ (of which there is a single one $0 \mapsto o$ in $U^o_0$) both measures coincide.

  The last statement follows because $\Opp(\zeta)$ is open and dense in $\Delta$.
\end{proof}

Carrying out the same proof in the direction of $\zeta$ gives the following statement:

\begin{lemma}\label{lem:projection_measure-preserving}
  Let $\zeta \in \Deltav[\bar{i}]$ and let $o \in X$ be a special vertex of type $0$.
  The map $\pr_\zeta \colon (\Delta,\mu^o_\Delta) \to (\Res(\zeta),\mu^o_{\Res(\zeta)})$ is measure preserving.
\end{lemma}

\begin{proof}
  It suffices to show that the measure is preserved in restriction to a sufficiently small open set $U$. Specifically we take $U$ to be contained in some $U^o(\sigma) = \{\sigma' \in \Delta \mid \pr_o \sigma' = \pr_o \sigma\}$ with $o$ a special vertex of type $0$, to lie in the basis $\calB_\zeta$ from Proposition~\ref{prop:opp_locally_proportional} and to decompose as $V \times W$ with $V \subseteq \Res(\zeta)$ and $W \subseteq \Oppv[i](\zeta)$ as in Proposition~\ref{prop:disintegration_opp}. More precisely then $U = \{\pr_\xi(\sigma) \mid \sigma \in V, \xi \in W\}$.
  Then
  \[
    \mu_{\Opp(\zeta)}^o|_U = \mu_{\Res(\zeta)}^o|_V \cdot \mu_{\Oppv[i](\zeta)}^o|_W
  \]
  showing that $(U,\mu_{\Opp(\zeta)}^o) \to (\Res(\zeta),\mu_{\Res(\zeta)}^o)$ preserves the measure up to a constant $\mu_{\Oppv[i](\zeta)}^o(W)$ since $\pr_\zeta$ does not depend on the factor $W$.
  Combining Proposition~\ref{prop:disintegration_opp} and Proposition~\ref{prop:disintegration_delta} we find that
  \[
    \frac{\mu_{\Delta}^o(U)}{\mu_{\Opp(\zeta)}^o(U)}= \frac{\mu_{\Deltav[i]}^o(W)}{\mu_{\Oppv[i](\zeta)}^o(W)}
  \]
  is proportionality constant of $\mu_\Delta^o$ and $\mu_{\Opp(\zeta)}^o$ on $U$.
\end{proof}

\subsection{The measures on $\Deltaop$ and $\Deltaopv[i]$}

To construct a measure on $\Deltaop$ we consider the map
\begin{align*}
  \pi \colon \Fregular &\to \Deltaop\\
  \iota & \mapsto (\iota(S(\infty)),\iota(-S(\infty))).
\end{align*}

We would like to define a measure on $\Deltaop$ by pushing $\mu_\Fregular$ forward by $\pi$, but this is not directly possible because $\mu_\Fregular$ is infinite. In order to circonvene this difficulty, we need to fix an origin, of which we will get rid later.

Let $\iota \in \Fregular$ be arbitrary and let $o = \iota(0)$ be the vertex of type $0$ that $\iota$ maps the vertex $0 \in \Sregular$ to. The restriction of $\pi$ to $\Fregular_o$ (defined in Section~\ref{sec:measures_on_delta}) is injective but not surjective: its image is the compact open set $\Deltaop_o = U_0^o(\iota)$ consisting of pairs of opposite chambers $c,d \in \Delta$ such that the unique apartment in whose boundary they lie contains $o$. Let $\mu_{\Deltaop}^o$ denote the measure on $\Deltaop_o$ obtained by pushing $\mu_\Fregular$ along the map above.

The key observation that makes the definition of the measure on $\Deltaop$ possible and which plays the role of the computations \cite[Theorem 3.17]{Parkinson2} and \cite[Proposition 6.1]{RemyTrojan} is the following consequence of Lemma~\ref{lem:morphism_independence}. In essence it asserts that when changing the basepoint, the measure on $\iota(S(\infty))$ is scaled up by the same amount that the measure on $\iota(S(-\infty))$ is scaled down, so both effects cancel. To appreciate the subtlety note that $o$ and $o'$ need not lie in the same $\Gamma$-orbit.

\begin{lemma}\label{lem:muop_basepoint_change}
  For two special vertices $o, o' \in X$ of type $0$ the restrictions of $\mu_{\Deltaop}^o$ and of $\mu_{\Deltaop}^{o'}$ to the compact open set $\Deltaop_o \cap \Deltaop_{o'} \subseteq \Deltaop$ coincide.
\end{lemma}

Thus we can simply define $\mu_{\Deltaop}$ to be the measure that coincides with $\mu_{\Deltaop}^o$ on $\Deltaop_o$ for every $o$.

\begin{proof}
  Let $\iota \in \calF$ be such that $\iota(0) = o$ and $\iota(0') = o'$ for some vertex $0'$ of type $0$. Then $\Deltaop_o \cap \Deltaop_{o'} = \pi(U_0(\iota) \cap U_{0'}(\iota))$ and the restriction of $\mu^o_{\Deltaop}$ to the intersection is the pushforward from $U_0(\iota)$, so by injectivity of $\pi$ on $\Fregular^o$  we have $\mu^o_{\Deltaop}( \Deltaop_o \cap \Deltaop_{o'} )   = \mu_\Fregular(U_0(\iota)\cap U_0(\iota'))$.

  Let $a \in A$ be the element such that $a0 = 0'$. Then $\iota' \defeq a^{-1} \cdot \iota$ has the property that $\iota'(0) = \iota(a0) = \iota(0') = o'$ and $\iota'(a^{-1}0) = \iota(0) = o$. Thus $\Deltaop_o \cap \Deltaop_{o'} = \pi(U_0(\iota') \cap U_{a^{-1}0}(\iota'))$ and similarly as for $\mu^o$ we get that $\mu^{o'}_{\Deltaop_o}(
 \Deltaop_o \cap \Deltaop_{o'}) = \mu_\Fregular(U_0(\iota')\cap U_{a^{-1}0}(\iota')).$

 Note that
  \begin{align*}
    U_0(\iota') &= U_0(a^{-1} \iota) = a^{-1} \cdot U_{a \cdot 0}(\iota) = a^{-1} \cdot U_{0'}(\iota) & \text{and}\\
    U_0(\iota) &= U_0(a \iota') = a \cdot U_{a^{-1} \cdot 0}(\iota') 
  \end{align*}
  so $U_0(\iota') \cap U_{a^{-1}0}(\iota') = a^{-1} (U_0(\iota) \cap U_{0'}(\iota))$.
  The point is now that the measure $\mu_\Fregular$ is invariant by $A$ by Lemma~\ref{lem:flow_measure_invariance}, so $\mu_\Fregular(U_0(\iota')\cap U_{a^{-1}0}(\iota'))=\mu_\Fregular(U_0(\iota)\cap U_0(\iota'))$, which proves the result.
  
\end{proof}

The construction of the measure on $\Deltaopv[i]$ is analogous. The map
\begin{align*}
  \pi_i \colon \Fregular &\to \Deltaopv[i]\\
  \iota &\mapsto (\iota(\ell_i(\infty)),\iota(\ell_i(-\infty)))
\end{align*}
restricts to a map of pointed embeddings $\Fregular_o \to \Deltaopv[i]$ and we can define the measure on the image to be the pushforward measure $\mu_{\Deltaopv[i]}^o$. We can then define $\mu_{\Deltaopv[i]}$ to be the measure that restricts to these since the proof of Lemma~\ref{lem:muop_basepoint_change} also gives:

\begin{lemma}\label{lem:muopv_basepoint_change}
  For two special vertices $o, o' \in X$ of type $0$ the restrictions of $\mu_{\Deltaopv[i]}^o$ and of $\mu_{\Deltaopv[i]}^{o'}$ to the compact open set $\pi_i(\Fregular_o) \cap \pi_i(\Fregular_{o'})$ coincide.
\end{lemma}

\begin{remark}
Note that since $\Fsingular_i \to \Fregular$ is measure-preserving, we could equivalently have used the map $\Fsingular_i \to \Deltaopv[i]$ or, in the other direction, the set of embeddings $\{\ell_i(\R) \to X\}$, to define the measures on $\Deltaopv[i]$.
\end{remark}

The fibers of $\pi$ are apartments of $X$ without a fixed basepoint but with the same orientation, so they are precisely $A$-orbits: $\pi^{-1}(\pi(\iota)) = A.\iota$ for $\iota \in \Fregular$. In other words $\pi$ induces a bijection $A \backslash \Fregular \to \Deltaop$. In fact there is a local trivialization:

\begin{lemma}\label{lem:F/A}
  The map $\Fregular \to \Deltaop$ is a locally trivial principal bundle with bundle group $A$. The trivializations \begin{align*}\tau_o \colon A . \Fregular_o &\to A \times \Deltaop_o\\a.\iota &\mapsto (a,(\iota(S(\infty)),\iota(-S(\infty))))\end{align*} are measure-preserving. In particular, any measurable section $s \colon \Deltaop \to \Fregular$ induces a measure preserving bijection $\Fregular \to A \times \Deltaop$.
\end{lemma}

\begin{proof}
  For $\iota \in \Fregular_o$ and $\iota'$ with $\pi(\iota) = \pi(\iota')$ there is a unique $a$ with $a.\iota(0) = \iota'(0)$ and then $a.\iota = \iota'$. This shows bijectivity of the trivialization. The map $\Fregular_o \to \Delta_o^\op$ is measure-preserving by definition and it extends to a measure preserving map on the $A$-orbit since the measure on $\Fregular$ (and on $A$) is $A$-invariant.

  A measurable section $s \colon \Delta^\op \to \Fregular$ amounts to a measurable map $s' \colon \Delta^\op \to X_0,(c,c) \mapsto s(c,c')(0)$ such that $s'(c,c')$ lies in the apartment that has $c$ and $c'$ in the boundary. It gives rise to a map
  \begin{align*}
    \tau_s \colon \Fregular &\to A \times \Delta^\op\\
    \iota &\mapsto \tau_{s'(0)}(\iota)
  \end{align*}
  that is measurable. If $U \subseteq A.\Fregular_o$ is measurable then
  \begin{multline*}
    \mu_{\Fregular}(U) = \int_{\Delta_o^\op} \abs{\tau_o(U) \cap (A \times \{(c,c'\})} d \mu_{\Delta_o^\op}(c,c')\\ = \int_{\Delta_o^\op} \abs{\tau_s(U) \times (A \times \{(c,c')\})} d\mu_{\Delta_o^\op}(c,c'),
  \end{multline*}
since $\tau_s$ and $\tau_o$ differ in the fibers of $A \times \Delta^\op \to \Delta^\op$ but not their size. This shows that $\tau_s$ is a measure preserving.
\end{proof}
From the construction we see:

\begin{lemma}\label{lem:FregMeasPres}
  The map $\Deltaop \to \Deltaopv[i]$ that takes a pair $(c,c')$ of opposite vertices to the pair $(\xi,\xi')$ where $\xi$ is the vertex of type $i$ of $c$ and $\xi'$ is the vertex of type $\bar{i}$ of $c'$ is measure preserving. The natural map $\Fregular \to \Deltaop$ is measure-class preserving.

  In particular, the maps $\Fregular \to \Deltaop \to \Deltaopv[i]$ are measure-class preserving.
\end{lemma}

\begin{proof}
  The first map takes $(\iota(S(\infty)),\iota(S(-\infty)))$ to $(\iota(\ell_i(\infty)),\iota(\ell_i(-\infty))$ and is measure-preserving by definition.

  The second map is measure-class preserving by Lemma~\ref{lem:F/A}.
\end{proof}

Similarly to Lemma~\ref{lem:F/A} we have:

\begin{proposition}\label{prop:Y/M}
  The map $\pi_i \colon \Fsingular_i \to \Deltaopv[i]$ is a principal bundle with bundle group $M_i$. Any measurable section $s \colon \Deltaopv[i] \to \Fsingular_i$ induces an isomorphism $\Fsingular_i \to M_i \times \Deltaopv[i]$.
\end{proposition}

\begin{proof}
  For $(\xi,\zeta) \in \Deltaopv[i]$ the fiber $\pi_i^{-1}(\xi,\zeta)$ consist of isomorphisms $\Ssingular_i \to Y_{\xi,\zeta}$ that are in $\Fregular_i$. The group $M_i$ acts regularly on these by precomposition. Since $\mu_{\Fsingular_i}$ is $M_i$-invariant by Lemma~\ref{lem:flow_measure_invariance} the truth of the proposition is not affected if we modify $s$ by a measurable map $\Deltaopv[i] \to M_i$. We use this freedom by considering an arbitrary vertex $o$ of type $0$ and restricting to $\pi_i \colon (\Fsingular_i)_o \to \pi_i((\Fsingular_i)_o)$ and assuming that $s$ splits this restriction, i.e.\ $s(\pi_i((\Fsingular_i)_o)) \subseteq (\Fsingular_i)_o$.

  The set $\Hom_{\WallEmb_i}(\Ssingular_i,\Ssingular_i)$ is $M_i$, not only as a set but as a measurable space as in Example~\ref{ex:prouni_motivation_haar} (this is where our efforts in formulating pro-uniform measures in this generality pays off). Similarly the space $\Hom_{\WallEmb_i,*}((\Ssingular_i,0),(\Ssingular_i,0))$ of pointed maps is the stabilizer $(M_i)_0$. Finally, writing $\ell_i$ for the set $\ell_i(\R)$ the space $\Hom((\Ssingular_i,\ell_i,0),(\Ssingular_i,\ell_i,0))$ of morphisms in the category of pointed pairs is the pointwise stabilizer $(M_i)_{\ell_i}$, that could conceivably have index $2$ in $(M_i)_0$.

  The fibers of the restriction map $\Hom_{\WallEmb_i}(\Ssingular_i,X) \to \Hom_{\WallEmb_i}(\ell_i,X)$ can be identified with $\Hom((\Ssingular_i,\ell_i,0),(\Ssingular_i,\ell_i,0))$: if $\iota \colon \Ssingular_i \to X$ lies in the fiber of $\mu \colon \ell_i \to X$, i.e.\ $\iota|_{\ell_i} = \mu$, then every element $\iota'$ in the fiber is obtained as $\iota \circ \alpha$ with $\alpha \in \Hom((\Ssingular_i,\ell_i,0),(\Ssingular_i,\ell_i,0))$, i.e.\ $\alpha \colon \Ssingular_i \to \Ssingular_i$ fixes $\ell_i$ pointwise (and lies in $M_i$).

  The statement of the previous paragraph remains true for the restricted map $\Hom_{\WallEmb_i,*}((\Ssingular_i,0),(X,o)) \to \Hom_{\WallEmb_i,*}((\ell_i,0),(X,o))$ whose domain is $(\Fsingular_i)_o$ and whose codomain can be identified with $\pi_i((\Fsingular_i)_o)$.

  We can obtain the measure on $\Hom_{\WallEmb_i,*}((\Ssingular_i,0),(X,o))$ as follows. We let $Z_a^b \cong \ell_i \times T$ for $a,b \in \N$ be the convex hull of the segment $\ell_i([-b,b])$ and of the points $(0,v)$ where $v$ has combinatorial distance $a$ from $0$. Then $Z_0^b = \ell_i([-b,b])$ and the complexes $Z_a^b$ exhaust $\Ssingular_i$. Then \[\lim_{a,b} \Hom_{\WallEmb_i,*}((Z_a^b,0),(X,o)) = (\Fsingular_i)_o\] as a measure space and \[\lim_b \Hom_{\WallEmb_i,*}((Z_0^b,0),(X,o))\] can be identified with $\pi_i((\Fsingular_i)_o)$. In fact, Theorem~\ref{thm:disintegration} implies that the limit measure is in fact the pushforward measure $\mu_{\Deltaopv[i]}|_{\pi_i(\Fsingular_i)}$ and moreover it provides a disintegration formula
  \[
    \int_{(\Fsingular_i)_o}g(\iota) d\mu_{\Fsingular_i} = \int_{\pi_i((\Fsingular_i)_o)} \int_{\iota \in \pi_i^{-1}(\xi,\zeta)} g(\iota) d\lambda^{(\xi,\zeta)}(\iota) d\mu_{\Deltaopv[i]}(\xi,\zeta)
  \]
  where $\lambda^{\xi,\zeta}$ is a measure on $\pi_i^{-1}(\xi,\zeta) = \Hom((\Ssingular_i,\ell_i,0),(X,(\zeta,o,\xi),o))$ where $(\zeta,o,\xi)$ denotes the line in $X$ from $\zeta$ to $\xi$ through $o$.

  The section picks a preferred element $s(\xi,\zeta)$ in this space, allowing to identify it with $(M_i)_{\ell_i} = \Hom((\Ssingular_i,\ell_i,0),(\Ssingular_i,\ell_i,0))$ via $(M_i)_{\ell_i} \to \pi_i^{-1}(\xi,\zeta), m \mapsto m.s(\xi,\zeta)$. It remains to observe that this identification is measure preserving since $(M_i)_{\ell_i}$ carries the limit measure of the $\Hom_{\Wall_i,*}((Z_a^{\infty},0),(\Ssingular_i,0))$ (note that $\Hom_{\Wall_i,*}((Z_0^1,0),(\Ssingular_i,0)) \to \Hom_{\Wall_i,*}((Z_0^0,0),(\Ssingular_i,0))$ has fiber size at most two and $\Hom_{\Wall_i,*}((Z_0^\infty,0),(\Ssingular_i,0)) \to \Hom_{\Wall_i,*}((Z_0^1,0),(\Ssingular_i,0))$ has fiber size one).
  Thus
  \[
    \int_{(\Fsingular_i)_o}g(\iota) d\mu_{\Fsingular_i} = \int_{\pi_i((\Fsingular_i)_o)} \int_{m \in (M_i)_{\ell_i}} g(m.s(\xi,\zeta)) d\mu_{M_i}(m) d\mu_{\Deltaopv[i]}(\xi,\zeta)
  \]
  showing that $(\Fsingular_i)_o$ is measure equivalent to $(M_i)_{\ell_i} \times \pi_i((\Fsingular_i)_o)$ as needed.
\end{proof}

Recall the definition of $\overline\Delta_i$, which is the set of all pairs $(u,\iota)$ with $u\in \Delta_i$ and $\iota : T\to T_u$, and is equipped with a map $\overline \sigma_+:\Fsingular_i\to \overline \Delta_i$. We equip $\overline \Delta_i$ with the pushforward of the measure class on $\Fsingular_i$ by $\sigma_+$: a set $A\subset \overline \Delta_i$ has measure $0$ if and only if $\sigma_+^{-1}(A)$ has measure 0. We define similarly the measure class on $\overline \Delta_{\overline i}$.

\begin{lemma}\label{lem:MeasDeltaBar}
    
    The measure class on $\overline\Delta_i$ is obtained by integrating the Haar measure on $M_i/A_i$ on each fiber along the prouniform measure on $\Delta_i$.  The projection $p_i:\overline\Delta_i\to \Delta_{ i}$ is measure-class preserving.

    Furthermore, the measure class $\mu_{\overline\Delta_i}$ is also equal to the pushforward of the measure class of $\mu_{\Fsingular_{\overline i}}$ by $\sigma_-$.

\end{lemma}

\begin{proof}
    By proposition \ref{prop:Y/M} we can identify $\Fsingular_i$ to $M_i\times \Delta_i^{\op}$. By definition the map $\sigma_+$, with this identification, is the map obtained by quotienting $M_i$ by $A_i$ and taking the first projection $\Delta_i^\op\to\Delta_i$. Hence the measure-class on $\overline \Delta_i$ is the class of the product of Haar on $M_i/A_i$ and the prouniform measure on $\Delta_i$. This proves the first part.  With this identification $p_i$ is just the second projection and it is measure-class preserving. 

    For the last part let $s \colon \Ssingular_i \to \Ssingular_{\bar{i}}$ be a homeomorphism as in Section~\ref{sec:extended_projectivity_groupoid} and note that it can be taken so that the map $s^* \colon \iota \colon \iota \circ s$ takes $\Fsingular_i$ to $\Fsingular_{\bar{i}}$. Then $S^*$ is measure preserving and the claim follows from the fact that $\overline{\sigma}_+(\iota) = \overline{\sigma}_-(\iota \circ s)$.
\end{proof}

\section{Ergodicity of Cartan flows}

The main result of the section is that the singular Cartan flow is ergodic, Theorem~\ref{thm:SingularCartanErgodic}. Before we prove that the regular Cartan flow is ergodic, Theorem~\ref{thm:CartanErgodic}. Both proofs use the Hopf argument, which was conveniently packaged in \cite[Theorem~B.1]{BCL}. Two sequences $(a_n)_n$ and $(b_n)_n$ in a compact space $V$ are \emph{proximal} if every neighborhood $U$ of the diagonal in $V \times V$ contains all but finitely many $(a_n,b_n)$.

\begin{theorem}[Hopf argument]\label{thm:Hopf}
  Let $Z$ be a locally compact space with commuting proper actions $T \curvearrowright Z$, $\Lambda \curvearrowright Z$. Denote the quotients and quotient maps by $\phi \colon Z \to V \defeq \Lambda \backslash Z$ and $\psi \colon Z \to Y \defeq T \backslash Z$. Let $(\pi_i \colon Y \to W_i)_{i \in I}$ be a family of $\Lambda$-equivariant maps:
  \begin{equation}
\begin{tikzcd}
  & Z \arrow[ld, "\phi"'] \arrow[rd, "\psi"] &                                            &     \\
V &                                          & Y \arrow[ld, "\pi_i"'] \arrow[rd, "\pi_j"] &     \\
  & W_i                                      &                                            & W_j
\end{tikzcd}
  \end{equation}
  Let all spaces be equiped with measure classes making the maps measure-class preserving.

  Assume that
  \begin{enumerate}
    \item $T$ is amenable;
    \item $V$ is compact and the measure class on $V$ contains a $T$-invariant probability measure;
    \item for every $i \in I$ there is a central element $t_i \in T$ such that for all $z,z' \in Z$ with $\pi_i\psi(z) = \pi_i\psi(z')$ there exist $t,t' \in T$ such that the sequences $(t_i^nt \phi(z))_{n \in \N}$ and $(t_i^nt'\phi(z'))_{n \in \N}$ are proximal in $V$;
    \item every function in $\bigcap_{i \in I} \pi_i^*(L^\infty(W_i))^\Lambda$ is essentially constant.
  \end{enumerate}
  Then the actions $T \curvearrowright V$ and $\Lambda \curvearrowright Y$ are ergodic.
\end{theorem}

\subsection{Ergodicity of the Cartan flow} In preparation of the section's main goal, ergodicity of the singular Cartan flow, we first prove ergodicity of the regular Cartan flow:

\begin{theorem}\label{thm:CartanErgodic}
    The action of $\Gamma $ on $\Fregular/A$ is ergodic.
\end{theorem}

The proof is basically the same as in \cite[Theorem 7.1]{BCL}.

\begin{proof}
  We apply Theorem~\ref{thm:Hopf} with $T=A$, $\Lambda=\Gamma$, $Z=\Fregular$, and $W_i=\Delta$ for $i=1,2$.
  By Lemma \ref{lem:F/A} we can identify in a measure-class preserving way $\Fregular/A$ with $\Delta^{\op}$, and the maps $\pi_+$ and $\pi_-$ will be the projections $\pi_\pm \colon (c_+,c_-) \mapsto c_\pm$ (induced by $\sigma_\pm \colon \Fregular \to \Delta$). By Lemma \ref{lem:probaFregular} we know that there is indeed an invariant probability measure on $\Gamma\backslash\Fregular$. That $f\in \pi_+^*L^\infty(\Delta)\cap \pi_-^*L^\infty(\Delta)$ consists of essentially constant functions follows from Fubini using that $\Deltaop$ and $\Delta \times \Delta$ are measure-class equivalent.

    The remaining assumption to check is the proximality condition. Let $A^+\subset A$ be the set of translations sending $0$ to some point in the the interior of the sector $S$. Let $t_- \in A^+$.
    
    If $\iota,\iota'\in\Fregular$ are such that $\pi_-(\iota)=\pi_-(\iota')$ then the image of $\iota$ and $\iota'$ contain a common subsector of the negative sector  $-S$. This means that there exists $t,t'\in A$ such that $\iota(x-t)=\iota'(x-t')$ for every $x\in -S$. Since $t_-$ is regular, we see that if $y\in F$ then for $n$ large enough we have $y-nt_-\in -S$. It follows that the sequences $(t_-^n t\iota)_{n\in\N}$ and $(t_-^nt' \iota')_{n\in\N}$ are proximal. The argument for $\pi_+$ is analogous.
\end{proof}

Next we want to extend the this ergodicity in a way that does not enter the proof of Theorem~\ref{thm:main_theorem} or Theorem~\ref{thm:factorintro} but is the first half of Theorem~\ref{thm:introergodic}.

Taking away the measure-theoretic subtleness ergodicity of the Cartan flow rests on two basic facts: the first is that any two elements of $\Deltaop$ are connected by a sequence of elements among which two consecutive ones coincide in the first or second coordinate; the second is that $A$ acts transitively on the fibers of $\Fregular \to \Deltaop$. For instance it trivially follows from both facts that an actual function $f \colon \Fregular \to \C$ that is $A$-invariant and constant on fibers to $\pi_\pm$ is constant. This is analogous to the geodesic flow on a hyperbolic manifold or more generally hyperbolic dynamical systems.

If one replaces $A$ by a subgroup $B$ that is not transitive on these fibers, the situation becomes analogous to partially hyperbolic dynamical systems such as the step-$1$-geodesic flow, see \cite{BurnsWilkinson10} \cite[Sections~8,9]{Pesin04}. The relevant notion to extend the proof is \emph{accessibility} which in our setting can be described as follows.

Let $\sim$ be the equivalence relation on the space $\calS$ of type-preserving embeddings $S \to X$ of coinciding on a subsector of $S$. Among the two equivalence relations in the remark at the end of Section~\ref{sec:boundaries} the more inclusive one defines $\Delta$ and this is the more restictive one. Let $\hat{\Delta} = \calS/\mathop{\sim}$: it admits a projection to $\Delta$ and is locally $\Delta \times A$. Let $\hat{\pi}_\pm \colon \Fregular \to \hat{\Delta}$ be defined by taking $\iota$ to the class of $\iota|_S$ respectively $\iota|_{-S} \circ -$. Note that two elements of $\Fregular$ coincide iff their images under $\hat{\pi}_+$ and $\hat{\pi}_-$ coincide. In other words $\hat{\pi}_\pm$ are the quotient maps modulo the relations $\sim_+$ of coinciding on a subsector of $S$ and $\sim_-$ of coinciding on a subsector of $-S$. With these notions the dynamical system $\Gamma \times B \curvearrowright \Fregular$ is accessible if the following lemma is applicable.

\begin{lemma}\label{lem:accessibility}
  The equivalence relation generated by $\sim_+$ and $\sim_-$ contains the orbit relation of the affine Coxeter group $\tilde{W}$. As a consequence if $f \colon \Fregular \to \C$ is constant on fibers of $\hat{\pi}_\pm$ then it is constant.
\end{lemma}

\begin{proof}
  Let $\iota \in \Fregular$. Let $H^+,H^- \subseteq \Sregular$ be complementary half-apartments separated by a wall $H = H^+ \cap H^-$. Since $X$ is thick there is a $\iota_1$ that coincides with $\iota$ precisely on $H^-$. Further there is a $\iota_2$ that coincides with $\iota_1$ on $H^+$ and that takes $H^-$ to $\iota(H^+)$. Finally there is a $\iota_3$ that coincides with $\iota_2$ on $H^-$ and whose image coincides with that of $\iota$. Then $\iota_3$ is the result of composing $\iota$ with the reflection in $H$. Since these reflections generate $\tilde{W}$ this proves the first claim. The second claim follows from transitivity of $A$ on fibers of the projection $\Fregular \to \Deltaop$.
\end{proof}

Now, the Hopf argument does not ensure that the function is actually constant on $\sim_+$ and $\sim_-$ but only essentially so and hence the system is only essentially accessible.

Putting the measure-theoretic subtleties back in we get the desired extension of Theorem~\ref{thm:CartanErgodic}.
Recall that an element $t\in A$ is \emph{regular} if it does not leave a wall of $\Sregular$ invariant.

\begin{theorem}\label{thm:regular_element_ergodic}
    If $t\in A$ is a regular element, then the action of $t$ on $\Gamma\backslash \Fregular$ is ergodic.
\end{theorem}
  
\begin{proof}
  We choose a continuous section $s \colon \Deltaop \to \Fregular$ defining an isomorphism of measure spaces $\Fregular \cong A \times \Deltaop$ by Lemma~\ref{lem:F/A}. This can be achieved as follows. Fix $\iota_0 \in \Fregular$ and enumerate the chambers of $\Sregular$ as $\{c_n\mid n\in\N\}$. Then for $n \in \N$ on $U_{c_n}(\iota_0) \setminus \bigcup_{i < n} U_{c_i}(\iota_0)$ define $s((\iota(S(\infty)),\iota(-S(\infty))) = \iota$.

  We additionally want $s$ to be $\tilde{W}$-invariant in the sense that if $\iota \in \Fregular$ corresponds to $(a,\iota(S(\infty)),\iota(-S(\infty))$ and $w \in \tilde{W}$ then $w.\iota$ corresponds to $(w.a,\iota(w^{-1}S(\infty)),\iota(-w^{-1}S(\infty)))$. This can be achieved by replacing each $U_{c_n}(\iota_0)$ by $\bigcup_{w \in W} U_{c_n}(w.\iota_0)$ in the construction above. Note that $U_{c_n}(\iota_0) \cap U_{c_n}(w.\iota_0) = \emptyset$ for $w \in \tilde{W} \setminus \{1\}$. Since every chamber of $\Delta$ is opposite some chamber of $\iota(\partial \Sregular)$ this also ensures that $s$ is defined everywhere (the definition according to the previous paragraph is defined outside a null-set).

  Let $t \in A$ be regular and let $f \in L^\infty(\Fregular)^{\Gamma \times \gen{t}}$. If $w \in W$ and $w.f$ is essentially constant then clearly $f$ is essentially constant, so we may act by $w$ and thus assume that $t$ pushes $S$ into itself, i.e.\ that $t.S$ lies in the interior of $S$. We want to apply Theorem~\ref{thm:Hopf} with $T = \gen{t}$ and the projections $\hat{\pi}_\pm \colon Y = \gen{t}\backslash \Fregular \to \hat{\Delta} = W_{\pm}$. Note that under the identification $\Fregular \cong A \times \Deltaop \cong \Delta \times A \times \Delta$ as measure-class spaces, $\hat{\pi}_+$ corresponds to projection onto the first two factors (say) and $\hat{\pi}_-$ to projection onto the last two factors. Let us denote by $s_A$ the resulting map $\Fregular \to A$.

  If $\hat{\pi}_-(\iota) = \hat{\pi}_-(\iota')$ then there is a subsector $S'$ of $-S$ such that $\iota|_{S'} = \iota'|_{S'}$. From this it follows that the images of $(t^n\iota)_n$ and $(t^n\iota')_n$ in $\Gamma \backslash \Fregular$ are proximal. Similarly $(t^{-n}\iota)_n$ and $(t^{-n}\iota')_n$ are proximal if $\hat{\pi}_+(\iota) = \hat{\pi}_+(\iota')$.

  From the previous discussion we see that $\hat{\pi}_+^*(\hat{\Delta}) \cap \hat{\pi}_-^*(\hat{\Delta}) \subseteq  L^\infty(A)$. Note that $A$ is discrete so in order to show that $f(a) = f(a')$ it suffices to show that the measure of $(\iota,\iota') \in s_A^{-1}(a) \times s_A^{-1}(a')$ with $f(\iota) = f(\iota')$ is positive.

  It remains to see that $\hat{\pi}_+^*(L^\infty(\hat{\Delta})) \cap \hat{\pi}_-^*(L^\infty(\hat{\Delta}))$ consists of constants.
  Let $H^+,H^- \subseteq \Sregular$ be complementary half-apartments that meet in the wall $H = H^+ \cap H^-$. Let $r \in \tilde{W}$ be the reflection in $H$. Let $C_+ \subseteq H_+$ be a finite subcomplex and let $C_- = r.C_+ \subseteq H^-$. The set of $(\iota_1,\iota_2,\iota_3) \in \Fregular^3$ such that $\iota_1|_{C_+} = \iota_2|_{C_+}$, $\iota_2|_{C_-} = \iota_3|_{C_-}$, $\iota_3|_{C_+} = \iota_1|_{C_-} \circ r$ is open and therefore of positive measure. By adding a further open constraint (for instance $\iota_1|_{c_0} = \iota_0|_{c_0}$, $\iota_3|_{c_0} = \iota_0|_{c_0} \circ r$), we can ensure that $s_A$ is constant on possible $\iota_1$, on possible $\iota_3$, and that it commutes with $r$. From this we conclude that for $a \in A$ we have that $f(a) = f(r.a)$. Since these reflections $r$ generate $\tilde{W}$ we conclude that $\hat{\pi}_+^*(\hat{\Delta}) \cap \hat{\pi}_-^*(\hat{\Delta}) \subseteq L^\infty(A)^{\tilde{W}}$ which is reduced to constants.
\end{proof}

\subsection{Ergodicity of the singular flow} 

The main goal of the section is ergodicity of the singular Cartan flow:

\begin{theorem}\label{thm:SingularCartanErgodic}
  For $i \in \{1,2\}$ the $\Gamma$-action on $\Fsingular_i/A_i$ is ergodic.
\end{theorem}

We need two lemmas, one general and one about generic projectivities.

\begin{lemma}\label{lem:InvDenseSubgroup}
    Let $G$ be a locally compact second countable group and let $G'<G$ be a dense subgroup. If $f\in L^\infty(G)$ is $G'$-invariant, then it is $G$-invariant.
\end{lemma}

\begin{proof}
  Let $g\in G$ and let $(g_n)_n$ be a sequence of elements of $G'$ converging to $G$. Let $\varphi$ be a continuous $L^1$-function on $G$. Then, using Lebesgue's Dominated Convergence Theorem \cite[Theorem~2.11]{LeGall}, we have
    \begin{align*}
 \int_{G} f(x)\varphi(x) dx &= \int_{G} (g_n.f)(x)\varphi(x) dx
 = \int_{G} f(x)\varphi(g_n x) dx\\
 &\xrightarrow[n\to +\infty]{} \int_{G}f(x)\varphi(gx) dx=\int_{G} (g.f)(x)\varphi(x) dx
\end{align*}
(where all the integrals are taken with respect to the Haar measure on $G$). Since this is valid for every  $\varphi$, it follows that $f=g.f$, so that $f$ is $G$-invariant.
\end{proof}

\begin{lemma}\label{lem:FullMeasureDenseProj}
  Let $D\subset \Deltaopv[i]\cup \Deltaopv[\bar{i}]$ be a full measure subset (with respect to the measure $\mu_{\Deltaopv[i]} + \mu_{\Deltaopv[\bar{i}]}$). Consider the groupoid $\catname D$ generated by perspectivities $[u,v]$ for $(u,v)\in D$, and the action of $\Gamma$.

  Then for almost every $u,v \in \Deltav[i] \cup \Deltav[\bar{i}]$ the set $\catname D(u,v)$ of morphisms from $u$ to $v$ in $\catname D$ is dense in $\EProjGrpd(u,v)$. In particular, the isotropy group of $\catname D$ in $u$ is dense in $\EProjGrp_u$.
\end{lemma}

\begin{proof}
  Note first that $\mu_{\Deltav[i]}^o \times \mu_{\Deltav[\bar{i}]}^o$ and $\mu_{\Deltaopv[i]}$ define the same measure-class by the definition of $\mu_{\Deltaopv[i]}$.
  Thus Fubini implies that if $D$ has full measure then for almost every $u \in \Deltav[i]$ the set $\{v \in \Deltav[\bar{i}] \mid (u,v) \in D\}$ has full measure in $\Deltav[\bar{i}]$.
    
   Note next that since $\catname D$ contains the action of $\Gamma$, it does not change if we replace $D$ by the union over its $\Gamma$-orbit, so we may assume $D$ to be $\Gamma$-invariant. Since $\gamma \circ [u,v]=[\gamma u,\gamma v]$, it follows that any element of $\catname D$ can be written as one element of $\Gamma$ composed with a sequence of perspectivities. The same argument holds for $\EProjGrpd$, so that it suffices to prove that projectivities in $M_D(u)$ are dense in the subgroup of $\ProjGrp_u$ of $\EProjGrp_u$ consisting of projectivities.

   For $n\geq 2$ recall from Lemma~\ref{lem:projectivity_multiplication_continuous} that we write $\Delta_{i}^{n,\op}$  for the set of sequences $(u_1,u_2,\dots,u_n)$ with $u_1\in\Delta_i$ and such that $u_i$ is opposite $u_{i+1}$. Write $\Delta_i^{n,\op}(u,v)$ for the sequences with $u_1 = u$ and $u_n = v$. Let $\Delta_i^{n,\op}|D$, respectively $\Delta_i^{n,\op}|D(u,v)$, be the subset of sequences for which $(u_k,u_{k+1}) \in D$ for every $k$.

   Since $D$ has full measure, by Fubini, for almost every $u,v$ the set $\Delta_i^{n,\op}|D(u,v)$ has full measure in $\Delta_i^{n,\op}(u,v)$. Since $\mu_{\Delta_i}$ and $\mu_{\Delta_{\overline i}}$ have full support in $\Delta_i$, the product measure on the $(n-2)$-fold alternating product $\Deltav[\bar{i}]$ and $\Deltav[i]$ also has full support. Hence $\Delta_i^{n,\op}(u,v)$ is dense in this product. Furthermore, $\Delta_i^{n,\op}(u,v)$ is an open subset of this product, so $\Delta_i^{n,\op}|D(u,v)$ is also dense in $\Delta_i^{n,\op}(u,v)$. 

   By Lemma~\ref{lem:projectivity_multiplication_continuous} the map $\Psi \colon \Delta_{\overline i}^{n,\op}(u,v)\to P_n(u,v), (u_1,\dots,u_n) \mapsto [u_1,\dots,u_n]$ is continuous, and it is surjective by definition. Hence $\Psi(\Delta_i^{n,\op}|D(u,v))$ is dense in $P_n(u,v)$.
\end{proof}

Recall from \eqref{eq:extended_projectivity_group_extension} that $M_i$ acts on $T_i$ through the quotient $Q_i$, and that the kernel is $A_i$; it is an infinite cyclic group which acts on the $\R$-factor of $\Ssingular_i$. The space $\Fsingular_i/A_i$ is equipped with the pushforward measure class from $\Fsingular_i$; the group $\Gamma\times M_i/A_i$ acts on $\Fsingular_i/A_i$ by preserving this measure class.

Also recall from Section~\ref{sec:extended_projectivity_groupoid} that we defined space $\EDeltav[i]$ as a set of pairs $(u,f)$ where $u\in \Delta_i$ and $f\in \Isom(T_i,T_u)$, a map $\overline \sigma_i:\Fsingular_i\to \overline \Delta_i$ (which factors obviously through $\Fsingular_i/A_i$), and the measure class on $\overline \Delta_i$ which is the pushforward of the measure class of $\mu_{\Fsingular_i}$ by $\overline\sigma_+$.

\begin{proof}[Proof of Theorem~\ref{thm:SingularCartanErgodic}]

  Our goal is to apply Theorem~\ref{thm:Hopf} with the following diagram:

\begin{align} \label{diag:singcartan}
\xymatrix{
& \Fsingular_i \ar[dl]_\phi\ar[dr]^\psi& &\\
\Fsingular_i/\Gamma & & \Fsingular_i/A_i   \ar[dl]_{\overline \sigma_+}\ar[dr]^{\overline \sigma_{-}}& \\
&  \overline \Delta_i &  & \overline \Delta_{\overline i}
}
\end{align}
(where $\phi$, $\psi$ are the obvious quotient maps).

First we need to check the proximality condition. Let $w,w'\in\Fsingular_i/A_i$ be such that $\overline \sigma_+(w)=\overline\sigma_{+}(w')$. We choose lifts of $w,w'$ under $\psi$ and denote them by $w,w'$ as well. We claim that, up to replacing the representatives of $w,w'$ (corresponding to $t,t'$ in Theorem~\ref{thm:Hopf}), there exists an $s\in A_i$ such that the sequences $(s^n w)_{n\in\N}$ and $(s^n w')_{n\in\N}$ are proximal.

Indeed, by definition if $\overline\sigma_+(w)=\overline\sigma_+(w')$ then for every $x\in Y_i$, $w$ and $w'$ agree on some subray of the ray from $x$ to $\xi_i$ (as sets and adjusting representatives even as maps). If we fix a ball $B$ in $Y_i$, it follows that there is $T>0$ such that for every $x\in B$, if $\ell_x$ is the line from $x$ to $\xi_i$, then $w=w'$ in restriction to $\ell_x([T,+\infty))$. Now take $s\in A_i$ to be a translation in the direction of $\xi_i$. It follows that for $n$ large enough, $s^n w$ and $s^n w'$ agree on $B$.

 Since $B$ is arbitrary, this means that the sequences $(s^n.w)_{n\in\N}$ and $(s^n.w')_{n\in\N}$ are proximal. The argument is similar for $\overline \sigma_-$.

 The last hypothesis of Theorem~\ref{thm:Hopf} that we need to check is that
 \[
   \left((\overline\sigma_+)^*L^\infty(\overline\Delta_i)\cap (\overline\sigma_-)^*L^\infty(\overline\Delta_{\overline i})\right)^\Gamma
 \]
 is reduced to the constant functions. Let $f$ be a function in that set, that is, a $\Gamma$-invariant function which is in the pullback of both $\overline\Delta_i$ and $\overline\Delta_{\overline i}$. We claim that $f$ is in fact in the pullback of the factor map $\sigma_+\times\sigma_-:\Fsingular_i/A_i\to \Delta_i\times \Delta_{\overline i}$.
Since $\Delta_i\times\Delta_{\overline i}$ (with the product measure class) is a factor of $\Fregular/A$, by Theorem \ref{thm:CartanErgodic} the $\Gamma$-action on $\Delta_i\times\Delta_{\overline i}$ is ergodic, so that it will follow that $f$ is constant.  

The claim is roughly proved as follows. Let $f_+\in L^\infty(\overline \Delta_i)$, $f_-\in L^\infty(\overline \Delta_{\overline i})$ such that $(\overline\sigma_\pm)^*f_\pm=f$. 
In other words, $f(\iota)=f_+(\overline\sigma_+(\iota))=f_-(\overline\sigma_-(\iota))$ for almost every $\iota\in\Fsingular_i$. The general idea is that $f$ is then invariant by the positive equivalence relation $\sim_+$, the negative equivalence relation $\sim_-$, and is $\Gamma$-invariant. By definition of $\Fsingular_i$ it follows that $f_+$ is constant on each fiber $\overline \Delta_i\to \Delta_i$ (and similarly for $f_-$), which would prove the claim.

However, the actual proof is a bit more subtle, as the $\sim_+$ and $\sim_-$ invariance is only true almost everywhere. The actual argument is exactly the same as in \cite[Theorem 7.3]{BCL}, and we will reproduce it here.

From now on we fix actual $\Gamma$-invariant functions $f$, $f_+$, $f_-$ representing their respective classes. Let $E\subset\Fsingular_i$ be the full measure subset on which $f_+\circ \sigma_+$ and $f_-\circ\sigma_-$ agree. We get that when $\iota,\iota'\in E$ are such that $\iota\sim_+\iota'$ then $f(\iota)=f(\iota')$.

For $u\in\Delta_i$ let $f_u$ be the restriction of $f_+$ to the fiber over $u$ in $\overline \Delta_i$. Similarly for $v\in\Delta_{\overline i}$ let $f_v$ be the restriction of $f_-$ to the fiber over $v$ in $\overline \Delta_{\overline i}$. By Proposition~\ref{prop:measure_singular_to_regular} and Lemma \ref{lem:FregMeasPres} we have that $f_u$ is defined for $\mu_{\Delta_i}$-almost every $u$ and,  for almost every $\iota\in\Fsingular$, 
$f(\iota)=f_u(\mathrm{pr}_2(\overline\sigma_+(\iota)))$.
Furthermore by Lemma \ref{lem:MeasDeltaBar}, $f_u$ is identified to a function in $L^\infty(M_i/A_i)$.

Let $D=\sigma(F)\subset \Deltaopv[i]$. Since $f$ is also $\Gamma$-invariant, it follows that if $(u,v)\in D$ then $f_v\circ [u,v]=f_u$ almost everywhere. Hence for every extended projectivity $m=[\gamma;u_1,\dots,u_n]$ such that $(u_i,u_{i+1})\in D$ we have $f_u\circ m =f_u$.
Hence $f_u$ is almost everywhere invariant by the group $M_D$ generated by all these extended projectivities. By Lemma \ref{lem:FullMeasureDenseProj} the group $M_D$ is dense in $p_T(M_i)$. 
 By Lemma \ref{lem:InvDenseSubgroup} it is in fact invariant by $p_T(M_i)$.

Since $f$ is also $A_i$-invariant, it follows that $f$ 
is in fact $M_i$-invariant. Hence $f$ factors through $\Fsingular_i/M_i$, which is isomorphic to $\Deltaopv[i]$ (Proposition \ref{prop:Y/M}). By Theorem \ref{thm:CartanErgodic}, the action of $\Gamma$ on $\Delta^\op$, hence on $\Deltaopv[i]$, is ergodic. Hence $f$ is a pullback from a constant function, so it is constant.
\end{proof}

This allows us to conclude the proof of Theorem \ref{thm:introergodic}.

\begin{corollary}
    Let $t\in A$ be a singular element. Then the action of $t$ on $\Gamma \backslash\Fregular$ is ergodic.
\end{corollary}

\begin{proof}
    Consider the map $\Fsingular_i\to \Fregular$ as in Proposition \ref{prop:measure_singular_to_regular}. This maps is $\Gamma$-equivariant so descends to a map $\Gamma\backslash \Fsingular_i\to \Gamma\backslash \Fregular$. Choosing $i$ depending on $t$, the action of $t$ on $\Fregular$ lifts by this map to an action of a finite index subgroup of $A_i$.

    Let $\sim_{\pm}$ be the equivalence relation on $\Fsingular_i$ of coinciding on a convex subcomplex of $\Ssingular_i$ that contains $T_{\xi_i}$ respectively $T_{\zeta_i}$ in its boundary. And let $\hat{\pi}_\pm \colon \Fsingular_i \to \Fsingular_i/\sim_\pm$ be the projection. Applying Theorem~\ref{thm:Hopf} as in the proof of Theorem~\ref{thm:SingularCartanErgodic} we are left to showing that $\hat{\pi}_+^*(\Fsingular_i/\mathop{\sim_+})^\Gamma \cap \hat{\pi}_-^*(\Fsingular_i/\mathop{\sim_-})^\Gamma \subseteq L^\infty(A_i)^\Gamma$ is reduced to constants. As in the proof of Theorem~\ref{thm:regular_element_ergodic} one sees that $f \in \hat{\pi}_+^*(\Fsingular_i/\mathop{\sim_+})^\Gamma \cap \hat{\pi}_-^*(\Fsingular_i/\mathop{\sim_-})^\Gamma$ is $\tilde{W}$-invariant hence constant. Since the action of $t$ on $\Gamma\backslash \Fregular$ is a factor of that space, it is also ergodic.
\end{proof}

\section{Contracting sequences}\label{sec:contracting_sequences}

The following extends \cite[Proposition~4.8]{BFL}.

\begin{theorem}\label{thm:contracting}
  Let $i \in \{1,2\}$. For almost every pair of opposite vertices $(\xi,\zeta) \in \Delta^\op_i$ there exists a sequence $(\gamma_n)_n \in \Gamma^\N$ such that $\gamma_n$ converges to $\proj_\xi \circ \proj_\zeta$ in the compact-open topology on $C(\Opp(\zeta),\Opp(\zeta))$.

\end{theorem}

Instead of $\Opp(\zeta)$ we think about $\Fsingular$ first. We apply the following classical fact to $\Gamma \backslash \Fsingular$.

\begin{theorem}[Ergodic recurrence]
  Let $(X,\calF,\mu)$ be a probability space, and let $T \colon X \to X$ be measurable and measure-preserving. Let $A \subseteq X$ be measurable with $\mu(A) > 0$. Assume that $(X,T)$ is ergodic. Then for almost all $x \in X$ there exist infinitely many $n$ with $T^n(x) \in A$.
\end{theorem}

\begin{lemma}\label{lem:ergodicrecurrence}
 Let $i \in \{1,2\}$. For almost every pair of opposite vertices $(\xi,\zeta) \in \Delta^\op_i$, every finite subcomplex $F$ of the wall-space $Y_{\xi,\zeta}$ and every $N$ there exists $n>N$ and a $\gamma\in \Gamma$ which translates every point of $F$ on the line from $\zeta$ to $\xi$ by $n$.
\end{lemma}

\begin{proof}
  Let $i \in \{1,2\}$ . Let $F \subseteq Y_i$ be a finite subcomplex, so that for $h \in \Fsingular_i$ the set $U_F(h) = \{h' \in \Fsingular \mid h'|_F = h_F\}$ is an open subset, therefore of positive $\mu_{\Fsingular_i}$-measure by Lemma \ref{lem:fullsupport}. 

  The action of $A_i$ on $\Fsingular_i/\Gamma$ is ergodic by Theorem \ref{thm:SingularCartanErgodic}. Hence we can apply the Ergodic Recurrence with $T$ defined by $T(h)(x) = h(\weight[i]^{-1}x)$ and conclude that for almost every $y\in\Fsingular$ the set of all $n$ such that $T^n y\in \Gamma U_F(h)$ is of full measure. Let $E_{F,h}$ be this full measure subset of $\Fsingular$. Note that for a given $F$ there are countably many distinct sets $U_F(h)$, as there are countably many embeddings of $F$ into $X$.
  Let $E$ be the intersection of all $E_{F,h}$'s, which is therefore of full measure.

  For every $h\in E$ it follows that for every finite $F \subseteq \Ssingular_i$ there are infinitely many $n$ such that $T^nh|_F$ and $h|_F$ coincide up to an element of $\gamma$, meaning that there is a $\gamma_n \in \Gamma$ such that $T^nh|_F = \gamma_nh_F$, i.e.\ $\gamma_nh(f) = h(\weight[i]^{-n}f)$ for $f \in F$. In other words the pair $(\sigma_-(h),\sigma_+(h))$ satisfies the conclusion of the Lemma. Since $(\sigma_-,\sigma_+):\Fsingular\to \Delta^{op}_i$ is measure-class preserving we get the result.  
\end{proof}

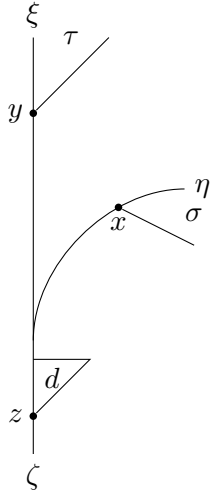
\begin{figure}
  \begin{tikzpicture}
    \node[above] (xi) at (0,5) {$\xi$};
    \node[below] (zeta) at (0,-.5) {$\zeta$};
    \node[right] (eta) at (2,3) {$\eta$};
    \draw (zeta) -- (xi);
    \draw (0,1) .. controls (0,2) and (1,3) .. (eta) node[pos=.7,dot, inner sep=1pt] (x) {};
    \node[below] at (x) {$x$};
    \node[dot, inner sep=1pt] (z) at (0,0) {};
    \node[left] at (z) {$z$};
    \node[dot, inner sep=1pt] (y) at (0,4) {};
    \node[left] at (y) {$y$};
    \draw (y) -- ++(1,1);
    \draw (x) -- ++(1,-.5);
    \draw (z) -- ++(.75,.75) -- ++(-.75,0);
    \node at ($(x) + (1,-.1)$) {$\sigma$};
    \node at ($(y) + (.5,1)$) {$\tau$};
    \node at ($(z) + (.25,.5)$) {$d$};
  \end{tikzpicture}
  \caption{The setup in the proof of Theorem~\ref{thm:contracting}.}
  \label{fig:contracting}
\end{figure}

\begin{proof}[Proof of Theorem~\ref{thm:contracting}]
  Let $E \subseteq \Deltaopv[i]$ be the full measure set from  Lemma~\ref{lem:ergodicrecurrence} and for $(\xi,\zeta) \in E$ let $(\gamma_n)_n \in \Gamma^\N$ be as in the lemma.
  Let $\sigma \in \Opp(\zeta)$ and put $\tau = \proj_\xi \circ \proj_\zeta(\sigma)$. Let $x,y$ be special vertices such that
  \[
    \proj_\xi \circ \proj_\zeta(U^x(\sigma)) \subseteq U^y(\tau)
  \]
  (recall from \eqref{eq:Uo} that $U^x(\sigma)$ is the space of chambers that have the same projection to $x$ as $\sigma$). Then the set
  \[
    V = \{\iota \in C(\Opp(\zeta),\Opp(\zeta)) \mid \iota(U^x(\sigma)) \subseteq U^y(\tau)\}
  \]
  is an open set in $C(\Opp(\zeta),\Opp(\zeta))$ and we want to show that $\gamma_n \in V$ for large enough $n$.

Note that $V$ is smaller the closer $y$ is to $\tau$ and the further $x$ is from $\sigma$. Nonetheless is suffices to consider $x$ arbitrarily close to $\sigma$ as long we consider $y$ arbitrarily close to $\tau$ as well: let $c_0 \subseteq \Sregular$ be the chamber at the tip of $S$, let $\iota_x$ take $0$ to $x$ and $S(\infty)$ to $\sigma$ so that $U^x(\sigma)$ consists of the $\sigma_+(\iota)$ with $\iota \in U_{c_0}(\iota_x)$. Similarly $U^y(\tau)$ consists of the $\sigma_+(\iota)$ with $\iota \in U_{c_0}(\iota_y)$ when $\iota_y$ takes $0$ to $y$ and $S(\infty)$ to $\tau$. If $c_1 \subseteq S$ is a chamber then $U^x(\sigma)$ is a (disjoint) union of finitely many $\sigma_+(U_{c_0 \cup c_1}(\iota_x'))$ (of course $\sigma_+(U_{c_0 \cup c_1}(\iota_x')) = \sigma_+(U_{c_1}(\iota_x'))$, we include $c_0$ only to ensure $\iota_x'(0) = x$). Any $\varphi \in V$ takes $\sigma_+(U_{c_0 \cup c_1}(\iota_x'))$ into some $\sigma_+(U_{c_0 \cup c_1}(\iota_y'))$, and fixing the finitely many representatives $\iota_x'$ and $\iota_y'$ taking every $\sigma_+(U_{c_0 \cup c_1}(\iota_x'))$ into some $\sigma_+(U_{c_0 \cup c_1}(\iota_y'))$ is a characterization of lying in $V$. Consequently if we require that $\varphi(\sigma_+(U_{c_0 \cup c_1})(\iota_x')) \subseteq \sigma_+(U_{c_0 \cup c_1})(\iota_y)$ for every $\iota_x'$, which is an intersection of basic open sets, we have defined an open sub-neighborhood of $V$ in terms of smaller neighborhoods of $\sigma$ as claimed.

  Let $\eta$ be the vertex of type $i$ of $\sigma$. Let $\ell_\eta$ be the line from $\zeta$ to $\eta$ that contains a subray of $[x,\eta)$. By the remark in the previous paragraph, we may arrange $x \in \ell_\eta$ by adjusting $y$ accordingly. Let $\ell_\xi$ be the line from $\zeta$ to $\xi$ that coincides with $\ell_\eta$ on an infinite ray to $\zeta$ (see Figure~\ref{fig:contracting}). Then the assumptions on $\sigma$, $\tau$, $x$, $y$ guarantee that an infinite ray of $\ell_\xi$ is contained in $\conv(y,\tau)$. Let $z$ be a special vertex that is a relatively interior point of  $\ell_\eta \cap \ell_\xi$; this ensures that the direction from $z$ to $y$ (which is also the direction to $\xi$) coincides with the direction to $x$ (which is also the direction to $\eta$). Then $U^z(\sigma) = U^z(\tau)$ is a compact open set containing $U^x(\sigma)$ and satisfying $\proj_\xi \circ \proj_\zeta(U^z(\sigma)) \subseteq U^y(\tau)$. Replacing $x$ by $z$ leads to a smaller set $V$ thus to a stronger statement. This shows that there is no loss in assuming $U^x(\sigma) = U^x(\tau)$, which we do.

  Let $d = \pr_x \tau$ so that $U^x(\tau) = \pr_x^{-1}(d)$.
  Take $h \in Y_{\xi,\zeta}$ arbitrary and let $c = h^{-1}(d)$. Let $N$ be sufficiently large such that $h(\weight[i]^n c) \subseteq U^y(\tau)$ for $n \ge N$. By choice of the sequence there is an $n \ge N$ and a $\gamma_n \in \Gamma$ with
  \[
    \gamma_n.d = \gamma_n.h(c) = h(\weight[i]^nc) \in U^y(\tau)
  \]
  Thus $U^x(\sigma) = U^x(\tau)$ is taken by $\gamma_n^{-1}$ to
  \[
    \gamma_nU^x(\tau) = U^{\gamma_nx}(\tau) \subseteq U^y(\tau).
  \]
  as desired.
\end{proof}

\section{Martingale convergence}\label{sec:martingale_convergence}

We need a uniform way to to describe neighborhoods $U_n(\xi)$ for varying $\xi \in \Delta_i$. To do so we pick a base vertex $o \in X$ of type $0$ and put
\[
  U_n(\xi) = U^0_{\ell_i(n)}(\xi)
\]
in the notation of Section~\ref{sec:basic_open_sets}. That is, if $\rho$ is the geodesic ray from $o$ to $\xi$, then
\[
  U_n(\xi) = \{\rho'(\infty) \mid \rho|_{[0,n]} = \rho'|_{[0,n]}\}.
\]

Our goal for this section is to prove the following statement, which can be seen as a form of Lebesgue differentiation.

\begin{theorem}\label{thm:martingale_convergence_fV}
  Let $i \in \{1,2\}$, let $f \colon \Deltaop \to \C$ be measurable and bounded, and let $\xi_0 \in \Deltav[i]$ and let $V \subseteq \Res(\xi_0)$ be compact open. For almost every $(\xi,\zeta) \in \Deltaopv[i]$ we have that
  \[
    \fint_{U_n(\xi)} \int_{V_{\eta,\zeta}} f(\sigma) d\mu_{\Res(\eta)}(\sigma) d\mu_{\Oppv[i](\zeta)}(\eta) \xrightarrow{n \to \infty}
    \int_{V_{\xi,\zeta}} f(\sigma) d\mu_{\Res(\xi)}(\sigma).
  \]

  where $V_{\xi,\zeta} = \proj_\xi \circ \proj_\zeta(V)$.
\end{theorem}

Recall from \eqref{eq:zeta_xi} that
\begin{align*}
  \int_{V_{\xi,\zeta}} f(\sigma) d\mu_{\Res(\xi)}(\sigma) &= \int_{\pr_{\zeta}(V)} f(\pr_\xi(\sigma)) d\mu_{\Res(\zeta)}(\sigma)\\& = \int_{V} f(\pr_\xi \circ \pr_\zeta) d\mu_{\Res(\zeta)}(\sigma).
\end{align*}

We first prove the following version in which $\zeta$ is fixed.

\begin{proposition}\label{prop:martingale_convergence_fV}
  Let $i \in \{1,2\}$ and $\zeta \in \Delta_{\bar{i}}$.
  Let $f \in L^1(\Opp(\zeta))$ and $V \subseteq \Res(\zeta)$ basic open and define $f_V \in L^1(\Oppv[i](\zeta))$ by
  \[
    f_V(\xi) = \int_V f(\pr_\xi \sigma) d\mu_{\Res(\zeta)}(\sigma).
  \]
  For $\xi \in \Oppv[i](\zeta)$ let $U_n(\xi)$ be as above and define
  \[
    f_{n,V}(\xi) = \fint_{U_n(\xi)} f_V(\eta) d\mu_{\Oppv[i](\zeta)}(\eta).
  \]
  Then $f_{n,V}$ converges to $f_V$ almost surely and in $L^1$ as $n \to \infty$. In particular, for almost every $\xi \in \Oppv[i](\zeta)$ we have
  \begin{equation}\label{eq:martingale}
    \int_V \fint_{U_n(\xi)} f(\pr_\eta \sigma) d\mu_{\Oppv[i](\zeta)}(\eta) d\mu_{\Res(\zeta)}(\sigma) \to \int_V f(\pr_\xi \sigma) d\mu_{\Res(\zeta)}(\sigma)
  \end{equation}
  for $n \to \infty$.
\end{proposition}

The proof is an application of Martingale convergence which we state combining Theorem~12.17 and Theorem~11.3 in \cite{LeGall}.

\begin{theorem}[Martingale convergence]\label{thm:martingale_convergence}
  Let $(X,\calF,\mu)$ be a probability space and let $(\calF_n)_{n \in \N}$ be an ascending sequence of sub-$\sigma$-algebras of $\calF$. Let $f \in L^1(X,\calF,\mu)$ and define $f_n \in L^1(X,\calF_n,\mu)$ by $f_n(x) = \E(f \mid \calF_n)(x)$.

  Then $(f_n)_{n \in \N}$ converges $\mu$-almost everywhere and in $L^1$ to $f_\infty \defeq \E(f \mid \calF_\infty) \in L^1(X,\calF_\infty,\mu)$ where $\calF_\infty = \gen{\calF_n,n \in \N}$.
\end{theorem}

Note that the theorem assumes $X$ to be a probability space which $\Opp(\zeta)$ is not. However, it is a union of probability spaces $U^o(\sigma)$ and in all applications the $\sigma$-algebra $\calF_n$ for large enough $n$ will contain the sets $U^o(\sigma)$. Thus we may apply martingale convergence piece by piece.

\begin{proof}[Proof of Proposition~\ref{prop:martingale_convergence_fV}]
    We identify $\Opp(\zeta)$ with $\Oppv[i](\zeta)\times \Res(\zeta)$ as in Proposition \ref{prop:disintegration_opp}. 
    For fixed $n \in \Z$ let $\calF_n$ be the $\sigma$-algebra generated by the sets $U_{n}(\xi)\times V$ for $\xi\in \Opp(\zeta)$,  and let $\calF_\infty$ be the $\sigma$-algebra generated by all $U_{n}(\xi)\times V$ with $n \in \Z$ and $\xi \in \Opp(\zeta)$, in other words by all the $U\times V$ for $U\subset \Oppv[i](\zeta)$ measurable.
  

  Note that the sets $U_{n}(\xi)\times V$ are minimal in $\calF_n$, and are of positive measure. Therefore we get
  $$\E(f\mid \calF_n)(\xi) = \E(f \mid U_{n}(\xi)\times V)=\frac1{\mu_{\Res(\zeta)}(V)}f_{n,V}(\xi).$$ 
  
  By Theorem \ref{thm:martingale_convergence}  there is an $\calF_\infty$-measurable function $f_\infty$ such that $(f_{n,V})_n$ converges to $f_\infty$ almost surely and in $L^1$, and $f_\infty$ is characterized (up to measure 0) by the equality $\E(f \mid U) = \E(f_\infty \mid U)$ for $U \in \calF_\infty$.

  Let $\overline f$ defined by $\overline f(\eta):=\frac{1}{\mu_{\Res(\zeta)}(V)}\int_V f(\proj_\eta\sigma) d\mu_{\Res(\zeta)}(\sigma)$. Then $\overline f$ is $\calF_\infty$-measurable. Furthermore by Proposition~\ref{prop:disintegration_opp} we get that for every $U\subset \Oppv[i](\zeta)$ of positive measure we have 
  \begin{align*}\E(f\mid U\times V) &= \fint_U \fint_V f(\proj_\eta\sigma)d\mu_{\Res(\zeta)}(\sigma)d\mu_{\Oppv[i](\zeta)}(\eta) \\
  &= \fint_U \overline f(\eta) d\mu_{\Oppv[i](\eta)} = \E(\overline f\mid U\times V). 
  \end{align*}
  
  Since the sets $U\times V$, for $U\subset \Oppv[i](\zeta)$ generate $\calF_\infty$ this proves that $\overline f=f_\infty$ almost everywhere.
\end{proof}

\begin{proof}[Proof of Theorem~\ref{thm:martingale_convergence_fV}]
  Note that in the context of the theorem $f$ is an honest map (not an equivalence class). Since it is measurable and bounded and the sets over which the integrals are taken are compact, the integrals exist. They are defined whenever $\zeta$ is opposite $\xi_0$ which is true on an open co-null set $E_0$ of $\Deltaop$. Let $E \subseteq E_0$ be the set on which the desired convergence takes place.

  We claim that it is measurable. In order to see this, let $(\xi,\zeta) \in E_0$ be arbitrary and let $f \in \Fsingular_i$ be such that $f(\ell_i(\infty)) = \xi$ and $f(\ell_i(-\infty)) = \zeta$. Let $B \subseteq \Ssingular_i$ be large enough so that $\pr_{\zeta}(f(B)) \supseteq \pr_\zeta V$. Consider
  \[
    U = \{(f'(\ell_i(\infty)),f'(\ell_i(-\infty))) : f' \in U_B(f)\}
  \]
  which is a direct product $U_+ \times U_-$.
  For $(\xi',\zeta') \in U$ the sets $V_{\xi',\zeta'}$ do not depend on $\zeta'$, in fact, these sets $\pr_{\zeta'} V$ and the sets $V_{\xi',\zeta'}$ are all mapped to each other under projection. Thus convergence does not depend on $\zeta'$ and the intersection of $U \cap E$ is a direct product $(U_+ \cap E_+) \times U_-$. The set $U_+ \setminus E_+$ has to be measurable (in fact, a null-set) by Proposition~\ref{prop:martingale_convergence_fV}, hence $U \cap E$ is measurable. This proves that every point of $E$ has a neighborhood $U$ such that $E\cap U$ is measurable. By separability of $E_0$, we see that $E$ is covered by a countable family of such neighborhoods, and therefore $E$ is measurable as a countable union of measurable sets. 

  Let $E_\zeta = \{\xi \mid (\xi,\zeta) \in E\}$. Then using Fubini and Proposition~\ref{prop:martingale_convergence_fV} we have that
  \[
    \int_{\Deltav[\bar{i}]} \int_{\Oppv[i](\zeta) \setminus E_\zeta} 1 d\mu_{\Oppv[i](\zeta)}(\xi) d\mu_{\Deltav[\bar{i}]}^o(\zeta) = 0.
  \]
  Since $\mu_{\Oppv[i](\xi)} \otimes \mu_{\Deltav[\bar{i}]}$ and $\mu_{\Deltaopv[i]}$ define the same measure class, we see that $\mu_{\Deltaopv[i]}(\Deltaopv[i] \setminus E) = 0$.
\end{proof}

\section{Weak-* Convergence}\label{sec:weak-stark_convergence}

The goal of this section is to prove the following extension of \cite[Theorem~4.20]{BFL}.

\begin{theorem}\label{thm:convergence}
  Let $f \in L^\infty(\Delta)$ and let $i \in \{1,2\}$. For almost every pair of opposite vertices $(\xi,\zeta) \in \Delta^\op_i$ there exists a sequence $(\gamma_n)_n \in \Gamma^\N$ such that $(\gamma_nf)_n$ converges to $x \mapsto f \circ \proj_\xi \circ \proj_\zeta (x)$ in weak-$*$ topology.
\end{theorem}

To prove it we will combine Theorem~\ref{thm:martingale_convergence_fV} and Theorem~\ref{thm:contracting}.

\begin{proof}[Proof of Theorem~\ref{thm:convergence}]
  First note that $L^\infty$ only depends on the measure class and that for $\zeta \in \Deltav[\bar{i}]$ the inclusion $\Opp(\zeta) \to \Delta$ is measure-class preserving by Proposition~\ref{prop:opp_locally_proportional} thus $L^\infty(\Delta)$ and $L^\infty(\Opp(\zeta))$ are isomorphic.
  For a sequence $(\gamma_n)_n \in \Gamma^n$ the sequence $(\gamma_n f)_{n\geq 0}$ is a bounded sequence in $L^\infty(\Opp(\zeta))$. In order for it to converge, by weak-* compactness of balls, it suffices to prove that every cluster value of this sequence is $f\circ \pr_\xi\circ \pr_\zeta$. Therefore we can start by taking a converging subsequence of $(\gamma_n f)_{n\geq 0}$, and to identify the limit it is sufficient to integrate it along a dense subset of $L^1(\Opp(\zeta))$. This justifies that it is sufficient to prove the convergence above for functions $\varphi$ that are characteristic functions $\chi_{W}$ where $W \subseteq \Opp(\zeta)$ are basic open:  the linear span of these is dense in $L^1(\Opp(\zeta))$. Thus we want to show that
  \begin{equation}\label{eq:weak-*-convergence}
    \int_{W} f(\gamma_n^{-1}.\sigma) d\mu(\sigma) \longrightarrow \int_W f \circ \proj_\xi \circ \proj_\zeta(\sigma) d\mu(\sigma)
  \end{equation}
  where we will take $W \subseteq \Opp(\zeta)$ and $\mu = \mu_{\Opp(\zeta)}$.

  Let $F \subseteq \Deltaopv[i]$ be a set of full measure on which Theorem~\ref{thm:contracting} holds. Fix $\xi_0 \in \Deltav[i]$ and let $E \subseteq \Deltaopv[i]$ be a set of full measure such that Theorem~\ref{thm:martingale_convergence_fV} for $V$ ranging over a countable basis for the topology of $\Res(\xi_0)$. We claim that the theorem holds for $(\xi,\zeta) \in E \cap F$, so take one and let $(\gamma_n)_n$ be the sequence whose existence is guaranteed by Theorem~\ref{thm:contracting}.

  Let $V \subseteq \Res(\xi_0)$ in the countable basis for which Theorem~\ref{thm:martingale_convergence_fV} holds and let $U \subseteq \Oppv[i](\zeta)$ be open. Let $W = \bigcup_{\xi \in U} \pr_\xi V$. In terms of the decomposition $\Opp(\zeta) \cong \Res(\zeta) \times \Oppv[i](\zeta)$ provided by Proposition~\ref{prop:disintegration_opp} the set $W$ is just $U \times V$. Therefore it suffices to show \eqref{eq:weak-*-convergence} of this form.
  
  Using that $\gamma_n \colon (W,\mu_W) \to (\gamma_n W,\mu_{\gamma_nW})$ is measure-preserving (cf.~Proposition~\ref{prop:locally_pro_uniform}) the left hand side of \eqref{eq:weak-*-convergence} is
  \begin{equation}\label{eq:weak_star_sequence_characteristic}
    \int_{W} \gamma_nf(\sigma) d\mu_{\Opp(\zeta)}(\sigma) = \frac{\mu(W)}{\mu(\gamma_nW)} \int_{\gamma_n W} f(\sigma) d\mu_{\Opp(\zeta)}(\sigma).
  \end{equation}
  
   By construction of the sequence $(\gamma_n)$, we know that for $n$ large enough, we have $\gamma_n V=V$, and we will consider only such $n$s.

  Note that for $\sigma\in \Res(\zeta)$ and $\eta\in \Oppv[i](\zeta)$ we have $\proj_\zeta\proj_\eta(\sigma)=\sigma$.
  Therefore using Proposition~\ref{prop:disintegration_opp}, the right hand side of \eqref{eq:weak-*-convergence} can be written
  \begin{align*}
    &\int_W f \circ \pr_\xi \circ \pr_\zeta(\sigma) d\mu_{\Opp(\zeta)}(\sigma)\\
    = &\int_U \int_V f(\pr_\xi \pr_\zeta \pr_\eta \sigma) d\mu_{\Res(\zeta)}(\sigma) d\mu_{\Oppv[i](\zeta)}(\eta) \\
    = &\int_U \int_V f(\pr_\xi \sigma) d\mu_{\Res(\zeta)}(\sigma) d\mu_{\Oppv[i](\zeta)}(\eta) \\
    = &\mu_{\Oppv[i](\zeta)}(U) \int_V f(\pr_\xi \sigma) d\mu_{\Res(\zeta)}(\sigma)  \\
    =& \mu_{\Oppv[i](\zeta)}(U) f_V(\xi)
  \end{align*}
  where $f_V$ is as in Theorem~\ref{thm:martingale_convergence_fV}.

  On the other hand \eqref{eq:weak_star_sequence_characteristic} is
  \begin{align*}
    &\frac{\mu(W)}{\mu(\gamma_nW)} \int_{\gamma_n W} f(\sigma) d\mu_{\Opp(\zeta)}(\sigma)\\
    =&\mu_{\Oppv[i](\zeta)}(U) \fint_{\gamma_n U} \int_V f(\pr_\eta \sigma) d\mu_{\Res(\zeta)}(\sigma) d\mu_{\Oppv[i](\zeta)}(\eta)\\
    =&\mu_{\Oppv[i](\zeta)}(U) \fint_{\gamma_n U} f_V(\eta) d\mu_{\Oppv[i](\zeta)}(\eta).
  \end{align*}

  Dividing by $\mu_{\Oppv[i](\zeta)}(U)$ the desired convergence is now the convergence in Theorem~\ref{thm:martingale_convergence_fV}.
\end{proof}

\section{Factor theorem}\label{sec:factor_theorem}

As a first step toward the Factor Theorem (Theorem~\ref{thm:factor}) we prove the following weaker topological version (extending \cite[Corollary~4.27]{BFL}):

\begin{proposition}\label{prop:topfactor}
    Any $\Gamma$-invariant closed (for the norm topology) subalgebra of $C(\Delta_1)$ is either only the constants or $C(\Delta_1)$. The analogous statement holds for $\Delta_2$.
\end{proposition}

\begin{proof}
  If $X$ is of type $\tilde A_2$, this is \cite[Corollary 4.27]{BFL}. So we consider the case where $X$ is of type $\tilde C_2$ or $\tilde G_2$, so that $\bar{i} = i$.
    
  By Gelfand--Naimark duality \cite[Theorem~1.1]{Khalkhali09}, the statement of the proposition is equivalent to saying that for every compact $\Gamma$-space $Y$, every equivariant continuous surjective map $\pi'\colon\Delta_1\to Y$ is either bijective or $Y$ is reduced to a point. So we need to show that if $\pi'\colon\Delta_1\to Y$ is non-injective then it is constant. Composing with the projection $\Delta \to \Delta_1$ we obtain a map $\pi \colon \Delta \to Y$ that is continuous, equivariant and constant on residues of type $1$. We assume that it admits chambers $c,c'$ from different $1$-residues with $\pi(c) = \pi(c')$ (corresponding to non-injectivity of $\pi'$) and need to show that it is constant on $2$-residues. Since $\Delta$ is gallery-connected it then follows that $\pi$ is constant.

  For $(\xi,\zeta) \in \Deltaopv[2]$ put $T_\zeta^\xi=\proj_\xi\circ\proj_\zeta$.  Note that the set of pairs $(p \circ T_\zeta^\xi)(c,c')$ with $(\xi,\zeta) \in \Deltaopv[2]$ and $p \in P_\xi$ (acting diagonally) includes all pairs $(d,d') \in \Delta \times \Delta$ with $d \ne d'$ lying in a common $2$-residue $\Res \xi$: take $\zeta$ opposite $\xi$ so that $\pr_\zeta c \ne \pr_\zeta c'$ (see below) and choose $p$ appropriately using $2$-transitivity of $Q_\xi$ (Lemma~\ref{lem:2-transitive}). If $\pi(d) = \pi(d')$ for every such pair then $\pi$ would have to be constant. 
  Let $D \subseteq \Delta \times \Delta$ be the closure of the $\Gamma$-orbit of $(c,c')$. By continuity of $\pi$ we get that $\pi(d) = \pi(d')$ for $(d,d') \in D$. Our goal is to show that $(p \circ T_\zeta^\xi)(c,c') \in D$ often enough to still conclude that $\pi$ is constant.

  Theorem \ref{thm:contracting} implies that $T_\zeta^\xi (c,c') \in D$ for almost every $(\xi,\zeta)$. Applying further Lemma~\ref{lem:FullMeasureDenseProj} and using the 2-transitivity of the projectivity group we get that for almost every $(\xi_0,\xi) \in \Delta_2^2$, if $(\Res(\xi_0) \times \Res(\xi_0)) \cap D$ contains $(d,d')$ with $d \ne d'$ then $(\Res(\xi) \times \Res(\xi)) \cap D$ is dense in $\Res(\xi) \times \Res(\xi)$. Using Fubini it is also true that for almost every $\xi_0$ if $(\Res(\xi_0) \times \Res(\xi_0)) \cap D$ contains $(d,d')$ with $d \ne d'$ then for almost every $\xi$ the space $(\Res(\xi) \times \Res(\xi)) \cap D$ is dense in $\Res(\xi) \times \Res(\xi)$. Let us call $\xi_0$ \emph{admissible} if it comes from a corresponding co-null set and call $\xi$ \emph{admissible for $\xi_0$} if it comes from the second co-null set.

  We claim that $E = \{\zeta \in \Delta_2 \mid \proj_\zeta(c) \ne \proj_\zeta(c')\}$ has positive measure in $\Delta_2$. Let $c \in \Delta$ be a chamber and let $\xi$ be its vertex of type $2$. Then $\Oppv[2](\xi) \subseteq \Delta_2$ is open and $\Oppv[2](\xi) \to \Delta, \zeta \mapsto \pr_\zeta c$ is continuous by Lemma~\ref{lem:opp_decomposition_topological}. It follows that $E$ is open and hence of positive measure since any $\mu_{\Delta_2}^o$ has full support (a witness that $E \ne \emptyset$ can be found inside an apartment containing $c$ and $c'$).

  As a consequence there is a pair $(\xi_0,\zeta)$ with $T_\zeta^{\xi_0}(c,c') \in D$, $\xi_0$ admissisble and $\zeta \in E$ and so applying $T_\zeta^{\xi_0}$ we may assume that $c$ and $c'$ lie in a common residue $\Res(\xi_0)$ of type $2$.

  Note that if $A = \{(c,c') \in \Res(\xi) \times \Res(\xi) \mid \pi(c) = \pi(c')\}$ is dense in $\Res(\xi) \times \Res(\xi)$ then $\pi|_{\Res(\xi)}$ is constant: the preimage under $\pi \times \pi$ of the complement of the diagonal in $\Res(\xi) \times \Res(\xi)$ is open and meets $A$ trivially, hence is empty. This is the case if $\xi$ is admissible for $\xi_0$.

  We have seen that for almost every $\xi \in \Delta_2$ the restriction $\pi|_{\Res(\xi)}$ is constant. Since the measure has full support, this is true on a dense subset of $\Delta_2$. By continuity of $\pi$ it follows that the same holds for every $\xi$. Hence $\pi$ is constant.
\end{proof}

Recall that our goal is to prove Theorem \ref{thm:factor}, which we recall here for the convenience of the reader, again in the formulation of Theorem~\ref{thm:factorintro}.

\begin{theorem}
    Let $A$ be a weak-* closed $\Gamma$-invariant subalgebra of $L^\infty(\Delta)$. Then $A$ is either $L^\infty(\Delta)$, $L^\infty(\Delta_1)$, $L^\infty(\Delta_2)$ or $\C$.
\end{theorem}

The following is \cite[Lemmas~5.1,~5.2]{BFL}.

\begin{lemma}\label{lem:treillis}
  The two algebras $L^\infty(\Delta_1)$ and $L^\infty(\Delta_2)$ generate a weak-*-dense subalgebra of $L^\infty(\Delta)$ and their intersection is $\C$.
\end{lemma}

\begin{proof}
  The characteristic functions $\chi_U$ for $U$ in a basis for the topology of $\Delta$ are dense in $L^\infty(\Delta)$. So for the first statement it suffices to show that they are contained in the algebra generated by $L^\infty(\Delta_1)$ and $L^\infty(\Delta_2)$. So let $o \in X$ be a vertex of type $0$, let $\sigma \in \Delta$. Let $c$ be the projection of $\sigma$ to $o$ and let $v,w$ be the other two vertices of $c$, chosen in such a way that a ray from $o$ through $v$ ends in a vertex of type $1$ and the ray through $w$ in a vertex of type $2$.

  The set $V$ of chambers $\sigma \in \Delta$ such that the projection to $o$ contains $v$ is an open subset of $\pi_1^{-1}(\Delta_1)$. Similarly the $W$ set of chambers such that the projection to $o$ contains $w$ is an open subset of $\pi_2^{-1}(\Delta_2)$. The intersection $V \cap W$ is the set $U = U^o(\sigma)$ of chambers whose porjection to $o$ is the chamber containing $v$ and $w$ (cf.\ \eqref{eq:Uo}). It follows that $\chi_U = \chi_V \cdot \chi_W$ lies in the algebra generated by $L^\infty(\Delta_1)$ and $L^\infty(\Delta_2)$.

  For the second statement recall from Lemma~\ref{lem:F/A} that the map $(\Fregular,\mu_\Fregular) \to \Delta \times \Delta,\iota \mapsto (\sigma_+(\iota),\sigma_-(\iota))$ is measure class-preserving when $\Delta \times \Delta$ carries either of $\mu_{\Deltaop}$ or $\mu_{\Delta}^o \times \mu_{\Delta}^o$. Let $f \in L^\infty(\Deltav[1]) \cap L^\infty(\Deltav[2])$. Then for any two chambers at infinity $\sigma,\sigma'$ of $\Sregular$, if $\sigma$ and $\sigma'$ share a vertex $u$ (say of type $i$), then the set of $A_u$ of all $\iota \in \Fregular$ such that $f(\iota(\sigma)) \ne f(\iota(\sigma'))$ has measure zero, as $f \in L^\infty(\Delta_i)$. The intersection of all $A_u$, for $u$ in the finitely many vertices in the boundary of $\Sregular$, is therefore of measure $0$, so that there is a co-null set $E \subseteq \Fregular$ such that $f(\iota(\sigma))$ is the same for all chambers at infinity $\sigma$ of $\Sregular$. In particular, $\{(\sigma,\sigma') \in \Delta \times \Delta \mid f(\sigma) \ne f(\sigma')\}$ has measure zero. Applying Fubini we see that for almost all $\sigma \in \Delta$ (in particular, for some $\sigma \in \Delta$) and almost all $\sigma' \in \Delta'$ we have $f(\sigma') = f(\sigma)$, i.e.\ $f$ is essentially constant, equal to $f(\sigma)$.
\end{proof}

For $u\in \Delta_i$ let $R_u=(\proj_u)^*L^\infty(\Res(u),\mu_u^o)$. In other words, $f\in R_u$ if and only if $f$ is essentially constant on almost all the preimages by $\proj_u$. By Lemma~\ref{lem:projection_measure-preserving}, $R_u$ is a subalgebra of $L^\infty(\Delta)$, in fact, $R_u\subset L^\infty(\Delta_{\overline i})$.

By 2-transitivity of $P_i$ on $\Res(u)$ we get easily the following (see \cite[Lemma 5.8]{BFL})

\begin{lemma}\label{lem:radial}
    Any $P_i$-invariant weak-* closed subalgebra of $R_u$ is either $R_u$ or reduced to the constants.
\end{lemma}

\begin{proof}
  It suffices to prove the statement with $R_u$ replaced by $L^\infty(\Res(u),\mu_u^o)$. Since any two vertices of $\Delta_i$ are connected by a projectivity we may assume $u = \xi_i$. Using transitivity of $P_i$ on $\Res(u)$ we may identify $\Res(u)$ with $P_i/M$ where $M$ is the stabilizer of some $\sigma \in \Res(u)$. Since the action of $P_i$ is $2$-transitive (Lemma~\ref{lem:2-transitive}), it is primitive and hence $M$ is a maximal subgroup. Now the only $P_i$-equivariant factors of $L^\infty(P_i/M)$ are $L^\infty(P_i/K)$ for $M < K < P_i$ and the claim follows from maximality of $M$.
\end{proof}

\begin{proof}[Proof of Theorem \ref{thm:factor}]
Let $A\subset L^\infty(\Delta)$ be a weak-* closed $\Gamma$-invariant subalgebra. 

We claim that if $A$ contains a function that is not in $L^\infty(\Delta_1)$ then it contains $L^\infty(\Delta_2)$. This (and its symmetric statement) together with Lemma \ref{lem:treillis} implies the result.

So let $f\in A$ which does not lie in $L^\infty(\Delta_1)$, i.e.\ there is a subset of positive measure of $\xi\in\Delta_1$ such that $f$ is not $\mu_{\Res(\xi)}^o$-essentially constant on $\Res(\xi)$.

Since $A$ is closed and $\Gamma$-invariant, by Theorem \ref{thm:convergence} it follows that $f \circ T_\zeta^{\xi'} \in A$ for every $f \in A$ and almost every $(\xi',\zeta) \in \Deltaopv[i]$ where $T_\zeta^{\xi'} = \pr_{\xi'} \circ \pr_\zeta$ as before. By composing such elements and applying Lemma~\ref{lem:FullMeasureDenseProj} we see that for almost every $\xi' \in \Deltav[1]$, every $f \in R_{\xi'} \cap A$ and a $p$ in a dense subset of the projectivity group $P_{\xi'}$ we have $f \circ p \in A$.

Let $\xi$ be as in the second paragraph, and consider $g=f\circ T_\zeta^\xi \in A$. By our choice of $\xi$ we get that $g$ is not constant, and by definition $g\in R_\xi\cap A$. Since $R_\xi \cap A$ is weak-$*$ closed and invariant under a dense subset of $P_\xi$, it is $P_\xi$-invariant and thus all of $R_\xi$ by Lemma~\ref{lem:radial} we deduce that $R_\xi\subset A$. 

Now let $h\in C(\Res(\xi))$ be a continuous non-constant function, and let $\tilde h=h\circ \proj_\xi \in R_\xi$. If we are not in type $\tilde A_2$, then $\tilde h$ is a continuous function on $\Delta_2$, and it is non-constant. By Proposition \ref{prop:topfactor}, we deduce that $C(\Delta_2)\subset A$. Since $C(\Delta_2)$ is weak-* dense in $L^\infty(\Delta_2)$ we conclude that $L^\infty(\Delta_2)\subset A$, which concludes the proof.
\end{proof}

\section{Non-linearity}

In this section we explain how to derive Theorem~\ref{thm:intrononlinearity}. Unlike to the rest of the article we do not give a self-contained treatment but explain how to adapt \cite[Section~8]{BCL}. 

\begin{proof}[Proof of Theorem~\ref{thm:intrononlinearity}]
Since $\Gamma$ has property (T), by \cite[Corollary A]{BCL}, it suffices to prove that any finite index subgroup of $\Gamma_0<\Gamma$ has no representation $\rho:\Gamma\to \bfG(K)$, where $K$ is a local field, $\bfG$ is a connected, adjoint $K$-simple $K$-group and the image of the representation is Zariski dense and unbounded. Since $\Gamma_0$ is still a lattice of $X$, we can assume $\Gamma=\Gamma_0$. By contraposition, we assume that such a representation exist, and aim to prove that $\overline P_i$ has a representation into some $\GL_n(K)$.

To do so, the main tool we will use is Bader--Furman's theory of algebraic representations \cite{BaderFurman13}. Let $Z$ be a $\Gamma$-ergodic standard measure space. An \emph{algebraic representation} of $Z$ is a $\rho$-equivariant map $\phi:Z\to \bfG/\bfH(K)$, where $\bfH<\bfG$ is an algebraic subgroup. Among the algebraic representations there is an intial one \cite[Theorem 4.3]{BaderFurmanSRProducts}) $\phi_0 \colon Z \to \bfG/\bfH_Z(K)$. This means that for every representation $\phi$ there is a unique $\bfG$-equivariant map $\psi \colon \bfG/\bfH_Z \to \bfG/\bfH$ defined over $K$ such that $\phi = \psi(K) \circ \phi_0$. We call $\bfH_Z$ (which is unique up to isomorphism of $\bfG$) the \emph{gate} of $Z$.

The universal property of the gate implies that if $M$ is a group acting on $Z$, and such that this action commutes with that of $\Gamma$, there exists a homomorphism from $M$ to the automorphism group of the gate, which is (in the category we consider) the quotient $N(\bfH_Z)/\bfH_Z(K)$. Furthermore, if $M$ is a Polish group and the action is Borel, then this map is continuous \cite[Proposition 8.5]{BCL}.

  Now we consider ergodic actions of $\Gamma$ on $\Delta$, on $\Deltaop$, and on $\Fsingular_i/A_i$ respectively (with their respective measure classes), and let $\bfH_\Delta$, $\bfH_{\Deltaop}$, $\bfH_{\Fsingular_i/A_i}$ be their respective gates. Composing an algebraic representation of $\Deltaop$ with the map $\Fsingular_i/A_i\to \Deltaop$ gives a representation of $\Fsingular_i/A_i$, so that we have inclusions $\bfH_\Delta > \bfH_{\Deltaop} > \bfH_{\Fsingular_i/A_i}$. 

  The proof follows closely \cite[p.~541]{BCL}, and we divide the proof into three steps as well (the fourth step is applying Schleiermacher's Theorem, which we cannot). Step (1) is claiming that $\bfH_{\Fsingular_i/A_i}$ is a proper subgroup of $\bfG$, or equivalently, that there exists an algebraic representation of $\Fsingular_i/A_i$ which is not reduced to a point. By the previous remark, it suffices to prove the same for $\bfH_\Delta$. Now, as discussed in \ref{thm:action_amenable} and Proposition~\ref{prop:amenable_action_map}, the action of $\Gamma$ on $\Delta$ is amenable, so using \cite[Theorem 6.1]{BDL} (together with the ergodicity of the action of $\Gamma$ on $\Delta\times\Delta\simeq \Deltaop$ from Theorem~\ref{thm:CartanErgodic}) we get the existence of a non-trivial representation of $\Delta$.
 
  From the above discussion, since the group $M_i$ acts on $\Fsingular_i/A_i$, commuting with the action of $\Gamma$, we get a homomorphism $M_i\to N_G(\bfH_{\Fsingular_i/A_i})/\bfH_{\Fsingular_i/A_i}$. More down-to-earth, the group $\Z/2\Z=\{1,\tau\}$ acts on $\Deltaop$ by $\tau(c,d)=(d,c)$, and similarly we get a homomorphism $\Z/2\Z\to N_G(\bfH_{\Deltaop})/\bfH_{\Deltaop}(K)$.

  Step (2) of the proof claims that this morphism is not trivial. If it were, the two maps $\Deltaop \to \bfG/\bfH_{\Delta}(K)$ given by composing the first and second projection with the gate representation of $\Delta$ would coincide. But since $\Deltaop\simeq \Delta\times \Delta$ this would mean that this gate representation of $\Delta$ is essentially constant, which, as we discussed, is not the case.

  Before moving to step (3), let us discuss the projectivity group $G_i=\overline P_i$. This group acts 2-transitively on the boundary of $T_i$, so we can apply \cite[Theorem 3.7]{BCL} (compiling some results from Burger and Mozes \cite{BurgerMozesLocalGlobal}) to conclude that it has a cocompact closed subgroup $G^+_i$ that is topologically simple. 

  In step (3) we show that the action of $G^+_i$ is non trivial. Indeed, since $G_i^+$ acts 2-transitively on $\partial_\infty T_i$, there exists a $m\in G_i^+$ such that the map $d:\Fsingular_i/A_i\to \Deltaop$ is such that $d(m\iota) = \tau d(\iota)$ for every $\iota $. The group $G_i^+$ acts on $\Fsingular_i/A_i$, commuting with the action of $\Gamma$, so we get a continuous homomorphism $G_i^+\to N_G(\bfH_{\Fsingular_i/A_i})/\bfH_{\Fsingular_i/A_i}$. Using all the previous equivariant maps, we see that the action of $m$ on that space cannot be trivial. Since $G_i^+$ is topologically simple (and the action is continuous), it means that the map $G_i^+\to N_G(\bfH_{\Fsingular_i/A_i})/\bfH_{\Fsingular_i/A_i}$ is injective. Hence the group $G_i^+$ is linear over a local field.
\end{proof}

\appendix

\section{Limit measures}\label{sec:limit_measures}

\subsection{Constant-to-one maps and measures}

Let $\Meas$ be the category whose objects are $\sigma$-finite measure spaces and whose morphisms are surjective measure-preserving maps. Since the morphisms are maps rather than maps up to measure $0$, this is not the category one might usually want to consider for measure theory but it suits our purposes. Note that this choice means in particular, that there is a forgetful functor to the category $\Mble$ of measurable spaces and measurable maps, and one further to $\Set$.

\begin{definition}
  A map $f \colon X \to Y$ is \emph{$n$-to-one} if $\abs{f^{-1}(y)} = n$ for all $y \in Y$. It is \emph{constant-to-one} if it is $n$-to-one for some $n \in \N_{>0}$. We call $n$ the \emph{fiber size} of an $n$-to-one map.
\end{definition}

Note that being constant-to-one is more restrictive than being finite-to-one. Further note that constant-to-one maps are surjective. The composite of an $m$-to-one and an $n$-to-one map is an $(mn)$-to-one map, so sets and constant-to-one maps form a category which we denote by $\CTO$.

We now consider constant-to-one maps in the context of measure spaces.

\begin{definition}
  Let $Y$ be a countable set and let $\kappa_Y$ denote the counting measure on $Y$ (with respect to the discrete $\sigma$-algebra). For $r > 0$ the measure $r \cdot \kappa_Y$ is a \emph{rescaled counting measure}.
Let $f \colon X \to Y$ be an $n$-to-one map and let $Y$ be equipped with a rescaled counting measure $\nu$. The \emph{uniform measure relative to $\nu$} is the measure $\mu$ on $X$ with respect to the discrete $\sigma$-algebra given by $\mu(\{x\}) = \frac{1}{n}\nu(\{f(x)\})$. Thus if $\nu = r \cdot \kappa_Y$ then $\mu = \frac{r}{n} \cdot \kappa_X$. In that case we also say that $f$ is \emph{uniformly measure-preserving}.
\end{definition}

The relatively uniform measure satisfies

\begin{equation}\label{eq:relatively_uniform_measure}
  \int_X g(x) d\mu(x) = \int_{y \in Y} \frac{1}{n} \bigg(\sum_{x \in g^{-1}(y)} g(x)\bigg) d\nu(y).
\end{equation}
and $f_*(\mu) = \nu$.

If $\lambda$ is the uniform measure relative to the counting measure on $X$ then \eqref{eq:relatively_uniform_measure} can be written as
\[
  \int_X g(x) d\mu(x) = \int_{y \in Y} \int_{x \in g^{-1}(y)} g(x) d\lambda(x) d\nu(y).
\]
This is the kind of expression that we are after, with the (relatively) uniform measures replaced by (relatively) pro-uniform ones, i.e.\ by limits.

Recall that a diagram $F \colon \calD \to \calC$ is just a functor from a small (index) category to $\calC$. For our purposes $\calD$ will always be a directed set so $(F(d))_{d \in \calD}$ is a directed ($F$ covariant) or inverse ($F$ contravariant) system. Also recall that $\calD$ is connected if any two objects are connected by a sequence of morphisms which can be traversed in forward or backward direction.

\begin{proposition}\label{prop:cto_limit_measure}
  Let $F \colon \calD \to \CTO$ be a diagram of constant-to-one maps and let $d_0 \in \calD$. Assume $\calD$ to be small and connected and $F(d_0)$ to be countable (possibly finite). Then there is a unique diagram $G \colon \calD \to \Meas$ such that $F$ and $G$ induce the same diagram in $\Set$ (composing with the respective forgetful functors), $G(d_0)$ carries the counting measure and every morphism $G(d \to e)$ is uniformly measure-preserving.
\end{proposition}

If $F(d_0)$ is countable then in fact all $F(d)$ are countable.

\begin{proof}
  As a measurable space we take $G(d)$ to be $F(d)$ with the discrete $\sigma$-algebra. We equip $G(d_0)$ with the counting measure. If $G(d)$ has been equipped with a measure and $d \to e$ is a morphism in $\calD$, we push the measure along $G(d \to e)$. If $G(e)$ has been equipped with a measure and $d \to e$ is a morphism we equip $G(d)$ with the relatively uniform measure with respect to $G(d \to e)$. Since $\calD$ is connected, this defines (at least) a measure on every $G(d)$. For well-definedness it remains to observe that pushforward of measures and taking relatively uniform measures are functorial and mutually inverse for countable spaces with rescaled counting measures and constant-to-one maps.
\end{proof}

\begin{internal}
\begin{example}
  That the measures be rescaled counting measures is important in order for taking the pushforward measure and the relatively uniform measure to be mutually inverse to each other: taking any measure $\nu$ on $\{0,1\}$ pushing it along $f \colon \{0,1\} \to \{0\}$ and taking the relatively uniform measure gives the rescaled counting measure $r \cdot \kappa_{\{0,1\}}$ where $r = \frac{1}{2} \nu(\{0,1\})$. So $\nu$ has to be rescaled counting in order to recover it.
\end{example}\sw{But if we just take relatively uniform measures on the $j \ge i$ they form a cofinal system?!}
\end{internal}

\begin{remark}
  \begin{enumerate}
    \item
      Note that replacing $d_0$ by $d$ where $F(d \to d_0)$ is $n$-to-one has the effect of scaling all measures by $n$. Replacing it by $d$ by $d_0$ where $F(d_0 \to d)$ is $n$-to-one has the effect of scaling all measures by $1/n$. It follows that all measures obtained by varying the choice of $d_0$ are rationally proportional.
    \item One could equip $G(d_0)$ with an arbitrarily chosen discrete measure at the expense of $G$ being defined only on the subdiagram of $d$ that map to $d_0$ in order to avoid pushforwards. Since we will eventually be interested in inverse systems this would not be an issue since the resulting system is cofinal. On the other hand we do not have use for the added generality.
  \end{enumerate}
\end{remark}

\subsection{Limit measures}

The following theorem should be compared to Kolmogorov's Consistency Theorem \cite[Theorem~10.6.2]{Cohn13} which asserts that the limit of an inverse system consisting of finite subproducts of an infinite product exists provided the spaces are totally finite and inner regular. More general limit theorems are proven in \cite{Choksi58}.

\begin{proposition}\label{prop:limit_measure}
  If $I$ is a directed set and $((X_i)_i,(f_{ij})_{ij})$ is an inverse system in $\Meas$ of countable discrete spaces with rescaled counting measures and uniformly measure-preserving maps then $\lim_i X_i$ exists.
\end{proposition}

\begin{proof}
  We denote the $\sigma$-algebra on $X_i$ by $\calS_i$ and the measure by $\mu_i$. Note that the category $\Mble$ of measurable spaces and measurable maps has all limits. The limit we produce has as underlying measurable space the limit in $\Mble$. So let $X = \lim_i X_i$ be the set theoretic limit with maps $f_i \colon X \to X_i$. Let $\calA \defeq \{f_i^{-1}(A) : i \in I, A \in \calS_i\}$ so that the limit $\sigma$-algebra $\calS$ is the one generated by $\calA$.

  Since $I$ is a directed system, $\calA$ is an algebra: if $A = f_i^{-1}(A_i)$ and $B = f_j^{-1}(B_j)$ are two arbitrary elements (with $A_i \in \calS_i$ and $B_j \in \calS_j$), there is a $k \ge i,j$ so that $A_k = f_{ik}^{-1}(A_i)$ and $B_k = f_{jk}^{-1}(B_j)$ satisfy $A = f_k^{-1}(A_k)$, $B = f_k^{-1}(B_k)$ with $A_k,B_k \in \calS_k$. It follows that $A \cup B = f_k^{-1}(A_k \cup B_k) \in \calA$ and $A \cap B = f_k^{1}(A_k \cap B_k) \in \calA$. If $A = f_i^{-1}(A_i)$ then $X \setminus A = f_i^{-1}(X_i \setminus A_i)$ since we took the $f_{ij}$ (and hence the $f_i$) to be surjective.

  For $A  = f_i^{-1}(A_i)$ define $\bar{\mu}(A) = \mu_i(A_i)$ which is well-defined by commutativity of the system. We claim that it is a premeasure. Before proving the claim we make the following reduction: if $A^j, j \in \N$ form a partition of $X$ consisting of sets in $\calA$ with $\bar{\mu}(A^j)$ finite (which exist by $\sigma$-finiteness of the $X_i$) then extending $\bar{\mu}$ to a measure $\mu$ on $X$ is tantamount to extending $\bar{\mu}|_{A^j}$ to a measure $\mu_j$ on $A^j$: clearly $\mu_j = \mu|_{A^j}$ while $\mu = \sum_{j \in \N} \mu_j$. Thus there is no loss in assuming that $\bar{\mu}(A)$ is finite which we do from now on. (This may be surprising in comparison to Kolmogorov's theorem where finiteness is a crucial assumption but if in the context of Kolmogorov's theorem one wanted that the measures on finite products are $\sigma$-finite one would still need that the measures on the individual factors are finite for all but finitely many.)

  To prove the claim let $(A_m)_{m \in \N}$ be a sequence of disjoint sets in $\calA$ such that the union $A = \bigcup_m A_m$ is again in $\calA$. As in the reasoning above, for finitely many of $A$ and the $A_m$ we can arrange to be preimages of measurable sets in some fixed $X_i$, so $\bar{\mu}$ is monotonous and finitely additive. Thus $\bar{\mu}(A) \ge \bar{\mu}(\bigcup_{m=0}^n A_m) = \sum_{m=0}^n \bar{\mu}(A_m)$ for every $n$. It follows that $\bar{\mu}(A) \ge \sum_m \bar{\mu}(A_m)$.

  In order for $\bar{\mu}$ to be a premeasure it remains to see that \begin{equation}\bar{\mu}(A) \le \sum_{m=0}^\infty \bar{\mu}(A_m)\label{eq:kolmogorov_inequality}\end{equation} when the right hand side converges. Let $F_n = A \setminus \bigcup_{m=0}^{n} A_m$ which is a descending sequence of elements of $\calA$. Then
  \[
    \bar{\mu}(F_n) = \bar{\mu}(A) - \sum_{m=0}^n \bar{\mu}(A_m) \stackrel{n \to \infty}{\longrightarrow} \bar{\mu}(A) - \sum_{m=0}^\infty \bar{\mu}(A_m)
  \]
  so we want to see that $\lim_n \bar{\mu}(F_n) = 0$ knowing that $\bigcap_n F_n = \emptyset$. In general this is the core of the problem. In our case we just note that an element of finite measure in $\calA$ is compact in the limit topology on $X$. Thus if $\bigcap_n F_n = \emptyset$ for $F_n$ of finite measure then already $F_n = \emptyset$ for some $n$ by compactness.

  By Carathéodory's extension theorem, \cite[Theorem1.14]{Folland99} there exists a measure $\mu$ on $\calS$ extending $\bar{\mu}$. Since the $X_i$ are $\sigma$-finite, so is $X$. So Carathéodory's theorem implies that $\mu$ is the unique measure extending $\bar{\mu}$. We claim that $X$ equipped with $\calS$ and $\mu$ is the desired limit.

  So let $(Y,\calT,\nu)$ be a measure space with measure-preserving maps $g_i \colon Y \to X_i$ commuting with the $f_{ij}$ and let $g \colon Y \to X$ be the set-theoretic limit. Then for $A = f_i^{-1}(A_i) \in \calA$ we have $g^{-1}(A) = (f_i \circ g)^{-1}(A_i) = g_i^{-1}(A_i)$ and $\nu(g^{-1}(A)) = \mu_i(A_i) = \mu(A)$. So the pushforward measure $g_*(\nu)$ coincides with $\mu$ on $\calA$. Thus the uniqueness statement of Carathéodory's theorem implies that they coincide on $\calS$ meaning that $g$ is measure-preserving.
\end{proof}

Besides Carathéodory's theorem, the technical core of the proof is the compactness argument to verify \eqref{eq:kolmogorov_inequality}. In order to extend the proposition to more general cofiltered limits one could work with inner regular Borel spaces or more generally with compact measure spaces in the sense of \cite[Definition~342a(c)]{Fremlin3}. The subtlety lies in choosing the right morphisms (roughly perfect maps).

\begin{internal}
  The following illustrates that for countable discrete spaces constant-to-one maps are appropriate as opposed to finite-to-one maps.

\begin{example}
  Let $(c_i)_{i \in \N}$ be a sequence of numbers in $(0,1)$ such that $\prod_{i=0}^n c_i$ converges to $\varepsilon > 0$ for $n \to \infty$, for instance $c_i = 2^{-1/i^2}$. Define a measure $\mu_n$ on $X_n \defeq \{0,1\}^n$ by
  \[
    \mu_n(a_0,\ldots,a_n) = \prod_{i=0}^n (a_ic_i + (1-a_i)(1-c_i))
  \]
  so the $i$th position is a biased coin flip with bias $c_i$. Then $(X_n,\mu_n)$ is an inverse system of probability spaces and measure-preserving finite-to-one maps. Let $X = \{0,1\}^\N$ be the limit with projections $\pi_n \colon X \to X_n$ and define $\bar{\mu}(\pi^{-1}(A)) = \mu_n(A)$ for $A \in X_n$. Take $A_n = \{1\}^n \times \{0\} \times \{0,1\}^\N \subseteq \N$ and let $A = \bigcup_n A_n = X \setminus \{1\}^\N$ \sw{But $A$ is not in the algebra!}. Put $F_n = A \setminus \bigcup_{i=0}^{n-1} A_i = (\{1\}^n \times \{0,1\}^\N) \setminus \{1\}^\N$.

  If the constant $1$-sequence is not to be an atom then $\mu(F_n) = \mu(\{1\}^n \times \{0,1\}^\N) = \prod_{i=0}^n c_n$ so $\lim_n \mu(F_n) = \varepsilon > 0$ although $\bigcap_n F_n = \emptyset$.
\end{example}
\end{internal}

\begin{remark}
  If $(X_i)_i$ is a an inverse system of countable discrete spaces and constant-to-one maps, then the measure defined in Proposition~\ref{prop:limit_measure} is a Borel measure with respect to the topological limit $\lim_i X_i$. That topological limit is completely metrizable and separable and it follows that the measure is a Radon measure \cite[Theorem~7.1.7]{Bogachev07}.
\end{remark}

\subsection{Disintegration}

We are approaching the main goal of the appendix, namely to prove a disintegration result for pro-uniform measures. In the first step the base space is is countable counting space while the fibers carry pro-uniform measures.

\begin{proposition}\label{prop:disintegration}
  Let $(X_i)_{i \in I}$ be an inverse system of countable sets and constant-to-one maps and let $o \in I$.
  We equip $X_o$ with the counting measure each $X_i$ with the uniform measure $\nu_i$ relative to $\nu$, regard $(X_i)_i$ as an inverse system in $\Meas$ and take its limit $X \defeq \lim_i X_i$ in $\Meas$ producing a measure $\mu$. Denote the projection maps by $f_{i} \colon X \to X_{i}$.
  For each $y \in X_o$ let $\lambda^y$ be the measure on $f_o^{-1}(y)$ that arises as the limit of the uniform measures on $f_{o,i}^{-1}(y)$ relative to the counting measure on $\{y\}$. Then
  \[
    \int_X g(x) d\mu(x) = \int_{y \in X_{o}} \int_{x \in f_{o}^{-1}(y)} g(x) d\lambda^y(x) d\nu(y)
  \]
  for every $g \in L_1(X,\mu)$.
\end{proposition}

\begin{proof}
  It suffices to verify the claim when $g$ is the characteristic function of a measurable set $A$ and more specifically when $A = f_i^{-1}(A_i)$ for some measurable $A_i$ in some $X_i$. Since $I$ is directed, way may assume $o \le i$ and so $A = f_{o}^{-1}(A_{o})$ where $A_{o} = f_{o,i}(A_i)$. Note that $\lambda^y$ is a probability measure on $f_o^{-1}(y)$ by definition. So
  \begin{multline*}
    \int_{X} g(x) d\mu(x) = \mu(A) = \nu_i(A_i) = \nu(A_0) = 
    \int_{y \in A_0} d\nu(y) =\\ \int_{y \in A_0} \int_{x \in f_{o}^{-1}(y)} d\lambda^y(x) d\nu(y) = \int_{y \in X_{o}} \int_{x \in f_{o}^{-1}(y)} g(x) d\lambda^y(x) d\nu(y) .\qedhere
  \end{multline*}
\end{proof}

We now prove the general disintegration result and need some preparation.
Note that if $I$ and $J$ are directed sets then $I \times J$ with the product order is a directed set as well: for $(i_1,j_1), (i_2,j_2) \in I \times J$ there exist $i_3 \in I$ and $j_3 \in J$ with $i_3 \ge i_1,i_2$ and $j_3 \ge j_1,j_2$ so $(i_3,j_3) \ge (i_1,j_1),(i_2,j_2)$.

In what follows we consider an inverse system $(X_i^j)_{i,j \in I \times J}$ and put $X_i = \lim_j X_i^j$, $X^j = \lim_i X_i^j$ and $X = \lim_{i,j} X_o^j$. We denote the connecting morphisms as $f_{i,h}^{j,k} \colon X_i^j \leftarrow X_h^k$ for $i \le h \in I$ and $j \le k \in J$. If one of $h$ or $k$ is replaced by the limit, it is omitted from notation. For instance we have limit morphisms $f_{i,h}^j \colon X_i^j \leftarrow X_h$. If $h$, respectively $k$ is omitted, $i$ respectively $j$ may be omitted as well to denote for instance $f^j \colon X^j \leftarrow X$.

\begin{theorem}\label{thm:disintegration}
  Let $I$ and $J$ be directed sets, $o \in I$, $p \in J$. Let $(X_i^j)_{i,j \in I \times J}$ be a inverse system of countable sets and constant-to-one maps and use the above notation.

  Let $\nu_o^p$ be the counting measure on $X_o^p$. We equip each $X_i^j$ with the uniform measure relative to $\nu_o^p$ and obtain limit measures $\nu_o$ on $X_o$ and $\mu$ on $X$.

  We assume that the inverse system $(X_i)_i$ consists of constant-to-one maps and consider for each $y \in X_o$ the measure $\lambda^y$ on $f_o^{-1}(y)$ obtained by equipping $\{y\}$ with the counting measure, each $f_{o,i}^{-1}(y)$ with the relatively uniform measure and taking the limit. Then
  \[
    \int_X g(x) d\mu = \int_{y \in X_o} \int_{f_o^{-1}(y)} g(x) d\lambda^y(x) d\nu_o(y).
  \]
\end{theorem}

\begin{proof}
  We will tacitly use that $X = \lim_{(i,j)} X_i^j = \lim_j \lim_i X_i^j = \lim_i \lim_j X_i^j$ which is abstract nonsense once the limits exist (which they do by the previous discussion). Let $\mu^j$ be the limit over $i$ of the uniform measures relative to $\nu_o^j$.

  For $y \in X_o$ and $j \in J$ let $y^j = f_{o,o}^j(y) \in X_o^j$. Equip $\{y^j\}$ with the counting measure, $(f_{o,i}^{j,j})^{-1}(y^j) \subseteq X_i^j$ with the relatively prouniform measure and $(f_o^{j,j})^{-1}(y^j) \subseteq X^j$ with the limit measure, call it $\lambda^{y,j}$. Since $(\{y^j\})_{j \in J}$ is a measure-preserving system, so is $((f_o^{j,j})^{-1}(y^j))_{j \in J}$. Its limit is $\lambda^y$.
   By Proposition \ref{prop:disintegration} we have
  \[
    \int_{X^j} g(x) d\mu^j(x) = \int_{y \in X_o^j} \int_{x \in (f_o^{j,j})^{-1}(y)} g(x) d\lambda^{y,j}(x) d\nu_o^j(y).
  \]

  As we said before, $\mu$ is the limit of the $\mu^j$ and $\lambda^y$ is the limit of the $\lambda^{y,j}$.

  Thus if $A^j \subseteq X^j$ is $\mu^j$-measurable and $A \subseteq X$ its preimage and $A_o^j \subseteq X_o^j$ its image then
  \begin{multline*}
    \mu(A) = \mu^j(A^j) \\= \int_{y \in A_o^j} \int_{x \in A_j \cap (f_o^{j,j})^{-1}(y)} 1 d\lambda^{y,j}(x) d\nu_o^j(y) = \int_{y \in A_o^j} \int_{x \in A \cap (f_o^j)^{-1}(y)} 1 d\lambda^y(x) d\nu_o^j(y)\\
  = \int_{y \in A_o} \int_{x \in A \cap (f_o)^{-1}(y)} 1 d\lambda^y(x) d\nu_o(y)
  \end{multline*}
  as needed.
\end{proof}

\begin{example}
 A typical application of Theorem \ref{thm:disintegration} is given by the following. Let $Z$ be the product of a $(p+1)$-regular and a $(q+1)$-regular tree. Let $I = J = \N$. Let $Y$ be a quadrant $[0,\infty) \times [0,\infty)$, and for $(i,j) \in I \times J$ let $Y_i^j$ be the rectangle $[0,i]\times [0,j]\subset Y$, $Y_i=\bigcup_{j\in\N} Y_i^j$ and $Y^j=\bigcup_{i\in \N}Y_i^j$.

 Then define $X_i^j$ (resp. $X$, $X_i$, $X^j$) to be the set of embeddings of $Y_i^j$ (resp. $Y$, $Y_i$, $Y^j$) into $Z$ (respecting axes). We equip each of these spaces with $\sigma$-algebra associated to the topology of pointwise convergence (which is the discrete topology on $X_i^j$.)
 The various restriction maps $X_i^j\to X_k^l$ for $k\leq i$ and $l\leq j$ are constant-to-one. Then in the category $\Meas$ we have $X_i=\lim_j X_i^j$, $X^j=\lim_i X_i^j$ and $X=\lim_{i,j} X_i^j$.

 The space $X_0^0$ is naturally identified with the space of vertices of $Z$. We start with the measure $\nu_0^0$ on $X_0^0$ which is just the discrete uniform (infinite) measure. Then $\mu$ is just the natural uniform measure on $X$. On the other hand, $\nu_o$ is the uniform measure on the set $X_o$ of embeddings of a half-line in $Z$. Then Theorem \ref{thm:disintegration} provides a description of the conditional measures of $\mu$ on fibers of the restriction map $X\to X_o$: these are the measures $\lambda^y$ and are obtained themselves as limit of relatively uniform measures.
\end{example}

\subsection{Comparable morphisms}

Let $\calC$ be a category. We define the equivalence relation on morphisms in $\calC$ of being \emph{comparable} to be the smallest equivalence relation such that morphisms $\iota \colon Y \to Z$ and $\iota' \colon Y' \to Z'$ in $\calC$ are comparable if
\begin{enumerate}
  \item there are isomorphisms $\kappa \colon Y \to Y'$ and $\lambda \colon Z \to Z'$ in $\calC$ with $\lambda \circ \iota = \mu \circ \iota'$.\label{item:conjugate}
  \item there are comparable morphisms $\alpha \colon Y \to U$ and $\alpha \colon Y' \to U'$ and comparable morphisms $\omega \colon U \to Z$ and $\omega' \colon U' \to Z'$ in $\calC$ such that $\iota = \omega \circ \alpha$ and $\iota' = \omega' \circ \alpha'$.\label{item:extend}
  \item there are comparable morphisms $\alpha \colon Y \to U$ and $\alpha' \colon Y' \to U'$, comparable morphisms $\omega \colon V \to Z$ and $\omega' \colon V' \to Z'$ and comparable morphisms $\mu \colon Y \to Z'$ and $\mu' \colon Y' \to Z'$ with $\mu = \omega \circ \iota \circ \alpha$ and $\mu' = \omega' \circ \iota' \circ \alpha'$.\label{item:factor}
\end{enumerate}
\[
\begin{tikzcd}
Z \arrow[r, "\omega"']                    & Z'                     \\
Y \arrow[u, "\iota"] \arrow[r, "\alpha"'] & Y' \arrow[u, "\iota'"]
\end{tikzcd}
\quad
\begin{tikzcd}
Z                                                    & Z'                                                       \\
U \arrow[u, "\omega"]                                & U' \arrow[u, "\omega'"]                                  \\
Y \arrow[u, "\alpha"] \arrow[uu, "\iota", bend left] & Y' \arrow[u, "\alpha'"] \arrow[uu, "\iota'", bend right]
\end{tikzcd}
\quad
\begin{tikzcd}
Z                                                   & Z'                                                     \\
V \arrow[u, "\omega"]                               & V' \arrow[u, "\omega'"]                                \\
U \arrow[u, "\iota"]                                & U' \arrow[u, "\iota'"]                                 \\
Y \arrow[u, "\alpha"] \arrow[uuu, "\mu", bend left] & Y' \arrow[u, "\alpha'"] \arrow[uuu, "\mu", bend right]
\end{tikzcd}\]

We have the following observation:

\begin{lemma}\label{lem:comparable}
  Let $F \colon \calC \to \CTO$ be a functor (covariant or contravariant). If $\iota$ and $\iota'$ are comparable in $\calC$ then $F(\iota)$ and $F(\iota')$ have the same fiber size.
\end{lemma}

The following is an abstraction of \cite[Lemma~1.3(3)]{BFL}.

\begin{example}\label{ex:unicomparable}
    Let $\calP$ be a category of non-empty connected simplicial complexes which contains a single vertex complex $O$, a single-edge complex $E$ and all simplicial embeddings $O \to Z$ and $E \to Z$. We claim that any two morphisms with isomorphic sources and isomorphic targets in $\calP$ are comparable. By \eqref{item:conjugate} we may as well assume that they have the same source and target. Since $O$ has morphisms to every other object, using \eqref{item:factor} it suffices to see that any two morphisms $\iota,\iota' \colon O \to Z$ are comparable. If the vertex images are adjacent in $Z$ there is a morphism $\omega \colon E \to Z$ through which both $\iota$ and $\iota'$ factor: $\iota = \omega \circ \alpha$, $\iota' = \omega \circ \alpha'$. Moreover $E$ has an automorphism $\lambda \in \calP$ so that $\alpha' = \lambda \circ \alpha$. It follows from \eqref{item:conjugate} that $\alpha$ and $\alpha'$ are comparable and from \eqref{item:factor} that $\iota$ and $\iota'$ are comparable. Since all complexes in $\calP$ are connected, any two morphisms $\iota,\iota' \colon O \to Z$ are connected via adjacent ones and therefore comparable.
\end{example}

\subsection{Embeddings}

In our applications we consider a category $\calC$ of cell complexes and injective cellular maps and the functor $F$ is of the form $\Hom(-,X)$ for some object $X$ of $\calC$. The limits arise from colimits
\[
  \lim_i \Hom(Y_i,X) = \Hom(\colim_i Y_i,X)
\]
in $\calC$. Fixing $o \in I$, equipping $\Hom(Y_o,X)$ with counting measure $\nu_o$ and every $\Hom(Y_i, X)$ with the relatively uniform measure $\nu_i$, gives rise to a measure on the space $\Hom(Y,X)$ of embeddings of $Y = \colim Y_i$.

\begin{proposition}\label{prop:limit_invariant}
  The limit measure on $\Hom(Y,X)$ is invariant under $\beta \in \Aut_\calC(X)$ and under $\alpha \in \Aut_\calC(Y)$ for which $\alpha^* \colon \Hom(\alpha(Y_o),X) \to \Hom(Y_o,X)$ is measure preserving.
\end{proposition}

\begin{proof}
  Let $\alpha \in \Aut_\calC(Y)$, $\beta \in \Aut_\calC(X)$, let $A \subseteq \Hom(Y,X)$. We need to show that $A$ and
  \[
    A' = \beta \circ A \circ \alpha = \{\beta \circ \gamma \circ \alpha : \gamma \in A\}
  \]
  have the same measure. Let $i \in I$ and write $\iota \colon Y_i \to Y$ for the inclusion. It suffices to show that $A_i = \iota^{-1}(A)$ and $A_i' = \iota^{-1}(A')$ have the same measure. The map $\gamma \circ \iota \mapsto \beta \circ \gamma \circ \alpha \circ \iota$ is a bijection $A_i \to A_i'$ so the measures of both sets are proportional since the measures are rescaled counting measures. The proportionality constant is the one scaling the uniform measure on $\alpha(Y_o)$ relative to the counting measure on $Y_o$ to the counting measure on $\alpha(Y_o)$. If $\alpha$ is as in the assumption, this constant is one.
\end{proof}

The following example illustrates that the measure is not invariant under all of $\Aut_\calC(Y)$ in general.

\begin{example}
  Let $Y$ be the biinfinite line considered as a directed graph with vertex set $\Z$ and and an edge from $n$ to $n+1$ for every $n \in \Z$.
  Let $q \in \N$, $q \ge 2$. Let $X$ be the directed tree that as one outgoing and $q$ incoming edges in every vertex.
  Let $\calC$ consist of all inclusions of subintervals of $Y$ into $Y$ or $X$.
  
  Equipping $\Hom(\{0\},X)$ with counting measure and $\Hom(Z,X)$ with the relatively uniform measure for every finite subinterval $Z$ of $Y$ defines a limit measure on $\Hom(Y,X) = \lim_Z \Hom(Z,X)$. However the shift $n \mapsto n+1$ of $Y$ does not preserve this measure. It does preserve the inverse system of constant-to-one maps $\Hom(Z,X)$ but the uniform measure on $\Hom(\{1\},X)$ relative to the counting measure on $\Hom(\{0\},X)$ is not the counting measure but the counting measure scaled by $q$.
\end{example}

\providecommand{\bysame}{\leavevmode\hbox to3em{\hrulefill}\thinspace}
\providecommand{\MR}{\relax\ifhmode\unskip\space\fi MR }
\providecommand{\MRhref}[2]{%
  \href{http://www.ams.org/mathscinet-getitem?mr=#1}{#2}
}
\providecommand{\href}[2]{#2}


\end{document}